\documentclass[11pt]{amsart}

\usepackage{comment}
\usepackage[rgb,dvipsnames]{xcolor}
\usepackage{stmaryrd}
\usepackage{amsfonts}
\usepackage{amssymb}
\usepackage{amsmath,amsthm}
\usepackage{latexsym}
\usepackage{amscd}
\usepackage{esint}
\usepackage{mathrsfs}
\usepackage{dsfont}
\usepackage{setspace}
\usepackage{enumitem}
\usepackage{resmes}
\usepackage[margin=0.9in]{geometry}
\usepackage{bm}
\usepackage[all, cmtip]{xy}
\usepackage{tikz-cd}

\usepackage{hyperref}
\usepackage{comment}
\usepackage{enumitem}

\usepackage{graphicx}
\usepackage{blindtext}
\usepackage{url}
\usepackage{bbm}
\usepackage{nicefrac,mathtools}
\usepackage{tabularx}
\DeclareGraphicsExtensions{.pdf,.jpeg,.png}
\usepackage{epstopdf}
\usepackage{cancel} 
\usepackage[normalem]{ulem} 
\usepackage{verbatim} 
\usepackage{scalerel}

\newtheorem{theorem}{Theorem}[section]
\newtheorem{corollary}[theorem]{Corollary}
\newtheorem{claim}[]{Claim}
\newtheorem{lemma}[theorem]{Lemma}

\newtheorem{proposition}[theorem]{Proposition}

\theoremstyle{definition}
\newtheorem{definition}[theorem]{Definition}
\newtheorem{conjecture}[theorem]{Conjecture}

\newtheorem*{acknowledgements}{Acknowledgements}

\theoremstyle{remark}
\newtheorem{remark}[theorem]{Remark}

\numberwithin{equation}{section}

\newcommand{\V}{\mathcal{V}}

\newcommand{\R}{\mathbb{R}}
\newcommand{\N}{\mathbb{N}}
\newcommand{\F}{\mathcal{F}}

\newcommand{\A}{\mathcal{A}}

\newcommand{\C}{\mathcal{C}}
\newcommand{\VC}{\mathcal{VC}}

\newcommand{\mZ}{\mathbb{Z}}
\newcommand{\Z}{\mathcal{Z}}

\newcommand{\M}{\mathbf{M}}

\newcommand{\bL}{\mathbf{L}}

\newcommand{\mF}{\mathbf{F}}

\newcommand{\spt}{\operatorname{spt}}
\newcommand{\dist}{\operatorname{dist}}
\newcommand{\Div}{\operatorname{div}}

\newcommand{\inj}{\operatorname{inj}}

\newcommand{\interior}{\operatorname{int}}
\newcommand{\Ric}{\operatorname{Ric}}
\newcommand{\Clos}{\operatorname{Clos}}

\newcommand{\Vol}{\operatorname{Vol}}
\newcommand{\An}{\operatorname{An}}
\newcommand{\am}{\operatorname{am}}
\newcommand{\dmn}{\operatorname{dmn}}

\newcommand{\rom}[1]{\expandafter\romannumeral #1}

\newcommand{\RP}{\mathbb{RP}}

\newcommand{\mS}{\mathbb{S}}
\newcommand{\mcA}{\mathcal{A}}
\newcommand{\Ah}{\mathcal{A}^h}
\newcommand{\Gpm}{{G_{\pm}}}
\newcommand{\CG}{\mathcal{C}^{G_{\pm}}}
\newcommand{\VCG}{\mathcal{VC}^{G_{\pm}}}
\newcommand{\sF}{\mathscr{F}}
\newcommand{\sE}{\mathscr{E}}

\newcommand{\sEG}{\mathscr{E}_{G_{\pm}}}

\newcommand{\sXG}{\mathscr{X}_{G_{\pm}}}

\newcommand{\mfX}{\mathfrak{X}}
\newcommand{\Diff}{\operatorname{Diff}}
\newcommand{\Is}{\mathfrak{Is}}
\newcommand{\xEG}{\mathscr{X}_{G_{\pm}}}

\newcommand{\msX}{\mathscr{X}}
\newcommand{\msY}{\mathscr{Y}}
\newcommand{\bmsX}{\overline{\mathscr{X}}}
\newcommand{\mcX}{\mathcal{X}}
\newcommand{\mcY}{\mathcal{Y}}
\newcommand{\bmcX}{\overline{\mathcal{X}}}

\newcommand{\bd}{\partial}
\newcommand{\mcS}{\mathcal{S}}
\newcommand{\mcR}{\mathcal{R}}
\newcommand{\mcH}{\mathcal{H}}

\DeclareMathOperator{\im}{im}
\DeclareMathOperator{\Index}{index}

\title{Minimal surfaces with low genus in lens spaces}

\author[Xingzhe Li]{Xingzhe Li}
\address{Cornell University, Department of Mathematics, Ithaca, New York 14850}
\email{xy346@cornell.edu}

\author[Tongrui Wang]{Tongrui Wang}
\address{Institute for Theoretical Sciences, Westlake Institute for Advanced Study, Westlake University, Hangzhou, Zhejiang, 310024, China}
\email{wangtongrui@westlake.edu.cn, wangtongrui@sjtu.edu.cn}

\author[Xuan Yao]{Xuan Yao}
\address{Cornell University, Department of Mathematics, Ithaca, New York 14850}
\email{xl833@cornell.edu}

\begin{document}

\begin{abstract}
    Given a Riemannian $\mathbb{RP}^3$ with a bumpy metric or a metric of positive Ricci curvature, we show that there either exist four distinct minimal real projective planes, or exist one minimal real projective plane together with two distinct minimal $2$-spheres. 
    Our proof is based on a variant multiplicity one theorem for the Simon-Smith min-max theory under certain equivariant settings. 
    In particular, we show under the positive Ricci assumption that $\RP^3$ contains at least four distinct minimal real projective planes and four distinct minimal tori.  
    Additionally, the number of minimal tori can be improved to five for a generic positive Ricci metric on $\RP^3$ by the degree method. 
    Moreover, using the same strategy, we show that in the lens space $L(4m,2m\pm 1)$, $m\geq 1$, with a bumpy metric or a metric of positive Ricci curvature, there either exist $N(m)$ numbers of distinct minimal Klein bottles, or exist one minimal Klein bottle and three distinct minimal $2$-spheres, where $N(1)=4$, $N(m)=2$ for $m\geq 2$, and the first case happens under the positive Ricci assumption. 
\end{abstract}


\maketitle

\section{Introduction}

In 1917, Birkhoff \cite{birkhoff1917dynamical} proposed a min-max method and showed the existence of a closed geodesic in any Riemannian $2$-sphere $S^2$. 
Using multi-parameter families of closed curves to sweep out $S^2$, Lusternik-Schnirelmann \cite{lusternik1947topological} (see also Grayson \cite{grayson1989shortening}) further showed that there are at least $3$ closed geodesics in any Riemannian $2$-sphere. 


In one higher dimension, investigating minimal surfaces in $S^3$ is also a significant topic in differential geometry. 
In \cite{smith1983existence}, Simon-Smith set out a min-max process to construct minimal surfaces with controls on topology, which is a variant of the min-max theory established by Almgren \cite{almgren1962homotopy}\cite{almgren1965theory} and Pitts \cite{pitts2014existence} for minimal hypersurfaces (see also Schoen-Simon \cite{schoen1981regularity}). 
In particular, they 
proved the existence of an embedded minimal $2$-sphere in any Riemannian $S^3$ (cf. \cite{colding2003min}). 
In analogy with the result of Lusternik-Schnirelmann, S. T. Yau posed the following problem in his famous 1982 Problem Section \cite{yau1982seminar}. 
\begin{conjecture}[S. T. Yau \cite{yau1982seminar}]\label{Conj: four spheres}
	There are four distinct embedded minimal spheres in any manifold diffeomorphic to $S^3$. 
\end{conjecture}

Towards this conjecture, White \cite{white1991space} first used degree methods to show the existence of at least two embedded minimal spheres in any Riemannian $S^3$ with positive Ricci curvature. 
In particular, if the metric is sufficiently close to the round metric, White \cite{white1991space} showed that there are at least four embedded minimal spheres. 
Additionally, it was also shown in \cite{white1991space} that the set of bumpy metrics is generic in the Baire sense, where a metric $g_{_M}$ on a given closed manifold $M$ is called {\em bumpy} if every closed embedded minimal hypersurface is non-degenerate. 
Using Simon-Smith min-max theory and mean curvature flow, Haslhofer-Ketover \cite{haslhofer2019minimalS2} applied the catenoid estimates (cf. \cite{ketover2020catenoid}) to overcome the difficulty of multiplicities and proved the existence of at least two embedded minimal spheres in $S^3$ with bumpy metrics. 
Recently, significant progress was made by Wang-Zhou \cite{wangzc2023existenceFour} which resolved Yau's conjecture in $S^3$ with bumpy metrics or positive Ricci curvature via Simon-Smith min-max.

Inspired by the work of Wang-Zhou \cite{wangzc2023existenceFour}, we set out this paper to investigate minimal surfaces with low genus in lens spaces. 

\subsection{Minimal projective planes in $\RP^3$}

The motivation of Yau's conjecture is closely connected to the topological structure of the space of embedded spheres in $S^3$ (cf. Hatcher's proof of Smale conjecture \cite[Appendix (14)]{hatcher1983smale}). 
In light of the generalized Smale conjecture for lens spaces (cf. \cite{bamler2019ricciContrac, bamler2023ricciDiffeo} and \cite{ketover2023smale}), one would expect that the space of embedded projective planes in $\RP^3$ can retract onto the space $\sE$ of great projective planes (the quotient of great spheres). 
Note $\sE$ can be parameterized by $\RP^3$. 
Therefore, except for the area minimizing projective plane (associated with $H_0(\sE,\mZ_2)$), we can build three non-trivial (multi-parameter) families of great projective planes sweeping out the ambient manifold $\RP^3$, each of which is associated with the generator of $H_i(\sE, \mZ_2)\cong\mZ_2$, $i=1,2,3$, respectively (cf. Section \ref{Sec: sweepouts in RP^3}). 
Analogous to Conjecture \ref{Conj: four spheres}, it is now reasonable to ask whether there exist at least four distinct embedded minimal projective planes in $(\RP^3,g_{_{\RP^3}})$? 

In this paper, we use equivariant min-max theory and a variant multiplicity one result to investigate the above question. 
Our first main result is the following dichotomy theorem.

\begin{theorem}\label{Thm: minimal RP2 (main)}
	Given a Riemannian $\RP^3$ with a bumpy metric or a metric of positive Ricci curvature, either
	\begin{itemize}
		\item[(i)] there exist at least four distinct embedded minimal projective planes, or
		\item[(ii)] there exist one embedded minimal projective plane and two embedded minimal spheres.
	\end{itemize}
 In particular, (i) is valid when $\RP^3$ has positive Ricci curvature.
\end{theorem}

In \cite{haslhofer2019minimalS2}, Haslhofer-Ketover 
proved in $(\RP^3,g_{_{\RP^3}})$ with positive Ricci curvature that there are at least two embedded minimal projective planes. 
Our result can now improve the number of minimal projective planes in $(\RP^3,g_{_{\RP^3}})$ to four under the positive Ricci assumption. 
Indeed, a minimal sphere $\Sigma$ in $\RP^3$ will be lifted to two disjoint minimal spheres $\Sigma_+\sqcup\Sigma_-$ in $S^3$ so that the antipodal map on $S^3$ permutes $\Sigma_\pm$. 
Combining with the {\em embedded Frankel property} in manifolds with positive Ricci curvature (i.e. any two closed embedded minimal hypersurfaces must intersect), we conclude Theorem \ref{Thm: minimal RP2 (main)}(ii) can not happen if $(\RP^3,g_{_{\RP^3}})$ has positive Ricci.

In recent years, we have witnessed the tremendous progress of Almgren-Pitts min-max theory. 
One peak of the development is the resolution of Yau's conjecture \cite[Problem 88]{yau1982seminar} on the existence of infinitely many closed minimal surfaces due to Marques-Neves \cite{marques2017existence} and Song \cite{song2018existence}. 
Meanwhile, several significant advancements \cite{irie2018density, marques2019equidistribution, song2020generic} in the spatial distribution of closed minimal hypersurfaces were made by the multi-parameter min-max theory coupled with the Weyl law for the volume spectrum \cite{liokumovich2018weyl}. 
Moreover, after the establishment of a min-max theory for prescribed mean curvature hypersurfaces \cite{zhou2018existence, zhou2019constant}, Zhou \cite{zhou2020multiplicity} solved the Multiplicity One Conjecture in the Almgren-Pitts setting. 
Combining with the Morse index estimates, Marques-Neves \cite{marques2016morse, marques2021morse} complete their celebrated program on the Morse theory for the area functional (see also \cite{marques2020morse}). 
We refer to \cite{Zhou-ICM22} for a more detailed history.

\subsection{Strategy of the proof}
Although our proof shares the same spirit as Wang-Zhou's resolution \cite{wangzc2023existenceFour}, one major difference is that the sweepouts are formed by $1$-sided projective planes. 
However, Wang-Zhou's multiplicity one theorem \cite[Theorem B]{wangzc2023existenceFour} relies heavily on the min-max theory \cite[Theorem 2.4]{wangzc2023existenceFour} for prescribing mean curvature functionals $\Ah$ \eqref{Eq: Ah-functional}, while it is inappropriate to define the $\Ah$ functional or prescribed mean curvature functions for $1$-sided surfaces. 
Hence, \cite[Theorem B]{wangzc2023existenceFour} cannot be applied directly.

Let $\pi: S^3\to\RP^3$ be the double cover and $g_-$ be the involution map in $S^3$ (i.e. $\pi(p)=\pi(g_-(p))$ and $g_-(p)\neq p$ for all $p\in S^3$). 
Heuristically, we notice that a projective plane $\RP^2$ in $\RP^3$ can be lifted to a sphere $S=\pi^{-1}(\RP^2)$ which separates $S^3$ into two parts $\Omega_+,\Omega_-$ so that $\Omega_\pm = g_-(\Omega_\mp)$. 
Hence, if $\nu$ is a unit normal along $S$ (w.r.t. to the lifted metric), then $dg_-(\nu) = -\nu$ and the mean curvature $H$ of $S$ w.r.t. $\nu$ is anti-symmetric $H(p)=-H(g_-(p))$. 
This inspired us to lift the sweepouts in $\RP^3$ back into $S^3$ and consider the $\mZ_2=\{id, g_-\}$ equivariant min-max theory in $S^3$ for the prescribed mean curvature functionals $\Ah$ with anti-symmetric $h$ (i.e. $h(p)=-h(g_-(p))$). 

As one key novelty of this paper, we show the following multiplicity one theorem in a certain equivariant setting analogy to the above, which provides an alternative way to handle the multiplicities in the $1$-sided Simon-Smith min-max theory. 

\begin{theorem}\label{Thm: multiplicity one (main)}
	Let $(M, g_{_M})$ be a closed connected orientable $3$-dimensional Riemannian manifold, and let $G$ be a finite group acting freely and effectively as isometries on $M$ so that $G$ admits an index $2$ subgroup $G_+$ with coset $G_-=G\setminus G_+$. 
	Suppose $\Sigma_0$ is an orientable $G$-invariant genus-$\mathfrak{g}_0$ surface, and $\xEG$ consists all the $G$-equivariant embedding $\Sigma=\phi(\Sigma_0)$ of $\Sigma_0$ into $M$ so that $M\setminus\Sigma=\Omega_+\sqcup\Omega_-$ and $G_+\cdot\Omega_\pm=\Omega_\pm$, $G_-\cdot\Omega_{\pm}=\Omega_{\mp}$. 
	Then given any homotopy class of smooth sweepouts formed by surfaces in $\xEG$, the associated min-max varifold can be chosen as a $G$-varifold induced by a closed embedded $G$-invariant minimal surface $\Gamma$ with connected components $\{\Gamma_{j}\}_{j=1}^J$ and integer multiplicities $\{m_j\}_{j=1}^J$ so that 
	\begin{itemize}
		\item[(i)] if $\Gamma_j$ is unstable and $2$-sided, then $m_j=1$;
		\item[(ii)] if $\Gamma_j$ is $1$-sided, then its connected double cover is stable. 
	\end{itemize}
	Moreover, the weighted total genus of $\Gamma$ given by \eqref{eq:genus bound0} is bounded by $\mathfrak{g}_0$.
\end{theorem}

\begin{remark}\label{Rem: idea}
    \begin{enumerate}[itemjoin=\\]
        \item By re-merging some components, we can assume each $\Gamma_j$ is $G$-invariant and $\Gamma_j/G$ is connected. 
        In applications, one can first take a sweepout of $M/G$ formed by 1-sided surfaces in $\{\Sigma/G:\Sigma\in\xEG\}$, and then pull it back to a sweepout of $M$ in $\xEG$. 
        Clearly, the min-max $G$-surfaces $\{\Gamma_j\}_{j=1}^J$ given by Theorem \ref{Thm: multiplicity one (main)} can be reduced to min-max surfaces $\{\Gamma_j/G\}_{j=1}^J$ in $M/G$ with the same multiplicity. 
        One should note that the multiplicity $m_j$ is related to the stability and the orientability of $\Gamma_j$ by Theorem \ref{Thm: multiplicity one (main)}, not to $\Gamma_j/G$. 
        \item For a general compact Lie group $G$ acting by isometries on $M^{n+1}$ with $3\leq {\rm codim}(G\cdot x)\leq 7$ for all $x\in M$, the $G$-equivariant min-max theory for closed minimal $G$-hypersurfaces has been established by Ketover \cite{ketover2016equivariant}, Liu \cite{liu2021existence}, and the second author \cite{wang2022min}\cite{wang2023min} in different settings. 
        However, our scenarios cannot be covered directly by these works since we also need an equivariant min-max for the $\A^h$-functional \eqref{Eq: Ah-functional} and both $h$ and $\Omega_\pm$ are $G_-$-antisymmetric. 
    \end{enumerate}
\end{remark}

After applying the above theorem to $M=S^3$, $G=\{id, g_-\}$, $G_+=\{id\}$, $G_-=\{g_-\}$, and $\xEG=\{\pi^{-1}(P): \mbox{$P$ is an embedded $\RP^2\subset\RP^3$}\}$, we obtain some $G$-invariant unions of disjoint minimal spheres $\{\Gamma_j\}_{j=1}^J$ with multiplicities $\{m_j\}_{j=1}^N$ satisfying (i)(ii), which alternatively gives a multiplicity result for the min-max outputs $\{\widetilde{\Gamma}_i= \pi(\Gamma_j)\}_{j=1}^J$ in $\RP^3$. 
 
Meanwhile, another difficulty arises as $\{\widetilde{\Gamma}_j\}_{j=1}^J$ may also contain minimal spheres instead of minimal projective planes. 
Since spheres have no contribution to the total genus, the genus bound does not help to distinguish them. 
Nevertheless, we can show that (Proposition \ref{Prop: exist 1-sided component}) every min-max minimal surface $\cup_{j=1}^J \widetilde{\Gamma}_j$ associated with the sweepouts formed by ($1$-sided) projective planes contains exactly one minimal projective plane, i.e. $\widetilde{\Gamma}_1\cong\RP^2$ and $\widetilde{\Gamma}_{j\geq 2}\cong S^2$. 
Hence, if $\RP^3$ has positive Ricci curvature, we see $\cup_{j=1}^J \widetilde{\Gamma}_j=\widetilde{\Gamma}_1\cong\RP^2$ (by the Frankel property) and $\pi^{-1}(\widetilde{\Gamma}_1)$ is unstable in $S^3$, which implies $m_j=1$ (Theorem \ref{Thm: multiplicity one (main)}) and gives four distinct minimal projective planes associated with $H_i(\sE,\mZ_2)$, $i=0,1,2,3$. 

However, in general, we can neither eliminate minimal spheres in $\{\widetilde{\Gamma}_j\}_{j=1}^J$ nor rule out the existence of minimal projective planes $P\subset \RP^3$ with stable $\pi^{-1}(P)\subset S^3$.  
Moreover, if we cut $\RP^3$ along a minimal projective plane $P$ with stable $\pi^{-1}(P)\subset S^3$, we would obtain one $3$-ball $B$ with boundary $\bd B=\pi^{-1}(P)$. 
After adding a cylindrical end $\pi^{-1}(P)\times [0,\infty)$ to $B$ and applying Song's technique \cite{song2018existence} as in \cite[Section 8.2, 8.4]{wangzc2023existenceFour}, one can only find minimal spheres (instead of minimal projective planes) confined in $\interior(B)$. 
Therefore, given the above phenomena, it is reasonable to propose the dichotomy theorem (Theorem \ref{Thm: minimal RP2 (main)}) for the coexistence of both minimal projective planes and minimal spheres. 

Indeed, we would conjecture that under a certain metric on $\RP^3$, every minimal real projective plane constructed via min-max could be the same area minimizing one, but with possibly different multiplicities. 
A possible example is the Riemannian $\RP^3$ so that its double cover $S^3$ with the pull-back metric looks like a symmetric dumbbell with a long thin neck $S^2\times [-R,R]$ satisfying that the center $S^2\times \{0\}$ covers the area minimizing $\RP^2$, and for $t\neq 0$, $S^2\times \{t\}$ has positive mean curvature pointing towards $S^2\times \{0\}$.

\subsection{Minimal surfaces with Euler characteristic $0$ in lens spaces $L(p,q)$}

The above idea can also be applied to investigate minimal Klein bottles and minimal tori in lens spaces $L(p,q)\cong S^3/\mZ_p$, where $p\geq 2$, $1\leq q\leq p-1$, and $\mZ_p=\langle\xi_{p,q}\rangle$ acts freely on $S^3$ by \eqref{Eq: lens space action}. 

In particular, we have a very similar dichotomy theorem for minimal Klein bottles in $L(4m, 2m\pm 1)$ as shown below. 
Note $L(4m, 2m\pm 1)$, $m\geq 1,$ are the only lens spaces that contain embedded Klein bottles (cf. \cite[Corollary 6.4]{bredon1969nonorientable}). 

\begin{theorem}\label{Thm: minimal Klein bottles (main)}
	Let $m\geq 1$ be an integer, $N(1)=4$, and $N(m)=2$ for $m\geq 2$.
	Then, given a lens space $L(4m, 2m\pm 1)$ with a bumpy metric or a metric of positive Ricci curvature,  
	\begin{itemize}
		\item[(i)] either there exist $N(m)$ numbers of embedded minimal Klein bottles,
		\item[(ii)] or there exist one embedded minimal Klein bottle and three embedded minimal spheres.
	\end{itemize}
	In particular, (i) is valid when $L(4m, 2m\pm 1)$ has positive Ricci curvature. 
\end{theorem}

One notices that an embedded Klein bottle $K$ in $L(4m,2m\pm 1)$ can be lifted to an embedded $G$-invariant torus $T$ in $S^3$ so that $S^3\setminus T$ has two $G_+$-invariant components $U_+\sqcup U_-$ with $G_-$ permuting them, where $G=\mZ_{4m}=\langle\xi_{4m,2m\pm 1}\rangle$, $G_+=\mZ_{2m}=\langle\xi_{4m,2m\pm 1}^2\rangle$, and $G_-=G\setminus G_+$. 
Hence, using the idea in Remark \ref{Rem: idea}, we can lift the sweepout in $L(4m,2m\pm 1)$ formed by Klein bottles to a family of $\mZ_{4m}$-invariant tori in $S^3$, and apply Theorem \ref{Thm: multiplicity one (main)} to obtain closed embedded $\mZ_{4m}$-invariant minimal surfaces $\cup_{j=1}^J\Gamma_j$ in $S^3$ with the total weighted genus bounded by $1$. 
Similarly, $\cup_{j=1}^J\widetilde{\Gamma}_j= \cup_{j=1}^J\Gamma_j/\mZ_{4m}$ contains exactly one embedded minimal Klein bottle $\widetilde{\Gamma}_1$ (Corollary \ref{Cor: exist min-max and minimizing K2}). 
Then, $\widetilde{\Gamma}_1$ must have multiplicity one due to the weighted genus bound. 
The difficulty mainly comes from the possible existence of minimal spheres, which can be handled similarly by considering the coexistence of them.

Moreover, we also establish a new result regarding minimal tori in lens spaces. 

\begin{theorem}\label{Thm: minimal tori (main)}
	Let $L(p,q)$, $p\geq 2$, be a lens space with a metric of positive Ricci curvature. 
	Then \begin{itemize}
		\item[(i)] $L(2,1)\cong \RP^3$ contains at least four embedded minimal tori; in particular, for almost every (in the sense of Baire  category) Riemannian metric of positive Ricci curvature on $\RP^3$, there are at least five embedded minimal tori;  
        \item[(ii)] $L(p,1)$ and $L(p,p-1)$ with $p>2$ contains at least three embedded minimal tori;
		\item[(iii)] $L(p,q)$ with $q\notin\{1,p-1\}$ contains at least one embedded minimal torus. 
	\end{itemize} 
\end{theorem}

\begin{remark}
     Using the weighted genus bound \cite[(2.6)]{ketover2022flipping} and the catenoid estimates \cite{ketover2020catenoid}, Ketover concluded under the positive Ricci assumption that there are three minimal tori in $L(p,1), p\neq 2$ (\cite[Theorem 4.9]{ketover2022flipping}), and two minimal Klein bottles in $L(4p,2p\pm 1), p >1$ (\cite[Theorem 4.8]{ketover2022flipping}).  
    Our Theorem \ref{Thm: minimal Klein bottles (main)} and \ref{Thm: minimal tori (main)} extend these results by some new constructions and the multiplicity one theorem. 
\end{remark}

As an equivalent formulation of the Smale conjecture in lens spaces, Ketover-Liokumovich showed (\cite[Theorem 1.4]{ketover2023smale}) that the space of Heegaard tori in $L(p,q)$ retracts onto the space $\sE_{p,q}$ of Clifford tori. 
Hence, the parameterization of $\sE_{p,q}$ (cf. \cite[Proposition 2.3]{ketover2023smale}) and the Lusternik-Schnirelmann theory support the counts of minimal tori in Theorem \ref{Thm: minimal tori (main)}. 

It is also noteworthy that the space of Clifford tori in $S^3$ and $\RP^3$ can both be parameterized by $\RP^2\times\RP^2$. 
One would then expect five distinct minimal tori in a Riemannian $S^3$ and $\RP^3$. 
However, after White's existence result of one minimal torus in any positive Ricci curved $S^3$ (\cite[Theorem 6.2]{white1991space}) via the degree method, there has been no substantial progress (to the best of the author's knowledge) in the general construction of more minimal tori in a Riemannian $S^3$. 
Nevertheless, our Theorem \ref{Thm: minimal tori (main)}(i) has fulfilled the expectation in $\RP^3$ to a certain extent. 

The main challenges in finding minimal tori lie in two aspects. 
Firstly, the moduli space $\msX$ of embedded tori in $S^3$ or $\RP^3$ is very complicated, making it difficult to find an element $\alpha\in H^1(\msX\cup\bd\msX,\bd\msX;\mZ_2)$ with $\alpha^5\neq 0\in H^5(\msX\cup\bd\msX,\bd\msX;\mZ_2)$, while such an $\alpha$ is crucial for applying the Lusternik-Schnirelmann theory. 
In Section \ref{Sec: sweepouts by T^2}, we address this by considering part of the boundary $\msY:=\{{\rm great~circles}\}\subset \bd\msX$ and identifying an element $\alpha\in H^1(\msX\cup\msY,\msY;\mZ_2)$ with $\alpha^4\neq 0$ (Lemma \ref{Lem: lambda4 not 0}). 
This approach allows us to find four minimal tori in $\RP^3$ with positive Ricci curvature and potentially five via White’s degree method. 
Indeed, this construction of $\alpha$ is also valid in $S^3$. 
Additionally, another challenge is to distinguish tori from $2$-spheres in min-max outputs. 
In lens space, the positive Ricci assumption is sufficient to exclude all the minimal $2$-spheres by the embedded Frankel property, which simplifies the problem considerably. 
In contrast, this advantage is absent in $S^3$, where distinguishing between these surfaces remains challenging.

\begin{acknowledgements}
	The authors would like to thank Professor Gang Tian for his interest in this work, and to Professor Xin Zhou for the guidance and discussions. They also thank Professor Zhichao Wang, Professor Jintian Zhu, Alex Xu, and Nick Sun for helpful discussions. X.L. and X.Y. are grateful to T.W. and Westlake University for hosting them for a visit during which many of these ideas were conceived. Part of this work was done when X.L. and X.Y. visited Princeton University and we would also like to thank Princeton University for their hospitality. T.W. is supported by China Postdoctoral Science Foundation 2022M722844. X.L. and X.Y. are supported by NSF grant DMS-1945178. 
\end{acknowledgements}


\section{Preliminaries}\label{Sec: Preliminaries}

In this section, we collect necessary preliminaries and notations in the equivariant setting. 

Let $(M^3, g_{_M})$ be an oriented connected closed Riemannian $3$-dimensional manifold and $G$ be a finite group acting freely and effectively as isometries on $M$. 
By \cite{moore1980equivariant}, $M^3$ can be $G$-equivariantly isometrically embedded into some $\R^L$. 
In addition, let $G_+$ be an index $2$ subgroup of $G$ (i.e. $[G:G_+]=2$), and denote by $G_-$ the coset of $G_+$ in $G$ (i.e. $G=G_+\sqcup G_-$). 
Note we only specify the ambient manifold $M^3$ and the free actions of $G$ to $S^3$ and $\mZ_p$ respectively in Section \ref{Sec: Minimal RP2}, \ref{Sec: Minimal K2}, \ref{Sec: minimal T2}. 

To signify the $G$-invariance, we will add a superscript or a subscript `$G$' to the usual terminologies. 
We also use `$\Gpm$' to emphasize the $G_+$-symmetry together with the $G_-$-antisymmetry. 
\begin{itemize}
    \item $\pi$ : the projection $\pi:M\mapsto M/G$ given by $p \mapsto [p]:=\{g\cdot p: g\in G\}$;
    \item $\mcH^m$ : the $m$-dimensional Hausdorff measure in $\R^L$;
    \item $B_r(p)$ : Euclidean open ball of radius $r>0$ centered at $p$;
	\item $A_{s,r}(p)$ : Euclidean open annulus $B_r(p) \setminus \Clos(B_s(p))$;
    \item $B_r^G(p) := B_r(G\cdot p) = G\cdot B_r(p)$, and $A_{s,r}^G(p) := A_{s,r}(G\cdot p) = G\cdot A_{s,r}(p) $;
    \item $h\in C^\infty_{\Gpm}(M)$: a smooth mean curvature prescribing function, where
    \begin{align}\label{Eq: Ck Gpm function}
        C^k_{\Gpm}(M):=
        \left\{f\in C^k(M): 
        \begin{array}{ll}
            f(p)=f(g\cdot p), &\forall p\in M, g\in G_+, \\
            f(p)= -f(g\cdot p), &\forall p\in M, g\in G_-.
        \end{array}
        \right\},
    \end{align}
    for $k\in\N\cup\{\infty\}$;
    \item $\C(M)$, $\C^G(M)$: the space of sets $\Omega \subset M$ with finite perimeter (i.e. Caccioppoli sets, cf. \cite[\S 40]{simon1983lectures}), and the space of $\Omega\in\C(M)$ with $g\cdot \Omega =\Omega$ for all $g\in G$ (i.e. $G$-invariant Caccioppoli sets);
    \item $\C^{\Gpm}(M)$: the space of $G_+$-invariant Caccioppoli sets $\Omega\in\C(M)$ so that $G_-$ permutes $\Omega$ and $M\setminus \Omega$ as two ($G_+$-invariant) Caccioppoli sets, i.e.
    \[\C^{\Gpm}(M) := 
    \left\{\Omega\in\C(M) : 
    \begin{array}{ll}
        g\cdot \Omega = \Omega, & \forall g\in G_+, \\
        g\cdot \Omega = M\setminus \Omega \mbox{ in }\C(M),& \forall g\in G_-.
    \end{array}
    \right\};\]
    \item $\V(M)$, $\V^G(M)$ : the space of $2$-varifolds in $M^3$, and the space of $V\in\V(M)$ with $g_\# V=V$ for all $g\in G$ (i.e. $G$-invariant $2$-varifolds);
    \item $\mfX(M)$, $\mfX^G(M)$ : the space of smooth vector fields in $M$, and the space of $X\in \mfX(M)$ with $dgX=X$ for all $g\in G$;
    \item $\Diff_0(M)$ : the connected component of the diffeomorphism group of $M$ containing the identity.
\end{itemize}

For any $\Omega\in\C(M)$, denote by $\bd\Omega$ the (reduced)-boundary of $[[\Omega]]$ as a $3$-dimensional mod $2$ flat chain, by $|\bd\Omega|$ the induced integral varifold, and by $\nu_{\bd \Omega}$ the outward pointing unit normal of $\bd\Omega$ (\cite[14.2]{simon1983lectures}). 
In addition, denote by $\|\bd\Omega\|$ and $\|V\|$ the induced Radon measure in $M$ associated with $\bd\Omega$ and $V\in\V(M)$. 

\begin{remark}\label{Rem: equivariant PMC vector}
    For any $\Omega\in\CG(M^3)$, the $2$-dimensional mod $2$ flat chain $\bd \Omega$ and the varifold $|\bd\Omega|$ are both $G$-invariant. 
    Besides, the outward unit normal $\nu_{\bd \Omega}$ is $G_+$-invariant so that $dg(\nu_{\bd \Omega})=-\nu_{\bd \Omega}$ for any $g\in G_-$. 
    Hence, $h\cdot \nu_{\bd \Omega}$ is $G$-invariant. 
\end{remark}

Let $U\subset M$ be an open set. 
Then we denote by $\Is(U)$ the set of isotopies of $M$ supported in $U$. 
Additionally, after replacing $M$ by $U$, one can similarly define the localized notions $\C(U), \V(U), \mfX(U)$, and $\C^G(U), \C^{\Gpm}(U),\V^G(U),\mfX^G(U)$ provided that $U$ is $G$-invariant. 

Note that a map $F:M\to M$ is said to be {\em $G$-equivariant} if $F(g\cdot p)=g\cdot F(p)$ for all $p\in M, g\in G$. 
We can now introduce the following notations:
\begin{itemize}
    \item $\Diff^G_0(M)$ : the connected component of the $G$-equivariant diffeomorphism group of $M$ containing the identity;
    \item $\Is^G(U)$ : the set of $G$-equivariant isotopies of $M$ supported in an open $G$-set $U$.
\end{itemize}
Furthermore, for any $G$-invariant subset $A\subset M$ with connected components $\{A_i\}_{i=1}^I$, we say $A$ is {\em $G$-connected} if for any $i,j\in\{1,\dots,I\}$, there is $g_{i,j}\in G$ with $g_{i,j}\cdot A_j=A_i$. 
Finally, given a ($G$-)connected open set $U\subset M$, we say a set of ($G$-)connected $C^1$-embedded ($G$-)surfaces $\{\Gamma^i\}_{i=1}^l \subset U$ with $\bd\Gamma^i\cap U =\emptyset$ is {\em ordered}, denoted by 
\[\Gamma^1\leq \cdots\leq \Gamma^l,\] 
if for each $i$, $\Gamma_i$ separates $U$ into two ($G$-)connected components $U^i_\pm$ (i.e. $U\setminus\Gamma_i = U^i_+\sqcup U^i_-$) so that $\Gamma^1,\dots,\Gamma^{i-1}\subset \Clos(U^i_-)$ and $\Gamma^{i+1},\dots,\Gamma^l\subset \Clos(U^i_+)$. 

Since the group actions in this paper are free and effective, we can make the following definition. 
\begin{definition}\label{Def: appropriately small}
    A $G$-invariant open set $U\subset M$ is said to be {\em appropriately small}, if each connected component of $U$ is diffeomorphic to $\pi(U)=U/G$. 
\end{definition}
Namely, an appropriately small open $G$-set $U$ shall have $\#G$ connected components $\{U_i\}_{i=1}^{\#G}$ with $G$ permuting them. 
For instance, any open $G$-set $U\subset B^G_r(p)$ is appropriately small provided $r>0$ is smaller than the injectivity radius of $G\cdot p$. 

\subsection{$\Ah$-functional and $\VCG$-space}

For any pairs $(V,\Omega)\in \V^G(M)\times\CG(M)\subset \V(M)\times\C(M)$ and $h \in C^\infty_{\Gpm}(M)$, define the prescribing mean curvature functional by 
\begin{align}\label{Eq: Ah-functional}
    \Ah(V, \Omega)=\|V\|(M)-\int_{\Omega} h ~d\mcH^{3}.
\end{align}
Note $F_\#(V,\Omega):=(F_\#V,F(\Omega))$ is still a pair in $\V^G(M)\times\CG(M)$ for any $G$-equivariant diffeomorphism $F$ of $M$. 
Hence, we have the following definitions generalizing \cite[Definition 1.1]{wangzc2023existenceFour}. 

\begin{definition}\label{Def: stationary pairs}
    A pair $(V,\Omega)\in \V^G(M)\times\CG(M)$ is said to be {\em $\Ah$-stationary} (resp. {\em $(G,\Ah)$-stationary}) in a (resp. $G$-invariant) open set $U\subset M$, if for any $X\in\mfX(U)$ (resp. $X\in\mfX^G(U)$) with generated diffeomorphisms $\{\phi^t\}$, 
    \begin{align}\label{Eq: 1st variation of pairs}
        \delta \Ah_{V, \Omega}(X) & :=\left.\frac{d}{d t}\right|_{t=0} \Ah\left(\phi_{\#}^{t}(V, \Omega)\right) \\ & =\int_{G_2(M)}  \operatorname{div}_{S} X(x) ~d V(x, S) - \int_{\partial \Omega}\left\langle X, \nu_{\partial \Omega}\right\rangle h ~d \mu_{\partial \Omega}=0. \nonumber
    \end{align}
    In addition, $(V,\Omega)$ is said to be {\em $\Ah$-stable} (resp. {\em $(G,\Ah)$-stable}) in $U$, if 
    \begin{align}\label{Eq: 2nd variation of pairs}
        \delta^{2} \Ah_{V, \Omega}(X, X):=\left.\frac{d^{2}}{d t^{2}}\right|_{t=0} \Ah\left(\phi_{\#}^{t}(V, \Omega)\right) \geq 0
    \end{align}
    for any $X\in\mfX(U)$ (resp. $X\in\mfX^G(U)$). 
\end{definition}

In the following lemma, we shall prove that $(G,\Ah)$-stationarity implies $\Ah$-stationarity. 
Note the proof also works for general (non-free non-effective) $G$-actions. 

\begin{lemma}\label{Lem: G-stationary and stationary for pairs}
    Given any $G$-invariant open set $U\subset M$, a pair $(V,\Omega)\in \V^G(M)\times\CG(M)$ is $(G,\Ah)$-stationary in $U$ if and only if $(V,\Omega)$ is $\Ah$-stationary in $U$. 
\end{lemma}
\begin{proof}
    It is sufficient to show that for any $X\in \mfX(U)$, there exists $X_G\in\mfX^G(U)$ so that $\delta\Ah_{V,\Omega}(X) = \delta\Ah_{V,\Omega}(X_G)$. 
    Indeed, define $X_G:= (\sum_{g\in G} dg^{-1}(X))/\#G$ to be the averaged $G$-invariant vector filed compactly supported in $U$. 
    Since $\bd\Omega$ and $h\nu_{\bd\Omega}$ are both $G$-invariant (Remark \ref{Rem: equivariant PMC vector}), we have $\int_{\partial \Omega} \langle X, h\nu_{\partial \Omega}\rangle ~d \mu_{\partial \Omega} = \int_{\partial \Omega} \langle X_G, h\nu_{\partial \Omega}\rangle ~d \mu_{\partial \Omega}$. 
    In addition, note the diffeomorphisms generated by $X$ and $dg^{-1}(X)$ are $\{\phi^t\}$ and $\{g^{-1}\circ \phi^t\circ g\}$ respectively. 
    Hence, the $G$-invariance of $V$ 
    implies $\|\phi^t_\# V\|(M) = \|(g^{-1}\circ \phi^t\circ g)_\# V\|(M) $ and $\int  \operatorname{div}_{S} X d V = \int  \operatorname{div}_{S} (dg^{-1}(X)) d V = \int \operatorname{div}_{S} X_G d V$. 
\end{proof}

\begin{remark}\label{Rem: G-stable and stable for pairs}
    In particular, if the open $G$-set $U$ is appropriately small (Definition \ref{Def: appropriately small}), then we also have the equivalence between the $G$-stability and stability in $U$. 
    Indeed, since the stability of $(V,\Omega)$ in $U$ is equivalent to the stability in every component $U_i$ of $U$, $i=1,\dots,\#G$, it follows from \eqref{Eq: Ah in appropriately small set} and the proof of Lemma \ref{Lem: G-minimizing and minimizing for boundary} that $\#G \cdot \delta^2\Ah_{V,\Omega}(X,X) = \delta^2\Ah_{V,\Omega}(X_G,X_G)$, where $X\in\mfX(U_i)$ and $X_G=\sum_{g\in G} dg^{-1}(X)\in\mfX^G(U)$. 
\end{remark}

Motivated by Almgren's VZ-space \cite{almgren1965theory}, we make the following definition for the weak topological completion of the diagonal pairs $\{(|\bd\Omega|,\Omega): \Omega\in\CG(M)\}\subset \V^G(M)\times\CG(M)$. 

\begin{definition}\label{Def: VC space}
    The {\em $\VC$-space $\VC(M)$} (resp. {\em $\VCG$-space $\VCG(M)$}) is the space of all pairs $(V,\Omega)$ in $\V(M)\times \C(M)$ (resp. $\V^G(M)\times\CG(M)$) so that $V=\lim_{k\to\infty}|\bd\Omega_k|$ and $\Omega=\lim_{k\to\infty}\Omega_k$ for some sequence $\{\Omega_k\}_{k\in\N}$ in $\C(M)$ (resp. $\CG(M)$). 

    For any $(V,\Omega), (V',\Omega')\in \VC(M)$, define the {\em $\sF$-distance} between them by
    \[\sF((V,\Omega), (V',\Omega')) := \mF(V,V') + \F(\Omega,\Omega'),\]
    where $\mF$ is the varifolds $\mF$-metric (\cite[2.1(19)]{pitts2014existence}) and $\F$ is the flat metric (\cite[\S 31]{simon1983lectures}). 
\end{definition}

Note $\V^G(M)$ and $\CG(M)$ are closed sub-spaces of $\V(M)$ and $\C(M)$ respectively. 
One verifies that $\VCG(M)$ is also a closed sub-space of $\VC(M)$. 
Hence, the $\sF$ distance can be induced to $\VCG(M)$. 

Moreover, combining Lemma \ref{Lem: G-stationary and stationary for pairs} with \cite[Lemma 1.4-1.7]{wangzc2023existenceFour}, we have the following results. 

\begin{lemma}\label{Lem: preliminary of VCG-space}
    Denote by $c:=\sup_{p\in M}|h(p)|$. 
    Then for any $C>0$, $(V,\Omega)\in\VCG(M)$, and $G$-invariant open set $U\subset M$, the following statements are valid:
    \begin{itemize}
        \item[(i)] $\spt(\bd\Omega)\subset \spt(\|V\|)$ and $\|\bd\Omega\| \leq \|V\|$;
        \item[(ii)] if $(V,\Omega)$ is $(G,\Ah)$-stationary in $U$, then $V$ has $c$-bounded first variation in $U$, i.e. $|\delta V(X)| := |\int_{G_2(M)} \Div_S X(x) dV(S,x)| \leq c \int_{M}|X| d \|V\|, \forall  X \in \mfX(U) $;
        \item[(iii)] $A^C:=\{(V,\Omega)\in\VCG(M): \|V\|(M)\leq C\}$ is a compact metric space with respect to the $\sF$-distance;
        \item[(iv)] $A^C_0:=\{(V,\Omega)\in A^C: (V,\Omega)\mbox{ is $(G,\Ah)$-stationary} \} $ is a compact subset of $A^C$ with respect to the $\sF$-distance.
    \end{itemize}
\end{lemma}
\begin{proof}
    Since $\VCG(M)$ is a closed sub-space of $\VC(M)$, (i) and (iii) follow directly from \cite[Lemma 1.4, 1.6]{wangzc2023existenceFour}. 
    Combining Lemma \ref{Lem: G-stationary and stationary for pairs} and \cite[Lemma 1.5, 1.7]{wangzc2023existenceFour}, we conclude (ii) and (iv). 
\end{proof}

\subsection{$C^{1,1}$ almost embedded $(G_{\pm},h)$-surfaces}

In this subsection, we consider the $\Ah$-functional for $C^{1,1}$ (almost embedded) $\Gpm$-boundaries. 

\begin{definition}\label{Def: almost embedding}
    Let $U\subset M$ be an open subset. 
    A $C^{1,1}$ immersed surface $\phi: \Sigma\to U$ with $\phi(\bd\Sigma)\cap U = \emptyset$ is said to be a {\em $C^{1,1}$ almost embedded surface} in $U$, if at any non-embedded point $p\in\phi(\Sigma)$, there is a neighborhood $W\subset U$ of $p$ so that 
    \begin{itemize}
        \item $\Sigma\cap\phi^{-1}(W)$ is a disjoint union of connected components $\sqcup_{i=1}^l\Gamma^i$;
        \item $\phi(\Gamma^i)\subset W$ is a $C^{1,1}$ embedding for each $i=1,\dots,l$;
        \item for each $i$, any other component $\phi(\Gamma^j)$ ($j\neq i$) lies on one side of $\phi(\Gamma^i)$ (i.e. $\phi(\Gamma^j)\leq\phi(\Gamma^i)$ or $\phi(\Gamma^i)\leq\phi(\Gamma^j)$).
    \end{itemize}
\end{definition}

For simplicity, we will denote $\phi(\Sigma)$ and $\phi(\Gamma^i)$ by $\Sigma$ and $\Gamma^i$ respectively in appropriate context. 
The subset of non-embedded points in $\Sigma$, denoted by $\mcS(\Sigma)$, is called the {\em touching set}, and $\mcR(\Sigma):=\Sigma\setminus\mcS(\Sigma)$ is the {\em regular set}. 

\begin{definition}[$C^{1,1}$ $G_\pm$-boundary]\label{Def: G-boundary}
    Let $U\subset M$ be a $G$-invariant open subset. Then a $G$-equivariant $C^{1,1}$ almost embedded surface $\phi:\Sigma\to U$ is said to be a {\em $C^{1,1}$ (almost embedded) $\Gpm$-boundary} in $U$, if $\Sigma$ is oriented and
    \begin{align}\label{Eq: G-boundary}
        \phi_\#([[\Sigma]]) = \bd\Omega
    \end{align}
    as $2$-currents with $\mZ_2$-coefficients in $U$ for some $\Omega\in\CG(U)$. 
\end{definition}

\begin{remark}\label{Rem: unit normal on G-boundary}
    Let $(\Sigma,\Omega)$ be a $C^{1,1}$ $\Gpm$-boundary in $U$. 
    Then both of $\mcR(\Sigma)$ and $\mcS(\Sigma)$ are $G$-invariant. 
    Additionally, it follows from \cite[Lemma 1.11]{wangzc2023existenceFour} that $\Sigma$ (as an immersed surface) admits a unit normal $\nu_\Sigma$ so that 
    \begin{itemize}
        \item $\nu_\Sigma=\nu_{\bd\Omega}$ along $\spt(\bd\Omega)$ provided $\Omega\neq \emptyset$ or $U$;
        \item if $\Sigma$ decomposes into ordered $G$-connected sheets $\Gamma^1\leq \cdots \leq \Gamma^l$ in a $G$-connected open set $W\subset U$, then $\nu_\Sigma$ must alternate orientations along $\{\Gamma^i\}_{i=1}^l$. 
    \end{itemize} 
\end{remark}

For any $C^{1,1}$ $\Gpm$-boundary $(\Sigma,\Omega)$ (in $M$), define its $\Ah$-functional by
\begin{align}\label{Eq: Ah-functinal for boundary}
    \Ah(\Sigma,\Omega) := \mcH^2(\Sigma) - \int_\Omega h ~d\mcH^3
\end{align}
It then follows from (\ref{Eq: 1st variation of pairs}) and \cite[Lemma 1.12]{wangzc2023existenceFour} that for any $X\in\mfX(M)$,
\begin{align}\label{Eq: 1st variation of boundary}
    \delta \Ah_{\Sigma, \Omega}(X) &  =\int_{\Sigma}  \Div_{\Sigma} X ~d\mcH^2 - \int_{\partial \Omega}\left\langle X, \nu_{\partial \Omega}\right\rangle h ~d \mu_{\partial \Omega}. \\
    & = \int_{\Sigma}  \Div_{\Sigma} X  - \left\langle X, \nu_{ \Sigma}\right\rangle h ~d\mcH^2, \nonumber
\end{align}
where $\nu_\Sigma$ is given in the above remark. 

\begin{definition}[$C^{1,1}$ $(\Gpm,h)$-boundary]\label{Def: stationary boundary}
    A $C^{1,1}$ $\Gpm$-boundary $(\Sigma,\Omega)$ in a $G$-invariant open set $U$ is said to be {\em $\Ah$-stationary} (resp. {\em $(G,\Ah)$-stationary}) in $U$, if for any $X\in \mfX(U)$ (resp. $X\in \mfX^G(U)$), $\delta\Ah_{\Sigma,\Omega}(X)=0$. 
    In particular, we say $(\Sigma,\Omega)$ is a $C^{1,1}$ {\em (almost embedded) $(\Gpm,h)$-boundary} in $U$ if it is $(G,\Ah)$-stationary in $U$. 
\end{definition}

Similar to Lemma \ref{Lem: G-stationary and stationary for pairs}, we have the following result. 

\begin{lemma}\label{Lem: G-stationary and stationary for boundary}
    Given any $G$-invariant open set $U\subset M$, a $C^{1,1}$ $\Gpm$-boundary $(\Sigma,\Omega)$ is $(G,\Ah)$-stationary in $U$ if and only if $(\Sigma,\Omega)$ is $\Ah$-stationary in $U$.
\end{lemma}
\begin{proof}
    For any $X\in\mfX(U)$, let $X_G:= (\sum_{g\in G} dg^{-1}(X))/\#G \in \mfX^G(U)$. 
    Since $\Sigma$ is $G$-invariant and $\Omega\in\CG(M)$, we can combine (\ref{Eq: 1st variation of boundary}) with the proof of Lemma \ref{Lem: G-stationary and stationary for pairs} to conclude $\delta\Ah_{\Sigma,\Omega}(X)=\delta\Ah_{\Sigma,\Omega}(X_G)$ and get the desired result.   
\end{proof}

By the above lemma, $(\Sigma,\Omega)$ is a $C^{1,1}$ $(\Gpm,h)$-boundary if and only if it is a $C^{1,1}$ $\Gpm$-boundary (cf. Definition \ref{Def: G-boundary}) and a $C^{1,1}$ $h$-boundary (in the sense of \cite[Definition 1.13]{wangzc2023existenceFour}).
In particular, combining the first variation formula (\ref{Eq: 1st variation of boundary}) with the standard elliptic regularity theory, we know the regular set $\mcR(\Sigma)$ of a $C^{1,1}$ $(\Gpm,h)$-boundary $(\Sigma,\Omega)$ is smoothly embedded of prescribed mean curvature 
\begin{align}
   \left. H \right|_{\mcR(\Sigma)} =\left.h\right|_{\mcR(\Sigma)}
\end{align}
with respect to the unit normal $\nu_\Sigma=\nu_{\bd\Omega}$.

\subsection{Strong $\A^h$-stationarity}

In \cite{wangzc2023existenceFour}, Wang-Zhou introduced the {\em strong $\A^h$-stationarity} which plays an important role in the regularity theory of PMC min-max and their multiplicity one theorem.

\begin{definition}[Strong $\Ah$-stationarity]\label{Def: strong Ah-stationary}
    A $C^{1,1}$ $(\Gpm,h)$-boundary $(\Sigma,\Omega)$ is said to be {\em strongly $\A^h$-stationary} in an open set $U$, if the following holds:

    For any $p\in\mcS(\Sigma)\cap U$, there exists a small neighborhood $W\subset U$ of $p$ and a decomposition $\cup_{i=1}^l\Gamma^i$ of $\Sigma\cap W$ by $l:=\Theta^2(\Sigma,p)\geq 2$ connected disks with a natural ordering $\Gamma^1\leq\cdots\leq\Gamma^l$. 
    Let $W^1,W^l$ be the bottom and top components of $W\setminus\Sigma$. 
    Then for $i\in\{1,l\}$ and all $X\in\mfX(W)$ pointing into $W^i$ along $\Gamma^i$, 
    \begin{align}\label{Eq: strong stationary}
        \begin{array}{rl}
            \delta \mathcal{A}_{\Gamma^{i}, W^{i}}^{h}(X) \geq 0, & \text { when } W^{i} \subset \Omega; \\
            \delta \mathcal{A}_{\Gamma^{i}, W \backslash W^{i}}^{h}(X) \geq 0, & \text { when } W^{i} \cap \Omega=\emptyset.
        \end{array}
    \end{align}
\end{definition}

\begin{remark}\label{Rem: strong stationary and strong G-stationary}
    One can also say $(\Sigma,\Omega)$ is {\em strongly $(G,\Ah)$-stationary} in $U$, if $\Sigma$ has a local decomposition by $G$-connected ordered sheets $\Gamma^1\leq\cdots\leq\Gamma^l$ in a $G$-neighborhood $W\subset U$ of $G\cdot p\subset \mcS(\Sigma)$ so that (\ref{Eq: strong stationary}) is valid for all $G$-invariant $X\in\mfX^G(W)$ pointing away from all other sheets along $\Gamma^i$ ($i\in\{1,l\}$). 
    Nevertheless, since the strong $\Ah$-stationarity is a local property, a $C^{1,1}$ $(\Gpm,h)$-boundary $(\Sigma,\Omega)$ is strongly $\Ah$-stationary in $U$ if and only if it is strongly $(G,\Ah)$-stationary in $U$. 
\end{remark}

In particular, the results in \cite[Section 1.3]{wangzc2023existenceFour} remains valid for strongly $\Ah$-stationary (or strongly $(\Gpm,\Ah)$-stationary) $C^{1,1}$ $(\Gpm,h)$-boundaries. The mean curvature $H$ of $\Sigma$ w.r.t. $\nu$ still satisfies 
\begin{equation}\label{eq: generalized mean curvature} 
    H(p) = \left\{\begin{array}{ll}
     h(p) & \text{ when } p\in \mathcal{R}(\Sigma)\cap U  \\
     0 & \text{ for $\mathcal{H}^2$-a.e. } p\in \mathcal{S}(\Sigma)\cap U 
\end{array}\right ..
\end{equation}
Moreover, by assuming that $\Omega$ does not contain the region above the top sheet $\Gamma^l$, we know $H^{l} = h|_{\Gamma^l}$ in a neighborhood $h > 0$ and $H^l \geq h|_{\Gamma^l}$ in a neighborhood $h < 0$. We have a corresponding statement for $\Gamma^1$ by flipping the order.

\subsection{$G$-stability and compactness.}

Since we only consider $G$-equivariant deformations of $\Gpm$-boundaries, we will extend Wang-Zhou's compactness theorem for stable $h$-boundaries to an equivariant version. 

\begin{definition}[$G$-stable $C^{1,1}$ $(\Gpm,h)$-boundary]\label{Def: G-stable for boundary}
    Let $U\subset M$ be a $G$-invariant open set, and $(\Sigma,\Omega)$ be a $C^{1,1}$ $(\Gpm,h)$-boundary in $U$. 
    Then $(\Sigma,\Omega)$ is said to be {\em stable} (resp. $G$-stable) in $U$, if for any flow $\{\phi^t\}$ generated by $X\in\mfX(U)$ (resp. $X\in\mfX^G(U)$), 
    \[\delta^{2} \Ah_{\Sigma, \Omega}(X):=\left.\frac{d^{2}}{d t^{2}}\right|_{t=0} \Ah\left(\phi_{\#}^{t}(\Sigma, \Omega)\right) \geq 0.\]
    If in addition $(\Sigma,\Omega)$ is strongly $\Ah$-stationary, then $\delta^{2} \Ah_{\Sigma, \Omega}(X)\geq 0$ is equivalent to the following stability inequality for all $X\in\mfX(U)$ (resp. $X\in\mfX^G(U)$),
    \[\int_{\Sigma}\left|\nabla^{\perp} X^{\perp}\right|^{2}-\Ric_M\left(X^{\perp}, X^{\perp}\right)
    -\left|A_{\Sigma}\right|^{2}\left|X^{\perp}\right|^{2} \mathrm{~d} \mathcal{H}^{2} \geq \int_{\partial \Omega}\left\langle X^{\perp}, \nabla h\right\rangle\langle X, \nu_{\bd\Omega}\rangle \mathrm{d} \mathcal{H}^{2},\]
    where $\perp$ denotes the normal part with respect to $\Sigma$, and $A_\Sigma$ denotes the second fundamental form of $\Sigma$ (as an immersion). 
\end{definition}

Note if the touching set $\mcS(\Sigma)$ is empty (or consider $\Sigma$ instead of $\phi(\Sigma)$), the first eigenfunction $\varphi_1$ of the Jacobi operator $L_\Sigma^h:=\triangle_\Sigma - (\Ric_M(\nu_\Sigma,\nu_\Sigma) + |A_\Sigma|^2 + \bd_{\nu_\Sigma} h) $ is $G$-invariant since $\bd_{\nu_\Sigma}h$ is $G$-invariant and $G$-actions are isometries. 
Hence, the (intrinsic) $G$-stability is equivalent to the (intrinsic) stability provided $\mcS(\Sigma)=\emptyset$. 
However, $\mcS(\Sigma)\neq\emptyset$ in general even if $(\Sigma,\Omega)$ solves the isotopy minimizing problem \cite[Theorem 1.25]{wangzc2023existenceFour}. 
Nevertheless, as the group actions we considered here are free and effective, we still have the following lemma. 

\begin{lemma}\label{Lem: G-stable and stable for boundary}
    Let $(\Sigma,\Omega)$ be a $C^{1,1}$ $(\Gpm,h)$-boundary in an appropriately small open $G$-subset $U$. 
    Then $(\Sigma,\Omega)$ is stable in $U$ if and only if $(\Sigma,\Omega)$ is $G$-stable in $U$. 
\end{lemma}
\begin{proof}
    The `only if' part is direct. 
    For the `if' part, we only need to show $(\Sigma,\Omega)$ is stable in each connected component $\{U_i\}_{i=1}^{k}$, $k:=\#G$, of $U$. 
    Indeed, given $X\in\mfX(U_1)$ (with $X\llcorner U_j = 0$ for $j\neq 1$), we can define $X_G\in\mfX^G(U)$ by $X_G\llcorner g\cdot U_1 :=dg(X)$ since $U$ is appropriately small and the $G$-action is free and effective. 
    Then we notice $\nabla h$ and $\nu_{\bd\Omega}$ will keep or change the sign simultaneously under the push forward of $g\in G_+$ or $ g\in G_-$. 
    Therefore, $\delta^{2} \Ah_{\Sigma, \Omega}(X) = \delta^{2} \Ah_{\Sigma, \Omega}(dg(X))$, and thus $0\leq \delta^{2} \Ah_{\Sigma, \Omega}(X_G) = k\cdot \delta^{2} \Ah_{\Sigma, \Omega}(X)$ by the appropriate smallness of $U$. 
\end{proof}

Combining Lemma \ref{Lem: G-stable and stable for boundary} with \cite[Proposition 1.24]{wangzc2023existenceFour}, we conclude the following compactness result. 

\begin{proposition}\label{Prop: compactness for G-stable boundary}
    Suppose $U\subset M$ is an appropriately small open $G$-subset, and $h_j,h\in C^2_\Gpm(M)$ (cf. (\ref{Eq: Ck Gpm function})) satisfies $\lim_{j\to\infty}\|h_j-h\|_{C^2}= 0$. 
    Let $\{(\Sigma_j,\Omega_j)\}_{j\in\N}$ be a sequence of $G$-stable $C^{1,1}$ $(\Gpm,h_j)$-boundaries in $U$ so that $\mcH^2(\Sigma_j)\leq \Lambda$ for some $\Lambda>0$. 
    Then there is a stable $C^{1,1}$ $(\Gpm,h)$-boundary $(\Sigma,\Omega)$ in $U$ so that 
    \begin{itemize}
        \item[(1)] $\Sigma_j$ converges to $\Sigma$ in $U$ as varifolds and also in the sense of $C^{1,\alpha}_{loc}$ for all $\alpha\in (0,1)$;
        \item[(2)] $\Omega_j$ converge to $\Omega$ as currents in $\CG(U)\subset\C(U)$. 
    \end{itemize}
    Moreover, 
    \begin{itemize}
        \item[(i)] if $\Sigma$ is smooth, then $\Sigma_j$ converges to $\Sigma$ in $U$ in the $C^{1,1}_{loc}$ topology;
        \item[(ii)] if $(\Sigma_j,\Omega_j)$ is strongly $\mcA^{h_j}$-stationary in $U$, then $(\Sigma,\Omega)$ is strongly $\mcA^{h}$-stationary in $U$.
    \end{itemize}
\end{proposition}
\begin{proof}
    By Lemma \ref{Lem: G-stable and stable for boundary}, we can apply the compactness result \cite[Proposition 1.24]{wangzc2023existenceFour} and the $C^{1,1}$-regularity result \cite[Theorem 1.1]{wangzc2023improved} to $(\Sigma_j,\Omega_j)$ and obtain $C^{1,1}$ $h$-boundary $(\Sigma,\Omega)$ satisfying (1)(2)(i)(ii). 
    Since $G$ acts by isometries, we get the $G$-invariance of $\Sigma$ and $\Omega\in\CG(U)$. 
\end{proof}

\subsection{$G$-isotopy minimizing problem}\label{Subsec: Preliminary-isotopy minimizing}

Let $0<\mathbf{r}_0< \min\{\rho_0 , \inf_{p\in M}\inj(G\cdot p)\}$ be a sufficiently small constant, where $\rho_0=\rho_0(M,g_{_M},\sup|h|)>0$ is given in \cite[\S 14]{sarnataro2023optimal} and $\inj(G\cdot p)$ is the injectivity radius of the normal exponential map $\exp^\perp_{G\cdot p}$. 
Note $B^G_{\mathbf{r}_0}(p)$ is appropriately small for all $p\in M$ and $\inf_{p\in M}\inj(G\cdot p)>0$ since $G$ acts freely by isometries. 

For any open $G$-subset $U\subset B^G_{\mathbf{r}_0}(p)$, suppose $\mcR\in\CG(U)$ has a smoothly embedded $G$-invariant boundary $\Sigma:=\bd\mcR\cap U$ in $U$. 
Then we say a pair $(V,\Omega)\in \VC(U)$ (resp. $(V,\Omega)\in \VCG(U)$) is an {\em isotopy minimizer} (resp. {\em $G$-isotopy minimizer}) of $(\Sigma,\mcR)$ in $U$, if there exists a sequence $\{\phi_k\}_{k\in\N}$ in $\Is(U)$ (resp. $\Is^G(U)$) so that 
\begin{itemize}
    \item $\lim_{k\to\infty}\Ah\left(\phi_k(\Sigma,\mcR)\right) = \inf\{\Ah\left(\phi(\Sigma,\mcR)\right) : \phi\in\Is(U) ~({\rm resp.} ~\phi\in\Is^G(U))\}$;
    \item $\lim_{k\to\infty} \sF((V,\Omega), \phi_k(\Sigma,\mcR)) = 0$,
\end{itemize}
where we used the notation $\phi(\Sigma,\mcR) := (\phi(1,\Sigma), \phi(1,\mcR))$ for simplicity. 

Similar to Lemma \ref{Lem: G-stable and stable for boundary}, we also have the following equivalence between isotopy minimizers and $G$-isotopy minimizers in appropriately small open $G$-sets by the definitions of $h\in C^\infty_\Gpm(U)$ and $\CG(U)$. 

\begin{lemma}\label{Lem: G-minimizing and minimizing for boundary}
    Let $U\subset M$ be an appropriately small open $G$-subset and $(\Sigma,\mcR)$ be given as above. 
    Then a $G$-isotopy minimizer $(V,\Omega)\in \VCG(U)$ of $(\Sigma,\mcR)$ in $U$ is also an isotopy minimizer of $(\Sigma,\mcR)$ in $U$. 
\end{lemma}
\begin{proof}
    Denote by $\{U_i\}_{i=1}^{\# G}$ the connected components of $U$, and fix any component, e.g. $U_1$. 
    Then for any $(V',\Omega')\in\VCG(U)$, we have 
    \begin{align}\label{Eq: Ah in appropriately small set}
        \Ah(V'\llcorner U_i,\Omega'\llcorner U_i) = \left\{
        \begin{array}{ll}
             \Ah(V'\llcorner U_1,\Omega'\llcorner U_1)&  {\rm if~} U_i\subset G_+\cdot U_1; 
             \\
             \Ah(V'\llcorner U_1,\Omega'\llcorner U_1) + \int_{U_1} h d\mcH^3 &  {\rm if~} U_i\subset G_-\cdot U_1; 
        \end{array}
        \right. .
    \end{align}
    Therefore, for any $G$-isotopy minimizing sequence $\{\phi_k(\Sigma,\mcR)\}_{k\in\N}$, 
    \[ \Ah(\phi_k(\Sigma,\mcR))=\sum_{i=1}^{\# G}\Ah(\phi_k(\Sigma\llcorner U_i,\mcR\llcorner U_i)) = (\# G)\cdot \Ah(\phi_k(\Sigma\llcorner U_1,\mcR\llcorner U_1)) + \frac{\# G}{2}\int_{U_1} h d\mcH^3. \] 
    Since any isotopy $\phi\in\Is(U_1)$ can be extended to a $G$-isotopy $\phi_G\in\Is^G(U)$ by taking $\phi_G\llcorner U_i = g\circ \phi \circ g^{-1} $ for $U_i=g\cdot U_1$, it follows from the above equality that $\phi_k(\Sigma\llcorner U_1,\mcR\llcorner U_1)$ is also an isotopy minimizing sequence of $(\Sigma,\mcR)$ in $U_1$, which implies $(V\llcorner U_1, \Omega\llcorner U_1)$ is also an isotopy minimizer in $U_1$. 
\end{proof}

By the choice of ${\bf r}_0$, any open $G$-set $U\subset B^G_{{\bf r}_0}(p)$ is appropriately small. 
Hence, we have the following regularity theorem for $G$-isotopy minimizers in an appropriately and sufficiently small open $G$-set. 

\begin{theorem}\label{Thm: isotopy minimizer}
    Let $U\subset B^G_{{\bf r}_0}(p)$ and $(\Sigma,\mcR)$ be given as above. 
    Suppose $(V,\Omega)\in\CG(U)$ is a $G$-isotopy minimizer of $(\Sigma,\mcR)$ in $U$. 
    Then $(V,\Omega)$ is a strongly $\Ah$-stationary and stable $C^{1,1}$ $(\Gpm,h)$-boundary in $U$. 
\end{theorem}
\begin{proof}
    This result follows directly from Lemma \ref{Lem: G-minimizing and minimizing for boundary} and \cite[Theorem 1.25]{wangzc2023existenceFour}. 
\end{proof}

\section{Min-max theory for $C^{1,1}$ $(G_{\pm},h)$-boundaries}\label{Sec: PMC min-max}

In this section, we follow the approach in \cite{wangzc2023existenceFour} to set up the relative $G$-equivariant min-max problem for the $\mathcal{A}^{h}$-functional and prove the main regularity results for $(G, \mathcal{A}^{h})$-min-max pairs. 

\subsection{$G$-equivariant min-max problem}\label{Sec: Equivariant min-max} 
Fix a $G$-connected closed surface $\Sigma_{0}$ of genus $\mathfrak{g}_{0}$. 
A $G$-equivariant embedding $\phi: \Sigma_{0} \rightarrow M$ is said to be {\em $\Gpm$-separating} if $M \setminus \phi(\Sigma_{0}) = \Omega_{+} \sqcup \Omega_{-}$, where $\Omega_{+}, \Omega_{-}$ are two nonempty domains sharing a common boundary $\phi(\Sigma_{0})$ so that $G_+ \cdot \Omega_{\pm} = \Omega_{\pm}$ and $G_- \cdot\Omega_{\pm}=\Omega_{\mp}$. 
For convenience, write $\Sigma = \phi(\Sigma_0)$ with the orientation induced by the outer normal $\nu$ of $\Omega$, where $\Omega$ is an arbitrary choice of $\{\Omega_{+}, \Omega_{-}\}$.  
We then denote 
\[\sEG := 
        \left\{(\Sigma,\Omega):
        \Sigma \text{ is a $G$-equivariant $\Gpm$-separating embedding of $\Sigma_0$ in $M$}
        \right\}\]
endowed with oriented smooth topology in the usual sense. 

Let $X$ be a finite dimensional cubical complex, and $Z \subset X$ be a subcomplex. Let $\Phi_{0}: X \rightarrow \sEG$ be a continuous map. We denote by $\Pi$ the set of all continuous maps $\Phi: X \rightarrow \sEG$ which is homotopic to $\Phi_{0}$ relative to $\Phi_{0}|_{Z}: Z \rightarrow \sEG$. We refer to such a $\Phi$ an $(X, Z)$-sweepout, or simply a sweepout. The definitions 2.1-2.3 in \cite{wangzc2023existenceFour} are directly transferred to our situation as follows.  

\begin{definition}
    Given $(X, Z)$ and $\Phi_{0}$ as above, we call $\Pi$ the {\em $(X, Z)$-homotopy class of $\Phi_{0}$}. 
    The {\em$h$-width} of $\Pi$ is defined by: 
    \[
    \mathbf{L}^{h} = \mathbf{L}^{h}(\Pi) = \inf_{\Phi\in\Pi} \sup_{x \in X} \mathcal{A}^{h}(\Phi(x)).   
    \]
    A sequence $\{\Phi_{i}\}_{i \in \mathbb{N}} \subset \Pi$ is called a {\em minimizing sequence} if 
    \[
    \mathbf{L}^{h}(\Phi_{i}) := \sup_{x \in X} \mathcal{A}^{h}(\Phi_{i}(x)) \rightarrow \mathbf{L}^{h}, \text{ when } i \rightarrow \infty. 
    \]
    A subsequence $\{\Phi_{i_j}(x_j): x_j \in X\}_{j \in \mathbb{N}}$ is called a {\em min-max (sub)sequence} if 
    \[
    \mathcal{A}^{h}(\Phi_{i_j}(x_j)) \rightarrow \mathbf{L}^{h}, \text{ when } j \rightarrow \infty.
    \]
    The {\em critical set} of a minimizing sequence $\{\Phi_{i}\}$ is defined by 
    \[ \mathbf{C}(\{\Phi_i\})=\left\{(V, \Omega)\in \VCG(M)\left|\,
    \begin{aligned}   
        & \exists \text{ a min-max subsequence }\{\Phi_{i_j}(x_j)\} \text{ such}\\
        & \text{that } \sF \big(\Phi_{i_j}(x_j), (V, \Omega)\big) \to  0 \text{ as } j\to\infty
    \end{aligned}\right\}\right..
    \]
\end{definition}

We have the following min-max theorem, and the proof will be given later. 

\begin{theorem}[PMC Min-Max Theorem]\label{thm:pmc min-max theorem}
    With all notions as above, suppose
    \begin{equation}\label{eq:width nontrivial1}
        \mathbf{L}^h(\Pi)> \max\left\{\max_{x\in Z}\A^h\big(\Phi_0(x)\big), 0\right\}.
    \end{equation}
    Then there exist a minimizing sequence $\{\Phi_i\}\subset \Pi$, and a strongly $\Ah$-stationary, $C^{1,1}$ $(\Gpm, h)$-boundary $(\Sigma, \Omega)$ lying in the critical set $\mathbf{C}(\{\Phi_i\})$ such that 
    $\A^h(\Sigma, \Omega) = \mathbf{L}^h(\Pi)$. 
\end{theorem}

\subsection{Tightening}\label{sec: tightening}  
We follow the pull-tight process in Section 2.2 of \cite{wangzc2023existenceFour} with some alterations. 

\begin{theorem}[Pull-tight]\label{thm:pull-tight}
    Let $\Pi$ be an $(X, Z)$-homotopy class generated by some continuous $\Phi_0: X \to \sEG$ relative to $\Phi_0|_Z$. Given a minimizing sequence $\{\Phi^*_i\}_{i\in \mathbb{N}}\subset \Pi$ associated with $\A^h$, there exists another minimizing sequence $\{\Phi_i\}_{i \in \mathbb{N}} \subset \Pi$, such that $\mathbf{C}(\{\Phi_i\})\subset \mathbf{C}(\{\Phi_i^*\})$ and every element $(V, \Omega)\in \mathbf{C}(\{\Phi_i\})$ is either $(G, \A^h)$-stationary (and thus $\Ah$-stationary by Lemma \ref{Lem: G-stationary and stationary for pairs}), or belongs to $B = \Phi_0(Z)\subset \sEG$.   
\end{theorem}

\begin{proof}
    Given $C := \mathbf{L}^h + \sup_M |h(p)| \cdot \Vol(M) + 1$, let $A^C$ and $A^C_0$ be defined as in Lemma \ref{Lem: preliminary of VCG-space}. 
    Denote by $B := \Phi_{0}(Z)$ and $A_0=A^C_0\cup B$. 
    For any $\mathcal{X}\in \mathfrak{X}(M)$, we choose $\mathcal{X}^{G}= (\Sigma_{g \in G} d g^{-1}(\mathcal{X})) /\#G$, which satisfies $\delta A^h_{V, \Omega}(\mathcal{X})=\delta A^h_{V, \Omega}(\mathcal{X}^G)$ for every $(V,\Omega)\in \VCG(M)$ (Lemma \ref{Lem: G-stationary and stationary for pairs}). 
    Then, combining with the constructions in \cite[Section 2.2, Step 1, 2]{wangzc2023existenceFour}, we obtain a continuous map $\mathcal{X}^G: A^C\to\mfX^G(M)$ (under the $C^1$-topology on $\mfX^G(M)$) so that $\mathcal{X}^G\llcorner A_0 =0$ and $\mathcal{X}^G\llcorner (A^C\setminus A_0)$ is continuous under the smooth topology on $\mfX^G(M)$. 
    Additionally, by \cite[Section 2.2, Step 3]{wangzc2023existenceFour}, we also have two continuous functions $T,\mathcal{L}:(0,\infty)\to (0,\infty)$ with $T(x),\mathcal{L}(x)\to 0$ as $x\to 0$ so that for each $(V,\Omega)\in A^C$ with $\gamma=\sF((V,\Omega),A_0)$, the homotopy map $H: I \times A^{C} \rightarrow A^{C}$ defined by
    \begin{align*}
    (t, (V, \Omega)) \mapsto (V_{T(\gamma)t}, \Omega_{T(\gamma)t}):= \Phi^{G}_{V, \Omega}(T(\gamma)t)_{\#}(V, \Omega)
    \end{align*}
    is continuous in the $\sF$-metric satisfying
    \begin{itemize}
        \item $H(t, (V, \Omega)) = (V, \Omega)$ if $(V, \Omega) \in A^C_{0}\cup B$,
        \item $\mathcal{A}^{h}(V_1, \Omega_1) - \mathcal{A}^{h}(V, \Omega) \leq -\mathcal{L}(\gamma)$,
    \end{itemize}
    where $\{\Phi^{G}_{V, \Omega}(t,\cdot)\}_{t\geq 0}\subset \Diff_0^G(M)$ is the flow associated with $\mathcal{X}^G(V,\Omega)$. 
    
    Without loss of generality, we may assume that $\Phi_{i}(x) \in A^{C}$ for all $i \in \mathbb{N}$ and $x \in X$. 
    Let $\mathcal{X}^{G}_{i}(x) = \mathcal{X}^{G}(\Phi^{*}_{i}(x))$ for $x\in X$, which is continuous under the $C^{1}$-topology on $\mathfrak{X}^{G}(M)$ with $\mathcal{X}^{G}_{i}\llcorner Z = 0$. 
    For each $i \in \mathbb{N}$, define $H_{i}: X \rightarrow \Is^{G}(M)$ by $H_{i}(x) = H(\cdot, \Phi_{i}^{*}(x))$. By smoothing out $\mathcal{X}^{G}_{i}$ to some $\widetilde{\mathcal{X}}^{G}_{i}: X \rightarrow \mathfrak{X}^{G}(M)$, which is continuous under the smooth topology with $\widetilde{\mathcal{X}}^{G}_{i}(x) = 0$ for any $x \in Z$ and $||\mathcal{X}^{G}_{i} - \widetilde{\mathcal{X}}^{G}_{i}||_{C^1} \leq 1/i$, we define $\widetilde{H}_{i}: X \rightarrow \Is^{G}(M)$ using $\widetilde{\mathcal{X}}^{G}_{i}$ rather than $\mathcal{X}^{G}_{i}$. Denoting $\Phi_{i}(x) = \widetilde{H}_{i}(1, \Phi_{i}^{*}(x))$, we obtain $\Phi_{i} \in \Pi$ and the same estimate (2.9) of \cite{wangzc2023existenceFour}:
    \[
    \mathcal{A}^{h}(\Phi_{i}(x)) - \mathcal{A}^{h}(\Phi^{*}_{i}(x)) \leq -\mathcal{L}(\sF(\Phi_{i}^{*}(x), A_{0})) + \frac{C}{i}, 
    \]
    where $C > 0$ is a universal constant. 

    Given a min-max sequence $\{\Phi_{i_{j}}(x_j)\}$, by the above estimate and the fact that $\{\Phi_{i}^{*}\}$ is a minimizing sequence, we know $\{\Phi^{*}_{i_{j}}(x_j)\}$ is also a min-max sequence and $\sF(\Phi_{i_j}^{*}(x_j), A_{0}) \to 0$ as $j \to \infty$. 
    Hence, $\lim_{j\to\infty}\sF(\Phi_{i_j}^{*}(x_j), \Phi_{i_j}(x_j)) = 0$ and $\mathbf{C}(\{\Phi_{i}\}) \subset \mathbf{C}(\{\Phi_{i}^{*}\})$. Since $\sF(\Phi_{i_j}(x_j), A_{0}) \rightarrow 0$ as $j \rightarrow \infty$, each element in $\mathbf{C}(\{\Phi_i\})$ is either $(G, \A^h)$-stationary or belongs to $B = \Phi_0(Z)\subset \sEG$.
\end{proof}

\subsection{Almost minimizing} 
We now adapt the almost minimizing property to the $(G, \mathcal{A}^{h})$-functional using $G$-equivariant embedded surfaces. 

\begin{definition}\label{Def: e,d-almost minimizing}
Given $\epsilon, \delta >0$, a $G$-invariant open set $U\subset M$, and $(\Sigma, \Omega)\in \sEG$, we say that {\em $(\Sigma, \Omega)$ is $G$-equivariantly $(\A^h, \epsilon, \delta)$-almost minimizing in $U$} if there does {\em not} exist any isotopy $\psi \in \Is^{G}(U)$, such that
\begin{itemize}
    \item $\A^h(\psi(t, \Sigma, \Omega)) \leq \A^h(\Sigma, \Omega) + \delta$ for all $t\in[0,1]$;
    \item $\A^h(\psi(1, \Sigma, \Omega)) \leq \A^h(\Sigma, \Omega) - \epsilon$.
\end{itemize}
\end{definition}

\begin{definition}[$(G, \A^h)$-almost minimizing pairs]\label{Def: almost minimizing pairs}
Given a $G$-invariant open subset $U\subset M$, a pair $(V, \Omega)\in \VCG(M)$, and a sequence $\{(\Sigma_j, \Omega_j)\}_{j\in \mathbb{N}}\subset \sEG$, we say that $(V, \Omega)$ is {\em $(G, \A^h)$-almost minimizing w.r.t. $\{(\Sigma_j, \Omega_j)\}$ in $U$}, if there exist $\epsilon_j\to 0$ and $\delta_j\to 0$, such that
\begin{itemize}
\item $(\Sigma_j, \Omega_j) \to (V, \Omega)$ in the $\sF$-metric as $j\to \infty$;
\item $(\Sigma_j, \Omega_j)$ is $G$-equivariantly $(\A^h, \epsilon_j, \delta_j)$-almost minimizing in $U$. 
\end{itemize}
Sometimes we also say {\em $(V, \Omega)$ is $(G, \A^h)$-almost minimizing in $U$} without referring to  $\{(\Sigma_j, \Omega_j)\}$.

Moreover, we say $(V_0, \Omega_0)$ is {\em $(G, \A^h)$-almost minimizing in small $G$-annuli w.r.t. $\{(\Sigma_j, \Omega_j)\}$}, if for any $p\in M$, there exists $r_{\am}(G \cdot p) > 0$ such that $(V_0, \Omega_0)$ is $(G, \A^h)$-almost minimizing w.r.t. $\{(\Sigma_j, \Omega_j)\}$ in every $A^{G}_{s, r}(p)\subset\subset A^G_{0,r_{\am}(G \cdot p)}(p)$. 
\end{definition}

The lemma below emerges as a consequence of being $(G, \A^h)$-almost minimizing.  
The proof imitates the argument in Lemma 3.3 of \cite{wangzc2023existenceFour}, which is skipped here. 

\begin{lemma}\label{lem:am implies stationary and stable}
    Let $(V, \Omega)\in \VCG(M)$ be $(G, \A^h)$-almost minimizing in a $G$-invariant open set $U \subset M$, then
    \begin{itemize}
        \item[(\rom{1})] $(V, \Omega)$ is $(G, \A^h)$-stationary in $U$;
        \item[(\rom{2})] $(V, \Omega)$ is $(G, \A^h)$-stable in $U$.
    \end{itemize}
    In particular, if $U$ is appropriately small, then $(V, \Omega)$ is $\Ah$-stationary and $\Ah$-stable in $U$ by Lemma \ref{Lem: G-stationary and stationary for pairs} and Remark \ref{Rem: G-stable and stable for pairs}. 
\end{lemma}

We also need the following $G$-equivariant notions generalized from \cite{colding2018classification, wangzc2023existenceFour}.   

\begin{definition}\label{Def: admissble annuli}
    Given $L \in \mathbb{N}$ and $p \in M$, let  
    \[
    \mathscr{C}^{G} = \{A^{G}_{s_1, r_1}(p), \cdots A^{G}_{s_L, r_L}(p)\}  
    \]
    be a collection of $G$-annuli centered at $G\cdot p$. 
    We say $\mathscr{C}^{G}$ is {\em $L$-admissible} if $2r_{j+1}<s_j$ for all $j=1,\cdots, L-1$, and $B^G_{r_L}(p)$ is appropriately small. 

    In addition, a pair $(V, \Omega)\in \VCG(M)$ is said to be {\em $(G, \A^h)$-almost minimizing in $\mathscr{C}^{G}$ w.r.t. a sequence $\{(\Sigma_j, \Omega_j)\}\subset \sEG$}, if there exists $\epsilon_j\to 0$ and $\delta_j\to 0$ so that 
    \begin{itemize}
        \item $(\Sigma_j, \Omega_j) \to (V, \Omega)$ in the $\mathscr{F}$-metric as $j\to \infty$;
        \item for each $j$, $(\Sigma_j, \Omega_j)$ is $G$-equivariantly $(\A^h, \epsilon_j, \delta_j)$-almost minimizing in at least one $G$-annulus in $\mathscr{C}^{G}$.
    \end{itemize}
\end{definition}

Comparing with \cite[Definition 3.9]{wangzc2023existenceFour}, our $L$-admissible collection $\mathscr{C}^{G}$ of $G$-annuli requires $r_L$ is less than the injectivity radius of $G\cdot p$ so that every ${\rm An}^G\in\mathscr{C}^{G}$ is appropriately small. 

Now, our goal is to find a $(G, \mathcal{A}^{h})$-min-max-pair $(V, \Omega)$ which is $(G, \mathcal{A}^{h})$-almost minimizing in small $G$-annuli. Consider the setup in Section \ref{Sec: Equivariant min-max} where the nontriviality condition (\ref{eq:width nontrivial1}) is met. If $\{\Phi_i\}_{i\in \mathbb{N}}\subset \Pi$ is a pull-tight minimizing sequence obtained by Theorem \ref{thm:pull-tight}, then every $(V, \Omega) \in \mathbf{C}(\{\Phi_i\})$ is $ \mathcal{A}^h$-stationary.

\begin{theorem}[Existence of $(G, \mathcal{A}^{h})$-almost minimizing pairs]\label{thm:existence of almost minimizing pairs}
Let $\Pi$ be an $(X, Z)$-homotopy class generated by some continuous $\Phi_0: X \to \sEG$ relative to $\Phi_0|_Z$ so that \eqref{eq:width nontrivial1} holds. 
Then there exists a min-max subsequence $\{(\Sigma_j, \Omega_j) = \Phi_{i_j}(x_j)\}_{j\in \mathbb{N}} \subset \sEG$ converging to an $\Ah$-stationary pair $(V_0,\Omega_0)\in \mathbf{C}(\{\Phi_i\})$ so that $(V_0,\Omega_0)$ is $(G, \A^h)$-almost minimizing in every $L$-admissible collection of $G$-annuli w.r.t. $\{(\Sigma_j,\Omega_j)\}$, where $L=L(m)$ is an integer depending only on $m:=\dim(X)$.

In addition, up to a subsequence of $\{(\Sigma_j,\Omega_j)\}$, $(V_0, \Omega_0)$ is $(G, \A^h)$-almost minimizing in small $G$-annuli w.r.t. $\{(\Sigma_j, \Omega_j)\}$.   
\end{theorem}

\begin{proof}
    The proof is essentially the same as that of \cite[Appendix]{colding2018classification} (for the area functional) and \cite[Section 3.3]{wangzc2023existenceFour} (for $\Ah$-functional). 
    Namely, if there is no such min-max sequence, then one can apply a combinatorial argument of Almgren-Pitts \cite{pitts2014existence} to find several isotopies supported in many disjoint annuli so that the $\Ah$-functional of a certain sweepout will be pulled down to strictly below $\bL^h(\Pi)$ via these isotopies, which contradicts the definition of $\bL^h(\Pi)$. 
    Since the proof is combinatorial, the argument would also carry over with our $G$-equivariant objects. 
\end{proof}

An immediate consequence of Lemma \ref{lem:am implies stationary and stable} and Theorem \ref{thm:existence of almost minimizing pairs} is that 

\begin{corollary}\label{cor:property R} 
    The limit $\Ah$-stationary pair $(V_0,\Omega_0)\in \mathbf{C}(\{\Phi_i\})$ satisfies 
    \begin{equation}\label{eq:property R}
    \begin{aligned}
    \text{Property {\bf(R)}}:\quad  & \text{for every $L(m)$-admissible collection $\mathscr{C}^{G}$ of $G$-annuli, }\\
    & \text{$(V_0, \Omega_0)$ is $\A^h$-stable in at least one $G$-annulus in $\mathscr{C}^{G}$.}
    \end{aligned}
    \end{equation}
\end{corollary}

\subsection{Regularity of min-max pairs: Part I}

In this subsection, we will introduce the notions concerning $(\Gpm, \mathcal{A}^{h})$-replacements, and show the regularity for the pairs $(V,\Omega)\in\VCG(M)$ with a certain $(\Gpm, \mathcal{A}^{h})$-replacement chain property. 

\begin{definition}
Given an open $G$-subset $U \subset M$ and $(V, \Omega)\in \VCG(M)$, a pair $(V^*, \Omega^*)\in \VCG(M)$ is said to be a {\em $(\Gpm, \A^h)$-replacement} of $(V, \Omega)$ in $U$ if 
\begin{itemize}
    \item[(i)] $(V^*, \Omega^*) = (V, \Omega)$ in $M\setminus\Clos(U)$; 
    \item[(ii)] $\A^h(V^*, \Omega^*) = \A^h(V, \Omega)$; 
    \item[(iii)] $(V^*, \Omega^*)$ is a strongly $\A^h$-stationary and stable $C^{1, 1}$ $(\Gpm, h)$-boundary in $U$.
\end{itemize}
\end{definition}

\begin{definition}
As above, $(V, \Omega)$ is said to have {\em (weak) good $(\Gpm, \mathcal{A}^{h})$-replacement property in $U$} if for any $p\in U$, there exists $r_{G \cdot p}>0$, such that $(V, \Omega)$ has an $(\Gpm, \A^h)$-replacement $(V^*, \Omega^*)$ in any open $G$-annulus $\An^{G} \subset\subset A^{G}_{0, r_{G \cdot p}}(p)$.  
\end{definition}

\begin{proposition}[Classification of tangent cones]\label{Prop: tangent cone}
    Let $(V,\Omega)\in\VCG(M)$ be $\Ah$-stationary in an open $G$-set $U$ and has (weak) good $(\Gpm,\Ah)$-replacement property in $U$. 
    Then $V$ is integer rectifiable in $U$, and for any $p\in \spt(\|V\|)\cap U$, every tangent varifold of $V$ at $p$ is an integer multiple of a plane in $T_pM$. 
\end{proposition}
\begin{proof}
    The proof can taken almost verbatim from \cite[Lemma 6.4]{colding2003min}\cite[Lemma 20.2]{sarnataro2023optimal}. 
\end{proof}

\begin{definition}[$(\Gpm, \mathcal{A}^{h})$-replacement chain property]
Let $(V,\Omega)\in\VCG(M)$ and $U \subset M$ be an open $G$-set. $(V, \Omega)$ is said to have the {\em $(\Gpm,\Ah)$-replacement chain property in $U$} if the following statement holds. 
For any sequence of open $G$-subsets $B_1^G,\cdots, B_k^G\subset \subset U$, there exist a sequence \[(V,\Omega)=(V_0,\Omega_0),(V_1,\Omega_1),\cdots, (V_k,\Omega_k) \subset \VCG(M)\]
satisfying that 
\begin{itemize}
    \item[(i)] $(V_j,\Omega_j)$ is an $(\Gpm, \mathcal{A}^h)$-replacement of $(V_{j-1},\Omega_{j-1})$ in $B_j^G$ for $j=1,\cdots,k$;
    \item[(ii)] $(V_j,\Omega_j)$ is $\A^h$-stationary and stable in $U$;
    \item[(iii)] for another sequence of open $G$-subsets $B_1^G, \cdots, B_{k}^G, \widetilde{B}_{k+1}^G, \cdots , \widetilde{B}_{\ell}^G\subset \subset U$, the sequence of $(\Gpm, \mathcal{A}^h)$-replacements $(\widetilde{V}_j, \widetilde{\Omega}_j)$ can be chosen so that $(\widetilde{V}_j, \widetilde{\Omega}_j) = (V_j,\Omega_j)$ for $j=1,\cdots, k$.
\end{itemize}
\end{definition}

Note the $(\Gpm,\Ah)$-replacement chain property of $(V,\Omega)\in\VCG(M)$ implies that $(V,\Omega)$ is $\Ah$-stationary and stable in $U$, and has the (weak) good $(\Gpm,\Ah)$-replacement property in $U$, which further indicates the rectifiability of $V$ by Proposition \ref{Prop: tangent cone}. 
Additionally, if $(V^*,\Omega^*)$ is a $(\Gpm, \A^h)$-replacement of $(V, \Omega)$ in $B^G\subset U$, then $(V^*,\Omega^*)$ not only is an $\Ah$-replacement of $(V, \Omega)$ in $B^G\subset U$ in the sense of \cite[Definition 3.4]{wangzc2023existenceFour} but also has certain symmetries, i.e. $(V^*,\Omega^*)\in\VCG(M)\subset\VC(M)$. 
Hence, the above definitions concerning $(\Gpm,\Ah)$-replacements are stronger than \cite[Definition 3.4, 3.5, 3.6]{wangzc2023existenceFour}. 
We then have the following regularity theorem by \cite[Theorem 4.4]{wangzc2023existenceFour}. 

\begin{theorem}[First Regularity]\label{Thm: first regularity}
    Let $(V,\Omega)\in\VCG(M)$ satisfy the $(\Gpm,\Ah)$-replacement chain property in a given open $G$-set $U\subset M$. 
    Then $(V,\Omega)$ is induced by a strongly $\Ah$-stationary and stable $C^{1,1}$ $(\Gpm, h)$-boundary in $U$. 
\end{theorem}
\begin{proof}
    By the above definitions, $(V,\Omega)\in\VCG(M)$ also has the replacement chain property in $U$ in the sense of \cite[Definition 3.6]{wangzc2023existenceFour}. 
    Hence, the desired regularity result follows from \cite[Theorem 4.4]{wangzc2023existenceFour} and the fact that $(V,\Omega)\in\VCG(M)$.  
\end{proof}

\subsection{Regularity of min-max pairs: Part II}

In this subsection, we prove the regularity of the $(G,\Ah)$-almost minimizing pairs $(V,\Omega)\in\VCG(M)$ in an appropriately small open $G$-set $U$ by constructing $(\Gpm,\Ah)$-replacement chains. 
Throughout this subsection, we always assume \[\mbox{$U\subset M$ is an {\em appropriately small} (cf. Definition \ref{Def: appropriately small}) open $G$-set}\]
with connected components $\{U_i\}_{i=1}^{\#G}$. 

To begin with, consider a constrained $\Ah$-minimizing problem. 
For any $G$-equivariant $\Gpm$-separating embedded surface $(\Sigma,\Omega)\in\sEG$ and $\delta>0$, take 
\begin{align*}
    \Is^{G,h}_\delta(U) &:= \{\psi\in\Is^G(U): \Ah(\psi(t,(\Sigma,\Omega)))\leq \Ah(\Sigma,\Omega) + \delta \}\\
    \Is^h_\delta(U_i) &:= \{\psi\in\Is(U_i): \Ah(\psi(t,(\Sigma,\Omega)))\leq \Ah(\Sigma,\Omega) + \delta \}
\end{align*}
Then a sequence $\{(\Sigma_k,\Omega_k)\}_{k\in\N}\subset \sEG$ is said to be {\em minimizing in Problem $(\Sigma,\Omega,\Is^{G,h}_\delta(U))$} if there exists a sequence $\{\psi_k\}_{k\in\N}\subset \Is^{G,h}_\delta(U)$ with $(\Sigma_k,\Omega_k)=\psi_k(1,(\Sigma,\Omega))$ so that 
\[\Ah(\Sigma,\Omega)\geq \Ah(\Sigma_k,\Omega_k)\to m^G_\delta:=\inf\{ \Ah(\psi(1,(\Sigma,\Omega)))  : \psi\in\Is^{G,h}_\delta(U)\} ~{\rm as}~k\to\infty\]
Similarly, if $\{\hat{\psi}_k\}_{k\in\N}\subset \Is^{h}_\delta(U_i)$ with $(\hat{\Sigma}_k,\hat{\Omega}_k)=\hat{\psi}_k(1,(\Sigma,\Omega))$ so that 
\[\Ah(\Sigma,\Omega)\geq \Ah(\hat{\Sigma}_k,\hat{\Omega}_k)\to m_\delta:=\inf\{ \Ah(\hat{\psi}(1,(\Sigma,\Omega)))  : \hat{\psi}\in\Is^{h}_\delta(U_i)\} ~{\rm as}~k\to\infty,\]
then we say $\{(\hat{\Sigma}_k,\hat{\Omega}_k)\}_{k\in\N}$ is {\em minimizing in Problem $(\Sigma,\Omega,\Is^h_\delta(U_i))$}. 

Using the arguments in Lemma \ref{Lem: G-minimizing and minimizing for boundary}, we also have the following result indicating the equivalence between minimizing in Problem $(\Sigma,\Omega,\Is^{G,h}_\delta(U))$ and minimizing in Problem $(\Sigma,\Omega,\Is^h_{\delta/\#G}(U_i))$. 

\begin{lemma}\label{Lem: constrained minimzing and minimizing}
    Given $i\in\{1,\dots,\#G\}$, $\{\psi_k\}_{k\in\N}\subset \Is^{G,h}_\delta(U)$ and  $(\Sigma_k,\Omega_k)=\psi_k(1,(\Sigma,\Omega))$, define
    \[
    \hat{\psi}_k :=\left\{
    \begin{array}{ll}
         \psi_k & {\rm in}~U_i,  \\
         id & {\rm in}~ M\setminus U_i,
    \end{array}
    \right . 
    \quad
    (\hat{\Sigma}_k,\hat{\Omega}_k) := \hat{\psi}_k(1,(\Sigma,\Omega)) = \left\{
    \begin{array}{ll}
        (\Sigma_k,\Omega_k) & {\rm in}~U_i, \\
        (\Sigma,\Omega) & {\rm in}~ M\setminus U_i.
    \end{array}
    \right.
    \]
    Then $(\Sigma_k,\Omega_k)$ is minimizing in Problem $(\Sigma,\Omega,\Is^{G,h}_\delta(U))$ if and only if $(\hat{\Sigma}_k,\hat{\Omega}_k)$ is minimizing in Problem $(\Sigma,\Omega,\Is^h_{\delta/\#G}(U_i))$. 
\end{lemma}
\begin{proof}
    By \eqref{Eq: Ah in appropriately small set} and the above definitions, we have 
    \begin{align}\label{Eq: Ah in appropriate small set 2}
        \Ah(\psi_k(t,(\Sigma,\Omega))) &= \#G\cdot \Ah(\hat{\psi}_k(t,(\Sigma\cap U_i,\Omega\cap U_i))) + \frac{\#G}{2}\int_{U_i}h + \Ah(\Sigma\setminus U,\Omega\setminus U) \nonumber
        \\&= \#G\cdot \Ah(\hat{\psi}_k(t,(\Sigma,\Omega))) + C_0,
    \end{align}
    where $C_0=C_0(U,\Sigma,\Omega,h,\#G)$ is a constant (independent of $\psi_k$). 
    In particular, 
    \begin{align}
    \Ah(\Sigma,\Omega)&=\#G\cdot \Ah(\Sigma\cap U_i,\Omega\cap U_i) + \frac{\#G}{2}\int_{U_i}h + \Ah(\Sigma\setminus U,\Omega\setminus U)\nonumber\\
    &= \#G\cdot \Ah(\Sigma, \Omega) + C_0.
    \end{align}
    Thus, $\{\hat{\psi}_k\}\subset \Is^h_{\delta/\#G}(U_i)$, and every $\hat{\phi}\in \Is^h_{\delta/\#G}(U_i)$ can be recovered to $\phi\in\Is^{G,h}_\delta(U)$ by taking $\phi\llcorner U_j = g\circ \hat{\phi} \circ g^{-1} $ for $U_j=g\cdot U_i$. 
    Moreover, the above formulae also imply that $m^G_\delta = \#G\cdot m_{\delta/\#G} + C_0$. 
    Therefore, $\Ah(\Sigma,\Omega)\geq \Ah(\Sigma_k,\Omega_k)\to m^G_\delta$ if and only if $\Ah(\Sigma,\Omega)\geq \Ah(\hat{\Sigma}_k,\hat{\Omega}_k)\to m_{\delta/\#G}$ by the above equalities again. 
\end{proof}

In the following lemma, we show that any isotopy in a small enough open $G$-set which doesn't increase $\Ah(\Sigma_k,\Omega_k)$ can also be replaced by an isotopy in $\Is^{G,h}_\delta(U)$. 

\begin{lemma}\label{Lem: constrained minimzing and minimizing 2}
    Let $\{(\Sigma_k,\Omega_k)\}_{k\in\N}$ be minimizing in Problem $(\Sigma,\Omega,\Is^{G,h}_\delta(U))$. 
    Given any $G$-subset $U'\subset\subset U$, there exists $\rho_0>0$ and $k_0>>1$ so that for any $k\geq k_0$ and $B^G_{2\rho}(x)\subset U'$ with $\rho<\rho_0$, if $\varphi\in\Is^G(B^G_\rho(x))$ with $\Ah(\varphi(1,(\Sigma_k,\Omega_k)))\leq \Ah(\Sigma_k,\Omega_k)$, then there is $\Phi\in\Is^G(B^G_{2\rho}(x))$ with 
    \[\Phi(1,\cdot)=\varphi(1,\cdot) \quad {and}\quad \Ah(\Phi(t,(\Sigma_k,\Omega_k)))\leq \Ah(\Sigma_k,\Omega_k) + \delta, ~\forall t\in [0,1] .\]
    Moreover, $\rho_0$ depends on $\mcH^2(\Sigma),\|h\|_{L^\infty},U',M , \delta$, but does not depend on $\{(\Sigma_k,\Omega_k)\}_{k\in\N}$. 
\end{lemma}
\begin{proof}
    Let $\{\hat{\psi}_k\}$ and $\{(\hat{\Sigma}_k,\hat{\Omega}_k)\}$ be given as in Lemma \ref{Lem: constrained minimzing and minimizing} w.r.t some fixed $i$. 
    Hence, $\{(\hat{\Sigma}_k,\hat{\Omega}_k)\}$ is minimizing in Problem $(\Sigma,\Omega,\Is^h_{\delta/\#G}(U_i))$ so that \cite[Lemma 4.5]{wangzc2023existenceFour} is applicable in $U_i':=U'\cap U_i$, which gives us the desired $\rho_0>0$ and $k_0>>1$. 
    
    Indeed, take any $\varphi\in\Is^G(B^G_\rho(x))$ so that $\Ah(\varphi(1,(\Sigma_k,\Omega_k)))\leq \Ah(\Sigma_k,\Omega_k)$, where $k\geq k_0$, $\rho<\rho_0$, and $x\in U_i$ with $B^G_{2\rho}(x)\subset U'$. 
    Let $\hat{\varphi} = \varphi$ in $B_\rho(x) =U_i\cap B^G_{\rho}(x)$ and $\hat{\varphi} = id$ outside $B_{\rho}(x)$. 
    Then $\hat{\varphi}\in\Is(B_\rho(x))$ and $\Ah(\hat{\varphi}(1,(\Sigma_k,\Omega_k)))\leq \Ah(\Sigma_k,\Omega_k)$ by \eqref{Eq: Ah in appropriate small set 2}. 
    Thus, by \cite[Lemma 4.5]{wangzc2023existenceFour}, we have an isotopy $\hat{\Phi}\in\Is(B_{2\rho}(x))$ with $\hat{\Phi}(1,\cdot)=\hat{\varphi}(1,\cdot)$ and $\Ah(\hat{\Phi}(t,(\Sigma_k,\Omega_k)))\leq \Ah(\Sigma_k,\Omega_k) + \delta/\#G$ for all $t\in [0,1]$. 
    After recovering $\hat{\Phi}$ to $\Phi\in\Is^{G}(B^G_{2\rho}(x))$ by taking $\Phi\llcorner U_j = g\circ \hat{\Phi} \circ g^{-1} $ for $U_j=g\cdot U_i$, we can use \eqref{Eq: Ah in appropriate small set 2} again to show that $\Phi$ is the desired isotopy. 
\end{proof}

Combining Lemma \ref{Lem: constrained minimzing and minimizing 2}, \ref{Lem: G-minimizing and minimizing for boundary} and Proposition \ref{Prop: compactness for G-stable boundary}, we have the following regularity result for the minimizers in Problem $(\Sigma,\Omega,\Is^{G,h}_\delta(U))$. 

\begin{proposition}\label{Prop: regularity of constrained minimizer}
    Suppose $(\Sigma,\Omega)\in\sEG$ is $G$-equivariantly $(\Ah,\epsilon,\delta)$-almost minimizing in an appropriately small open $G$-set $U$. 
    Let $\{\psi_k\}\subset\Is^{G,h}_\delta(U)$ so that $\{(\Sigma_k,\Omega_k)=\psi_k(1,(\Sigma,\Omega))\}$ is minimizing in Problem $(\Sigma,\Omega,\Is^{G,h}_\delta(U))$. 
    Then, up to a subsequence, $(\Sigma_k,\Omega_k)$ converges to some $(\hat{V},\hat{\Omega})\in\VCG(M)$ so that  
    \begin{itemize}
        \item[(i)] $\Ah(\Sigma,\Omega) - \epsilon \leq \Ah(\hat{V},\hat{\Omega})\leq\Ah(\Sigma,\Omega)$;
        \item[(ii)] $(\hat{V},\hat{\Omega})\llcorner U$ is a strongly $\Ah$-stationary and stable $C^{1,1}$ $(\Gpm,h)$-boundary in $U$. 
    \end{itemize}
\end{proposition}

\begin{proof}
    (i) follows directly from Definition \ref{Def: e,d-almost minimizing}. 
    Similar to the proof of Lemma \ref{lem:am implies stationary and stable} (cf. \cite[Lemma 3.3]{wangzc2023existenceFour}), one can show that $(\hat{V},\hat{\Omega})\llcorner U$ is strongly $(G,\Ah)$-stationary and $G$-stable in $U$ because $\{(\Sigma_k,\Omega_k)\}$ is minimizing in Problem $(\Sigma,\Omega,\Is^{G,h}_\delta(U))$.
    Then since $U$ is appropriately small, we have $(\hat{V},\hat{\Omega})\llcorner U$ is strongly $\Ah$-stationary and stable in $U$ by Lemma \ref{Lem: G-stationary and stationary for pairs} and Remark \ref{Rem: G-stable and stable for pairs}
    
    Next, for any $p\in U$, take $r_0\in (0, \min\{r_1,\rho_0\})$, where $r_1=\dist_M(p,\bd U)/4$ and $\rho_0$ is given by Lemma \ref{Lem: constrained minimzing and minimizing} for $U'=B_{r_1}^G(p)$. 
    We claim that $(\hat{V},\hat{\Omega})$ has the $(\Gpm,\Ah)$-replacement chain property in $B^G_{r_0}(p)$, which indicates the regularity of $(\hat{V},\hat{\Omega})\llcorner U$ by Theorem \ref{Thm: first regularity} and the arbitrariness of $p\in U$. 
    Indeed, since $U$ is appropriately small, the proof of \cite[Proposition 4.6]{wangzc2023existenceFour} would carry over with our $G$-equivariant objects by using Theorem \ref{Thm: isotopy minimizer}, Proposition \ref{Prop: compactness for G-stable boundary}, and Lemma \ref{Lem: constrained minimzing and minimizing 2} in place of \cite[Theorem 1.25, Proposition 1.24, Lemma 4.5]{wangzc2023existenceFour} respectively. 
\end{proof}

Now, we can show the regularity for $(G,\Ah)$-almost minimizing pairs. 

\begin{theorem}[Regularity of $(\Gpm,\Ah)$-almost minimizing pairs]\label{Thm: regularity of a.m. pairs}
    Given any appropriately small open $G$-set $U\subset M$, let $(V,\Omega)\in\VCG(M)$ be $(\Gpm,\Ah)$-almost minimizing w.r.t. $\{(\Sigma_j,\Omega_j)\}_{j\in\N}\subset \sEG$ in $U$. 
    Then $(V,\Omega)\llcorner U$ is induced by a strongly $\Ah$-stationary and stable $C^{1,1}$ $(\Gpm,h)$-boundary. 
\end{theorem}
\begin{proof}
    Since $U$ is appropriately small, it follows from Lemma \ref{lem:am implies stationary and stable} that $(V,\Omega)\llcorner U$ is strongly $\Ah$-stationary and stable. 

    Fix any open $G$-subsets $B_1^G\subset \subset U'\subset\subset U$. 
    Note $(\Sigma_j,\Omega_j)$ is $G$-equivariantly $(\Ah,\epsilon_j,\delta_j)$-almost minimizing in $U$ for some $\epsilon_j\to 0$ and $\delta_j\to 0$ as $j\to\infty$. 
    Hence, for each $j\in\N$, we can take a minimizing sequence $\{(\Sigma_{j,l}^1,\Omega_{j,l}^1)\}_{l\in\N}$ for Problem $(\Sigma_j,\Omega_j,\Is^{G,h}_{\delta_j}(B^G_1))$, and apply Proposition \ref{Prop: regularity of constrained minimizer} to see $\lim_{l\to\infty}(\Sigma_{j,l}^1,\Omega_{j,l}^1) = (V_j^1,\Omega_j^1)\in\VCG(M)$ (up to a subsequence) so that 
    \begin{itemize}
        \item[(1)] $\Ah(\Sigma_j,\Omega_j) - \epsilon_j \leq \Ah(V^1_j,\Omega^1_j)\leq\Ah(\Sigma_j,\Omega_j)$;
        \item[(2)] $(V^1_j,\Omega^1_j) = (\Sigma_j,\Omega_j)$ in $M\setminus\Clos(B^G_1)$;
        \item[(3)] $(V^1_j,\Omega^1_j)$ is a strongly $\Ah$-stationary stable $C^{1,1}$ $(\Gpm,h)$-boundary in $B^G_1$. 
    \end{itemize}
    Up to a subsequence, $(V^1_j,\Omega^1_j)$ converges to $(V^1,\Omega^1)\in\VCG(M)$. 
    Combining Proposition \ref{Prop: compactness for G-stable boundary} with (1)-(3), we see $(V^1,\Omega^1)$ is a $(\Gpm,\Ah)$-replacement of $(V,\Omega)$ in $B^G_1$. 
    Moreover, $(\Sigma^1_{j,l},\Omega^1_{j,l})$ is also $G$-equivariantly $(\Ah,\epsilon_j,\delta_j)$-almost minimizing in $U$ since $(\Sigma^1_{j,l},\Omega^1_{j,l}) = \psi^1_l(1,(\Sigma_{j},\Omega_{j}))$ for some $\psi^1_l\in\Is^{G,h}_{\delta_j}(B^G_1)$ and $\Ah(\Sigma^1_{j,l},\Omega^1_{j,l})\leq \Ah(\Sigma_j,\Omega_{j})$. 

    Next, for another open $G$-set $B^G_2\subset\subset U'$, the above arguments can be applied in $B^G_2$ to a diagonal sequence $\{(\Sigma^1_{j,l(j)},\Omega^1_{j,l(j)})\}_{j\in\N}$ that converges to $(V^1,\Omega^1)$. 
    Then we obtain a $(\Gpm,\Ah)$-replacement $(V^2,\Omega^2)=\lim(\Sigma^2_{j,l(j)},\Omega^2_{j,l(j)})$ of $(V^1,\Omega^1)$ in $B^G_2$ so that each $(\Sigma^2_{j,l(j)},\Omega^2_{j,l(j)})$ is $G$-equivariantly $(\Ah,\epsilon_j,\delta_j)$-almost minimizing in $U$. 
    Hence, $(V,\Omega)$ has $(\Gpm,\Ah)$-replacement chain property in $U'$ by repeating this procedure, which indicates the regularity of $(V,\Omega)$ by Theorem \ref{Thm: first regularity}. 
\end{proof}

\subsection{Proof of Theorem \ref{thm:pmc min-max theorem}}

By Theorem \ref{thm:existence of almost minimizing pairs}, there exists an $\A^h$-stationary pair $(V_0,\Omega_0)\in \mathbf{C}(\{\Phi_i\})$ that is $(G, \A^h)$-almost minimizing in small $G$-annuli w.r.t. a min-max subsequence $\{(\Sigma_j, \Omega_j) = \Phi_{i_j}(x_j)\}_{j\in \mathbb{N}} \subset \sEG$. 
Let $\{B^G_{r_i}(p_i)\}_{i=1}^{\mathfrak{m}}$ be a finite set of appropriately small open $G$-sets covering $M$ with radius $r_i=\frac{1}{2}\min\{ r_{am}(G\cdot p_i), \inj(G\cdot p_i)\}$, where $r_{am}(G\cdot p_i)$ is given by Definition \ref{Def: almost minimizing pairs}. 
After applying Theorem \ref{Thm: regularity of a.m. pairs} to $(V_0,\Omega_0)$ in any open $G$-set $U\subset \subset A^G_{0,r_i}(p_i)$, we see $(V_0,\Omega_0)=(\Sigma_0,\Omega_0)$ is a strongly $\Ah$-stationary $C^{1,1}$ $(\Gpm, h)$-boundary in $M\setminus\{G\cdot p_1,\cdots, G\cdot p_{\mathfrak{m}}\}$, which is also $\Ah$-stable in any $U\subset\subset A^G_{0,r_i}(p_i)$. 
Finally, using Proposition \ref{Prop: compactness for G-stable boundary}, \ref{Prop: tangent cone} in place of \cite[Proposition 1.24, 4.1]{wangzc2023existenceFour}, the arguments in \cite[Section 4.4]{wangzc2023existenceFour} can be taken almost verbatim to show the regularity of $(V_0,\Omega_0)$ extends across each $G\cdot p_i$.

\section{Compactness for min-max $(G_{\pm},h)$-boundaries}\label{Sec: Convergence}

\subsection{Strong convergence and weighted genus bound}  
Given an $h \in C^{\infty}_{\Gpm}(M)$ and a sequence of positive numbers $\epsilon_k \to 0$ as $k \to \infty$, we denote $\A^{\epsilon_k h}$ simply by $\A^{k}$. With notations from Section \ref{Sec: Equivariant min-max}, consider the equivariant min-max problem associated with $\Pi$ for each $\A^k$, $k \in \mathbb{N}$. By assuming that the nontriviality condition (\ref{eq:width nontrivial1}) is met for all $k$, we apply Theorem \ref{thm:existence of almost minimizing pairs} to the $\A^{k}$-functional for each $k$. This yields a $(G, \A^{k})$-min-max pair $(V_{k}, \Omega_{k}) \in \VCG(M)$ and an associated min-max sequence $\{(\Sigma_{k, j}, \Omega_{k, j})\}_{j\in \mathbb{N}}\subset \sEG$, such that $(V_k, \Omega_k)$ is $\A^k$-stationary and $(G, \A^k)$-almost minimizing in small $G$-annuli w.r.t. $\{(\Sigma_{k, j}, \Omega_{k, j})\}$. By Theorem \ref{thm:pmc min-max theorem}, $(V_k, \Omega_k)$ is a strongly $\A^k$-stationary, $C^{1,1}$ $(\Gpm, \epsilon_k h)$-boundary $(\Sigma_k, \Omega_k)$ with $\A^k(\Sigma_k, \Omega_k) = \mathbf{L}^{\epsilon_k h}(\Pi)$. 

In this part, our goal is to show the smooth regularity of a subsequential varifold limit $V_{\infty}$ of $\{\Sigma_k\}$ and upgrade the convergence to $C^{1, 1}_{loc}$. Moreover, for specially chosen $h$, we prove the weighted genus bound (see \eqref{eq:genus bound0}).  

To begin with, by Corollary \ref{cor:property R}, we note that $V_{\infty}$ satisfies 
\begin{equation}\label{eq:property R2}
\begin{aligned}
    \text{Property {\bf(R')}}:\quad  & \text{for every $L(m)$-admissible collection $\mathscr{C}^{G}$ of $G$-annuli,}\\
    & \text{$V_\infty$ is stable (for area) in at least one $G$-annulus in $\mathscr{C}^{G}$.}
\end{aligned}
\end{equation}

Arguing similarly as in the proof of \cite[Proposition 5.1]{wangzc2023existenceFour}, we have the following proposition, which is crucial for the removable singularity step. 

\begin{proposition}\label{prop:annulus picking for Ak almost minimizing}
There exists a subsequence of $\{(\Sigma_k, \Omega_k)\}_{k\in \mathbb{N}}$, such that
\begin{equation}\label{eq:property S}
\begin{aligned}
    \text{Property {\bf(S)}}:\quad  & \text{given any $p\in M$, there exists $r_{G \cdot p}>0$, such that}\\
    & \text{for each $A_{s, r}^{G}(p)$ with $0 < s < r < r_{G \cdot p}$, }\\
    & \text{$(\Sigma_k, \Omega_k)$ is $\A^k$-stable in $A_{s, r}^{G}(p)$ for all sufficiently large $k$.}
\end{aligned}
\end{equation}
\end{proposition}

\begin{theorem}\label{thm:convergence of pmc to minimal}
$\spt\|V_\infty\|$ is a closed embedded $G$-invariant minimal surface $\Sigma_\infty$. Moreover, there exists a finite set of points $\mathcal{Y}\subset M$, such that up to a subsequence, $\{\Sigma_k\}_{k \in \mathbb{N}}$ converges in $C^{1,1}_{loc}$ to $\Sigma_\infty$ in any compact subset of $M \setminus \mathcal{Y}$.
\end{theorem}
\begin{proof}
The theorem is readily verified by combining Proposition \ref{Prop: compactness for G-stable boundary}, Property {\bf(S)} \eqref{eq:property S}, and the standard removable singularity theorem (see \cite{schoen1981regularity}). 
\end{proof}

Although $V_\infty$ may not be $G$-equivariantly $\A^{0}$-almost minimizing in small $G$-annuli, we may choose a special $h \in C^{\infty}_{\Gpm}(M)$ so that the limit minimal surface still has the total genus less than $\mathfrak{g}_0$ -- the genus of elements in $\sEG$.  

\begin{theorem}[Genus bound]\label{thm:genus bound}
Let $(M, g_{M})$ be a closed orientable $3$-dimensional Riemannian manifold and $G$ be a finite group acting freely and effectively as isometries on $M$ so that $G$ admits an index $2$ subgroup $G_{+}$ with coset $G_{-} = G \setminus G_{+}$. Consider $V_{\infty}$ as above. Suppose that there are finitely many pairwise disjoint appropriately small open $G$-balls $B^{G}_{1}, \cdots, B^{G}_{\alpha} \subset M$ such that 
\begin{enumerate}
    \item $\pi(\spt\|V_\infty\| \cap B^{G}_j)$ is an embedded disk for $j=1,\cdots,\alpha$;
    \item $h\equiv 0$ in a small neighborhood of $\spt\|V_\infty\| \setminus \cup_{j} B^{G}_j$.
\end{enumerate}
Assume that $ V_\infty=\sum_{i=1}^N m_i[\Gamma_i]$,
where $\{\Gamma_i\}_{i=1}^N$ is a pairwise disjoint collection of connected, closed, embedded, minimal surfaces. 
Denote by $I_O \subset \{1,\cdots,N\}$ (resp. $I_U$) the collection of $i$ such that $\Gamma_i$ is orientable (resp. non-orientable). Then we have
\begin{equation}\label{eq:genus bound0}
\sum_{i\in I_O} m_i\cdot \mathfrak{g}(\Gamma_i)+\frac{1}{2}\sum_{i\in I_U} m_i \cdot (\mathfrak{g}(\Gamma_i)-1)\leq \mathfrak{g}_0,
\end{equation}
where $\mathfrak{g}_0$ and $\mathfrak{g}(\Gamma_i)$ are the genus of $\Sigma_0$ and $\Gamma_i$, respectively. 
\end{theorem}

\begin{proof}
    Let $\{\gamma_i\}_{i=1}^k$ be a collection of simple closed curves contained in $\cup_{i=1}^N \Gamma_i$. By assumption (1), we can perturb $\{\gamma_i\}_{i=1}^k$ in the same isotopy class so that $\cup_i\gamma_i$ does not intersect $\cup_j B^{G}_j$. Hence by assumption (2), $\epsilon_k h \equiv 0$ in a neighborhood of $\cup_i\gamma_i$. 
    Additionally, one notices that \cite[Proposition 5.3]{wangzc2023existenceFour} can be easily generalized to our $G$-equivariant setting. 
    Hence, by possibly perturbing $\{\gamma_i\}$ and shrinking $r_{G \cdot p}>0$, 
    we can assume that $(\Sigma_{k,j(k)},\Omega_{k,j(k)})$ is $G$-equivariantly $(\A^0, \epsilon_k, \delta_k)$-almost minimizing in $B^{G}_{r_{G\cdot p}}(p)$ for any $p\in\cup_i\gamma_i$. To prove the curve lifting lemma, we follow the strategy in \cite[Proposition 2.2]{ketover2019genus} (see also \cite{de2010genus}). For simplicity, we only consider the setting of two appropriately small open $G$-balls $B_{1}^{G}$ and $B_{2}^{G}$ intersecting along a curve $\gamma$. Suppose further that the $G$-equivariantly $\A^{0}$-almost minimizing property holds in $B_{1}^{G} \cup B_{2}^{G}$. By taking successive $(\Gpm, \A^0)$-replacements of $(\Sigma_{k,j(k)},\Omega_{k,j(k)})$ on $B_{1}^{G}$ and $B_{2}^{G}$, we obtain a new $G$-invariant surface $W_{k}$ (arises topologically from $\Sigma_{k, j(k)}$ after finitely many $G$-equivariant surgeries), which has the same limit as $\Sigma_{k, j(k)}$. Combining Schoen's estimates, a No Folding property, and an integrated Gauss-Bonnet argument, we show the graphical smooth convergence of $W_{k}$ away from finitely many points. Hence, we can lift $\gamma$ with the correct multiplicity by perturbing it slightly to avoid those singularities. 
\end{proof}

\subsection{Existence of supersolution} 
For a sequence of strongly $\A^{\epsilon_k h}$-stationary $C^{1,1}$ $\epsilon_kh$-boundaries $\{(\Sigma_k,\Omega_k)\}$ converging as varifolds to a closed $2$-sided minimal surface $\Sigma$, Wang-Zhou \cite{wangzc2023existenceFour} proved in the non-equivariant setting that $\Sigma$ admits a nonnegative weak supersolution to a variant of the Jacobi equation provided that $h\llcorner\Sigma$ changes sign and the convergence is $C^{1,1}_{loc}$ away from a finite set $\mathcal{Y}$ with multiplicity $m\geq 2$. 
Noting $G$ is finite and $h\in C^\infty_{\Gpm}(M)$ must change sign unless $h\equiv 0$, we immediately have the following generalization. 

\begin{proposition}\label{Prop: supersolution}
    Let $\{(\Sigma_k,\Omega_k)\}_{k\in\N}$ be a sequence of strongly $\A^{\epsilon_k h}$-stationary $C^{1,1}$ $(\Gpm,\epsilon_kh)$-boundaries with $\lim_{k\to\infty}\epsilon_k= 0$ so that $\Sigma_k$ converges as varifolds to a closed embedded $2$-sided minimal $G$-surface $\Sigma$ with multiplicity $m\geq 2$. 
    Suppose the convergence is also $C^{1,1}_{loc}$ away from a finite $G$-set $\mathcal{Y}$. 
    Then $\Sigma$ admits a non-negative $G$-invariant function $\varphi\in W^{1,2}(\Sigma)$ satisfying $\|\varphi\|_{L^2(\Sigma)}=1$ and 
    \begin{align}\label{Eq: supersolution}
        \int_{\Sigma}\langle\nabla\varphi,\nabla f\rangle - \left(\Ric(\nu,\nu) + |A^\Sigma|^2\right)\varphi f d\mathcal{H}^2 \geq \int_\Sigma 2chf d\mathcal{H}^2,\quad \forall f\in C^1(\Sigma) {\rm ~and~} f\geq 0, 
    \end{align}
    where $c\geq 0$ is a constant so that $c=0$ if $m\geq 3$ is odd. 
\end{proposition}
\begin{proof}
    As we explained before, \cite[Proposition 6.4]{wangzc2023existenceFour} can be applied directly to get a non-negative function $\varphi\in W^{1,2}(\Sigma)$ satisfying \eqref{Eq: supersolution}. 
    To show $\varphi$ is $G$-invariant, we recall the constructions in \cite[Section 6]{wangzc2023existenceFour}. 
    Take any open $G$-set $\mathcal{U} \subset \subset \Sigma \setminus \mathcal{Y}$ and a unit normal $\nu$ of $\Sigma$. 
    Then for sufficiently large $k$, $\Sigma_{k}$ admits a decomposition by ordered $m$-sheets $\Gamma_k^1\leq \cdots \leq \Gamma_k^m$ inside a thickened $G$-neighborhood $\mathcal{U}_\delta = \mathcal{U} \times (-\delta, \delta)$, and each sheet $\Gamma_k^\iota$ ($1\leq \iota \leq m$) is a normal graph of some function $u_k^\iota\in C^{1,1}(\mathcal{U})$, such that $u_k^1\leq \cdots \leq u_k^m$, and $u_k^\iota \to 0$ in $C^{1,1}(\mathcal{U})$ as $k\to\infty$. 
    Although a single sheet $\Gamma_k^i$ may not be $G$-invariant, we notice $\Gamma_k^m\cup\Gamma_k^1$ is $G$-invariant, and thus $|\varphi_k| = \varphi_k :=u_k^m -u_k^1$ is $G$-invariant. Since $\varphi\llcorner\mathcal{U}$ is the $C^{1,\alpha}$ limit of $\varphi_k/\|\varphi_k\|_{L^2(\mathcal{U})}$, we know $\varphi$ is $G$-invariant. 
\end{proof}

\section{Multiplicity one for generalized Simon-Smith min-max theory}\label{Sec: multiplicity one}

In this section, we will generalize two multiplicity one theorems in \cite{wangzc2023existenceFour} to the equivariant setting, namely the relative min-max in the space of oriented $G$-equivariant $\Gpm$-separating surfaces in Section \ref{Sec: multi one for relative} and the classical min-max in the space of un-oriented $G$-equivariant $\Gpm$-separating surfaces in Section \ref{Sec: multi one for classical}.

\subsection{Multiplicity one for relative equivariant min-max}\label{Sec: multi one for relative}
To start with, we have the following compactness theorem, which admits a similar proof as in \cite{wangzc2023existenceFour}.   

\begin{theorem}\label{Thm: compactness for property R2}
    Let $L \in \mathbb{N}$ and $C > 0$ be a constant. Consider $\{\Sigma_k\}$ a sequence of closed embedded $G$-invariant minimal surfaces satisfying $\sup_{k} \mathcal{H}^2(\Sigma_k) \leq C$ and Property \emph{\textbf{(R')}} (\ref{eq:property R2}). Then $\Sigma_k$ converges subsequentially as varifolds to a closed embedded $G$-invariant minimal surface $\Sigma$ possibly with integer multiplicities. Moreover, $\Sigma$ is degenerate if $\Sigma_k \neq \Sigma$ for infinitely many $k$. 
\end{theorem}

Recall that the space $\sEG$ of $G$-equivariant $\Gpm$-separating surfaces of genus $\mathfrak{g}_{0}$ is defined in Section \ref{Sec: PMC min-max}. By choosing an appropriate prescribing function $h \in C^\infty_{\Gpm}(M)$, we obtain the first multiplicity one type result as follows. 

\begin{theorem}\label{Thm: multiplicity one}
	Let $(M, g_{_M})$ be a closed connected orientable $3$-dimensional Riemannian manifold, and $G$ be a finite group acting freely and effectively as isometries on $M$ so that $G$ admits an index $2$ subgroup $G_+$ with coset $G_-=G\setminus G_+$. Suppose that $X$ is a finite dimensional cubical complex with $Z \subset X$ a subcomplex. Consider $\Phi_{0}: X \rightarrow \sEG$ a continuous map and $\Pi$ the $(X, Z)$-homotopy class of $\Phi_{0}$. Assume that 
    \[
    \mathbf{L}(\Pi) > \max_{x \in Z} \mathcal{H}^2(\Phi_0(x)).  
    \]
	Then there exists a closed embedded $G$-invariant minimal surface $\Gamma$ with connected components $\{\Gamma_{j}\}_{j=1}^J$ and integer multiplicities $\{m_j\}_{j=1}^J$ so that 
	\begin{itemize}
		\item[(i)] if $\Gamma_j$ is unstable and $2$-sided, then $m_j=1$;
		\item[(ii)] if $\Gamma_j$ is $1$-sided, then its connected double cover is stable. 
	\end{itemize}
	Moreover the weighted total genus of $\Gamma$ (\ref{eq:genus bound0}) is bounded by $\mathfrak{g}_0$.
\end{theorem}

\begin{proof}
    It is sufficient to verify the theorem when $g_{_M}$ is $G$-bumpy, i.e. any finite cover of a closed embedded $G$-invariant minimal hypersurface in $(M, g_{_M})$ is non-degenerate. Given a constant $C > 0$ (e.g. $C := \mathbf{L}(\Pi) + 1$), set $\mathcal{M}(C)$ the space of all closed embedded $G$-invariant minimal surfaces $\Gamma$ satisfying $\mathcal{H}^2(\Gamma) \leq C$ and Property \textbf{(R')}. 
    By Theorem \ref{Thm: compactness for property R2}, $\mathcal{M}(C) = \{S_1, \ldots, S_\alpha\}$ is a finite set. Take $p_1, \ldots, p_{\alpha}$ in $M$ so that $p_{i} \in S_{j}$ if and only if $j = i$. 
    Let $r > 0$ be small enough such that 
    \begin{itemize}
        \item $B^{G}_{r}(p_1), \cdots, B^{G}_{r}(p_\alpha)$  are pairwise disjoint appropriately small open $G$-balls;
        \item $B_{r}^G(p_{i})$  intersects $S_j$ if and only if $j = i$;
        \item $\pi(B_{r}^G(p_{i}) \cap S_{i})$ is an embedded disk for all $i = 1, \cdots, \alpha$. 
    \end{itemize}
    Now, choose $h \in C^\infty_{\Gpm}(M)$ with $h(M) \subset [-1, 1]$ satisfying that for all $i = 1, \cdots, \alpha$, 
    \begin{enumerate}
        \item $h = 0$ outside $\cup_{i} B^G_r(p_i)$;
        \item $h > 0$ in some component of $B_{r/2}^G(p_{i})$ (then $h < 0$ in another component of $B_{r/2}^G(p_{i})$);
        \item if $S_i$ is $2$-sided, then $\int_{S_i} h\phi_i d\mathcal{H}^2 = 0$, where $\phi_i$ is the first eigenfunction of the Jacobi operator on $S_i$;
        \item if $S_i$ is $1$-sided, then $\int_{\tilde{S}_i} h\phi_i d\mathcal{H}^2 = 0$, where $\phi_i$ is the first eigenfunction of the Jacobi operator on $\tilde{S}_i$ and $\tilde{S}_i$ is the connected double cover of $S_i$.
    \end{enumerate}

    Choose $\epsilon_k \rightarrow 0$ so that the nontriviality condition (\ref{eq:width nontrivial1}) with $h$ replaced by $\epsilon_k h$ is met for sufficiently large $k$. By combining Theorem \ref{thm:pmc min-max theorem}, Theorem \ref{thm:convergence of pmc to minimal}, and arguing similarly as in \cite{wangzc2023existenceFour}, we obtain a closed embedded $G$-invariant minimal surface $\Gamma_{\infty} = \cup_{j = 1}^J \Gamma_j$, where each $\Gamma_j$ is $G$-invariant and belongs to $\mathcal{M}(C)$. Assuming $\Gamma_j = S_{i_j}$ for $i_j \in\{ 1, \cdots, \alpha \}$, the sign of $h$ must change on each $\Gamma_j$. If $\Gamma_j$ is $2$-sided with $m_{j} \geq 2$, it follows from Proposition \ref{Prop: supersolution} (nonnegative and nontriviality of $\varphi$), the positivity of $\phi_{i_j}$, and our choice of $h$ that the first eigenvalue satisfies $\lambda_{1}(\Gamma_j) \geq 0$. Hence, any connected component $\Gamma_j$ is stable if $\Gamma_j$ is $2$-sided with $m_j \geq 2$, which shows the first item. The same argument applied to the double cover of $1$-sided $G$-invariant $\Gamma' \subset \Gamma_\infty$ proves the second item. 

    By the choice of $h$, we have $h = 0$ outside $B_{r}^G(p_i)$. Since $\pi(B_{r}^G(p_i) \cap \Gamma_j)$ is a disk, it follows from Theorem \ref{thm:genus bound} that $\Gamma$ admits the genus bound. 
\end{proof}

\subsection{Multiplicity one for classical equivariant min-max}\label{Sec: multi one for classical}

In this part, we show an equivariant version of a multiplicity one type result in analogy to \cite[Theorem 7.3]{wangzc2023existenceFour}. This relies on a version of Simon-Smith min-max theory for un-oriented $G$-equivariant $\Gpm$-separating surfaces. 

Consider $(M, g_{_M})$ and $G$ as in Theorem \ref{Thm: multiplicity one}. Fix $\Sigma_{0}$ a $G$-connected closed surface of genus $\mathfrak{g}_{0}$.  We equip 
\[
\xEG(\Sigma_{0}) := \{\phi(\Sigma_{0})|\phi: \Sigma_0 \rightarrow M \text{ is a $G$-equivariant $\Gpm$-separating embedding}\} 
\]
with the un-oriented smooth topology. Let $X$ be a finite dimensional cubicial complex. Given  a fixed continuous map $\Phi_{0}: X \rightarrow \xEG$, we denote by $\Pi$ the collection of all continuous $\Phi: X \rightarrow \xEG$ which is homotopic to $\Phi_0$. Such a $\Phi$ is called a \emph{sweepout by $\Sigma_0$}, or simply a sweepout. Define 
\[
\mathbf{L}(\Pi) := \inf_{\Phi \in \Pi} \sup_{x \in X} \mathcal{H}^2(\Phi(x)). 
\]
In our setting, the space $\sEG$ defined in Section \ref{Sec: PMC min-max} double covers $\xEG$. Let $\bar{\lambda} \in H^{1}(\xEG; \mathbb{Z}_2)$ be the generator dual to the nontrivial element of $\pi_1(\xEG)$  corresponding to the projection $\mathbf{\pi}: \sEG \rightarrow \xEG$. Note a loop $\phi$ in $\xEG$ forms a sweepout in the sense of Almgren-Pitts if and only if $\bar{\lambda}[\phi] \neq 0$.

\begin{theorem}\label{Thm: multiplicity one classical}
    Consider the above setup and let $\Pi$ be a homotopy class of sweepouts by $\Sigma_0$ with 
    \[
        \mathbf{L}(\Pi) > 0. 
    \]
    Then there exists a closed embedded $G$-invariant minimal surface $\Gamma$ with $G$-connected components $\{\Gamma_{j}\}_{j=1}^J$ and integer multiplicities $\{m_j\}_{j=1}^J$ so that 
    \[
        \mathbf{L}(\Pi) = \mathcal{H}^2(\Gamma) = \sum_{j = 1}^{J} m_j \mathcal{H}^2(\Gamma_j)
    \]
    and 
	\begin{itemize}
		\item[(i)] if $\Gamma_j$ is unstable and $2$-sided, then $m_j=1$;
		\item[(ii)] if $\Gamma_j$ is $1$-sided, then its connected double cover is stable. 
	\end{itemize}
	Moreover the genus bound (\ref{eq:genus bound0}) holds for $\Gamma$ with $\mathfrak{g}_{0} = \mathfrak{g}(\Sigma_0)$.
\end{theorem}
\begin{proof}
    We follow the proof of \cite[Theorem 7.3]{wangzc2023existenceFour} (also see \cite[Theorem 5.2]{zhou2020multiplicity}) with some alterations.
    Similar to before, we only need to check the theorem for $G$-bumpy metric $g_{_M}$. The existence of a closed embedded $G$-invariant minimal surface $\Sigma$ satisfying $\mathcal{H}^2(\Sigma) = \mathbf{L}(\Pi)$ and Property \textbf{(R') }is guaranteed by Theorem \ref{thm:pull-tight}, Theorem \ref{thm:existence of almost minimizing pairs}, and Theorem \ref{thm:pmc min-max theorem}. 

    Let $\mathcal{S}$ be the collection of all stationary $2$-varifolds with mass lying in $[\mathbf{L}(\Pi) - 1, \mathbf{L}(\Pi) + 1]$, whose support is a closed embedded $G$-invariant minimal surface satisfying \textbf{(R')}.  Note that $\mathcal{S}$ is a finite set by bumpiness. Given a small $\bar{\epsilon} > 0$, set 
    \[
    Z_{i} = \{x \in X: \mathbf{F}(\Phi_{i}(x), \mathcal{S}) \geq \bar{\epsilon}\}, \text{ and } Y_i = \overline{X \setminus Z_i}.
    \]

    Since each $Y_i$ is topologically trivial, by adapting a continuous version of Pitts' combinatorial argument to $\{\Phi_i\}$, we can find another minimizing sequence, still denoted by  $\{\Phi_i\}$, such that $\mathbf{L}(\{\Phi_i|_{Z_i}\}) < \mathbf{L}(\Pi)$ for sufficiently large $i$. Lifting the maps $\Phi_i: Y_i \rightarrow \xEG$ to its double cover $\tilde{\Phi}_{i}: Y_i \rightarrow \sEG$, we have for $i$ large enough,
    \[
    \sup_{x \in \partial Y_i} \mathcal{H}^2(\tilde{\Phi}_i(x)) \leq \sup_{x \in Z_i} \mathcal{H}^2(\tilde{\Phi}_i(x)) < \mathbf{L}(\Pi).
    \]

    Denote $\tilde{\Pi}_{i}$ the $(Y_i, \partial Y_i)$-homotopy class associated with $\tilde{\Phi}_i|_{Y_i}$ in $\sEG$. By employing a contradiction argument as in \cite[Lemma 5.8]{zhou2020multiplicity}, we obtain 
    $\mathbf{L}(\tilde{\Pi}_i) \geq \mathbf{L}(\Pi) > \sup_{x \in \partial Y_i} \mathcal{H}^2(\tilde{\Phi}_i(x))$. 
    Then the proof is completed by applying Theorem \ref{Thm: multiplicity one} to $\tilde{\Pi}_i$ and letting $i \rightarrow \infty$. 
\end{proof}

\begin{proposition}\label{Prop: exist 1-sided component}
    In the above theorem, there exists a subset $\mathcal{I}\subset\{1,\dots,J\}$ so that $\Gamma'=\cup_{j\in\mathcal{I}}\Gamma_j$ is a $\Gpm$-separating $G$-surface. 
\end{proposition}

\begin{proof}
    By extracting a diagonal subsequence, one easily obtains a sequence $\Omega_k\in\CG(M)$ with $\bd\Omega_k\in\xEG$ so that $|\bd\Omega_k|\to V_\infty = \sum_{j=1}^J m_j|\Gamma_j|\in\V^G(M)$ as $k\to\infty$ in the varifolds sense, where $V_\infty$ is the min-max varifold. 
    It also follows from the compactness theorem that $\Omega_k$ converges (up to a subsequence) to some $\Omega_\infty\in\CG(M)$. 
    Hence, $\|\bd\Omega_\infty\|\leq \|V_{\infty}\|$ and $\bd\Omega_{\infty}$ is a $G$-invariant integral current supported in $\cup_{j=1}^J\Gamma_j$. 
    As an elementary fact, $\bd\Omega_\infty$ is also a $G$-invariant integral $n$-cycle in $ \cup_{j=1}^J\Gamma_j$ (cf. \cite[Appendix 8]{zhou2015min}), which implies $\bd\Omega_\infty=\sum_{j=1}^J k_j[[\Gamma_j]]$ for some $k_j\in \{0,1\}$ by the Constancy Theorem \cite[26.27]{simon1983lectures}. 
    Since $\Vol(\Omega_\infty)=\Vol(M)/2$, we see $\emptyset\neq \Gamma':=\cup\{ \Gamma_{j}: k_j =1\}$ is the smooth boundary of $\Omega_\infty\in\CG(M)$, which is $\Gpm$-separating.  
\end{proof}

\section{Minimal $\mathbb{RP}^2$ in $\mathbb{RP}^3$}\label{Sec: Minimal RP2}

In this section, let $M := S^3 \subset \mathbb{R}^4$ be the unit $3$-sphere, and $G:=\mZ_2$ acts on $M$ by the identity map $G_+=\{[0]\}$ and the antipodal map $G_-=\{[1]\}$. Consider the space
 \begin{align}
     \xEG:=\{\phi(S^2)|\phi:S^2\to S^3 \text{ a }G\text{-equivariant } G_{\pm}\text{-separating smooth embedding} \}
 \end{align}
endowed with the un-oriented smooth topology, where $G=\mathbb Z_2$ is given as above.

\subsection{Sweepouts formed by real projective planes}\label{Sec: sweepouts in RP^3}

We now describe three classes of sweepouts that detect three non-trivial cohomology classes in $H^k(\xEG;\mathbb Z_2)$, $k\in\{1,2,3\}$. 

To begin with, use $(a_1,a_2,a_3,a_4)$ and $[a_1,a_2,a_3,a_4]$ with $\sum_{i=1}^4 a_i^2=1$ to denote the points in $S^3$ and $\mathbb{RP}^3$ respectively. 
Define then 
\begin{align*}
    \widetilde{\mathcal G}((a_1,a_2,a_3,a_4))&:=\partial \left(\{a_1x_1+a_2x_2+a_3x_3+a_4x_4<0\}\cap S^3 \right),\\
    \mathcal G([a_1,a_2,a_3,a_4])&:=\{a_1x_1+a_2x_2+a_3x_3+a_4x_4=0\}\cap S^3,
\end{align*}
where $x_1,\dots, x_4$ are the coordinates functions in $\mathbb R^4$. 
Note $\widetilde{\mathcal G}(S^3)\cong S^3$ is the space of oriented great spheres and is a double cover of $\mathcal G(\RP^3)\cong \RP^3$ the space of un-oriented great spheres. 
Since $\mathcal G(\RP^3)\subset\xEG$, we can define our three maps:
\begin{align*}
    \Phi_1: \mathbb{RP}^1\to\xEG,\quad &[a_1,a_2]\to \mathcal{G}([a_1,a_2,0,0]);\\
    \Phi_2: \mathbb{RP}^2\to\xEG,\quad &[a_1,a_2,a_3]\to \mathcal{G}([a_1,a_2,a_3,0]);\\
    \Phi_3: \mathbb{RP}^3\to\xEG,\quad &[a_1,a_2,a_3,a_4]\to \mathcal{G}([a_1,a_2,a_3,a_4]).
\end{align*}
One notices that for any $k\in\{1,2,3\}$, $\Phi_k$ is a $k$-sweepout of $M=S^3$ in the sense of Almgren-Pitts. 
Indeed, consider the closed curve $\gamma:S^1=[0,2\pi]/\{0 \sim 2\pi\} \to \dmn(\Phi_k)= \RP^k$ given by $ \gamma(e^{i\theta})=[\cos(\theta/2),\sin(\theta/2),0,0]$, which is a generator of $\pi_1(\RP^k)$. 
Then the curve $\phi_k:=\Phi_k\circ\gamma$ in $\mathcal{G}(\RP^3)$ can be lifted to the curve $\tilde{\phi}_k:\theta\in [0,2\pi] \mapsto \widetilde{\mathcal{G}}((\cos(\theta/2),\sin(\theta/2),0,0))$ in the double cover $\widetilde{\mathcal{G}}(S^3)$ satisfying that $\tilde{\phi}_k(0)$ is $\tilde{\phi}_k(2\pi)$ with the opposite orientation. 
Hence, $\phi_k$ is a sweepout in the sense of Almgren-Pitts. 
Combining with the fact that the generator $\lambda\in H^1(\RP^k;\mZ_2) $ satisfies $\lambda(\gamma)=1$ and $\lambda^k\neq 0$, we conclude $\Phi_k$ is a $k$-sweepout in the sense of Almgren-Pitts (cf. \cite[Definition 4.1]{marques2017existence}).

Next, denote by $\iota: \xEG\to\Z_2(S^3;\mZ_2)$ the natural inclusion map into the space of mod-$2$ integral $2$-cycles, and by $\bar{\lambda}$ the generator of $H^1(\Z_2(S^3;\mZ_2);\mZ_2)$. 
Then the above result indicates $(\iota \circ \Phi_k)^*(\bar{\lambda}^k)\neq 0 \in H^k(\RP^k;\mZ_2)$ (cf. \cite[Definition 4.1]{marques2017existence}). 
In particular, we have  
\[\alpha:=\iota^*(\bar{\lambda})\in H^1(\xEG;\mZ_2)\]
satisfies $\alpha^k\neq 0\in H^k(\xEG;\mZ_2)$ for each $k\in \{1,2,3\}$. 

Finally, define $\mathscr{P}_k$, $k\in\{1,2,3\}$, to be the collection of continuous maps $\Phi$ from any cubical complex $X$ to $\xEG$ that detects $\alpha^k\in H^k(\xEG;\mZ_2)$, i.e. 
\[\mathscr{P}_k:=\left\{\Phi:X\to\xEG \left\vert \Phi^*(\alpha^k)\neq 0 \in H^k(X;\mZ_2) \right.\right\}. \]
Clearly, $\Phi_k\in\mathscr{P}_k$ for all $k=1,2,3$. 
Recall that 
\[\bL(\mathscr{P}_k):=\inf_{\Phi\in\mathscr{P}_k}\sup_{x\in\dmn(\Phi)} \mcH^2(\Phi(x)).\]

\subsection{Proof of Theorem \ref{Thm: minimal RP2 (main)}}
We have the following direct corollary by applying Theorem \ref{Thm: multiplicity one classical} and Proposition \ref{Prop: exist 1-sided component} to any homotopy class in $\mathscr{P}_k$. 
\begin{corollary}\label{Cor: exist min-max RP2}
    Suppose $\Pi\subset \mathscr{P}_k$, $k\in\{1,2,3\}$, is a homotopy class of sweepouts in $\xEG$. Then $\bL(\Pi)>0$ and the $G$-connected min-max minimal $G$-surfaces $\{\Gamma_j\}_{j=1}^J$ associated to $\Pi$ given by Theorem \ref{Thm: multiplicity one classical} satisfy that 
    \begin{itemize}
        \item[(i)] there is one $G$-component, say $\Gamma_1$, that is a $G_\pm$-separating minimal $2$-sphere, i.e. $\Gamma_1\in\xEG$; 
        \item[(ii)] every other $G$-component $\Gamma_j$, $j\geq 2$, is a $G$-invariant disjoint union of two minimal $2$-spheres. 
    \end{itemize}
\end{corollary}
\begin{proof}
    Firstly, $\bL(\Pi)>0$ since $\Pi\subset\mathscr{P}_k$ are $k$-sweepouts in the sense of Almgren-Pitts. 
    By Theorem \ref{Thm: multiplicity one classical}, $\{\Gamma_j\}_{j=1}^J$ satisfies the genus bound \eqref{eq:genus bound0}, which implies each $\Gamma_j$ is either a $G_\pm$-separating minimal sphere in $\xEG$ or a $G$-invariant disjoint union of two minimal spheres. 
    Combining Proposition \ref{Prop: exist 1-sided component} with the fact that every two elements in $\xEG$ have non-empty intersections, we know there is exactly one element of $\{\Gamma_j\}_{j=1}^J$ that is in $\xEG$. 
\end{proof}

Fix any Riemannian metric $g_{\mathbb{RP}^3}$ on
$\mathbb{RP}^3$. By \cite[Proposition 2.3]{bray2010area}, there exists an embedded area minimizing $\mathbb {RP}^2$ denoted as $\Sigma_0$, i.e. 
\[\mathcal{H}^2(\Sigma_0)=\inf\left\{\mathcal{H}^2(\Sigma):\Sigma \mbox{ is an embedded }\RP^2\subset\RP^3\right\}.\] 
Let $M=S^3$ with the pull-back metric $g_{_M}$, $G=\mZ_2$, and $\xEG$ be defined as above w.r.t. $(\RP^3,g_{_{\RP^3}})$. 
Then $\Sigma$ is an embedded minimal projective plane in $(\RP^3,g_{_{\RP^3}})$ if and only if $\pi^{-1}(\Sigma)\in\xEG$ is a minimal $2$-sphere in $(M,g_{_M})$. 

\begin{theorem}\label{Thm:L-S}
    Using the above notations, suppose $\xEG$ contains only finitely many elements that are minimal in $(M = S^3, g_{M})$. Then
    \begin{align*}
        0<2M_0<\mathbf{L}(\mathscr{P}_1)< \mathbf{L}(\mathscr{P}_2) < \mathbf{L}(\mathscr{P}_3), 
    \end{align*}
    where $M_0:=\mathcal{H}^2(\Gamma_0)$.
\end{theorem}
\begin{proof}
    The fact that $0<\mathbf{L}(\mathscr{P}_k)<\mathbf{L}(\mathscr{P}_{k+1})$  for $k=1,2$, follows from the standard Lusternik-Schnirelmann theory (see \cite[Lemma 8.3]{wangzc2023existenceFour}  or \cite[Theorem 6.1] {marques2017existence}). 
    Indeed, one can take $\mathcal{S}$ to be the collection of $G$-invariant integral varifolds with mass bounded by $\bL(\mathscr{P}_{k+1})$ whose support is an element in $\xEG$. 
    Then, using $\alpha\in H^1(\xEG;\mZ_2)$ and $\xEG$ in place of $\bar{\lambda}\in H^1(\Z_2(S^3;\mZ_2);\mZ_2)$ and $\Z_2(S^3;\mZ_2)$, the proof of \cite[Theorem 6.1]{marques2017existence} would carry over, leading to a contradiction to Corollary \ref{Cor: exist min-max RP2}(i). 
    Our goal is now to prove $2M_0<\mathbf{L}(\mathscr{P}_1)$. 
    
    Suppose by contradiction that $2M_0=\mathbf{L}(\mathscr{P}_1)$. 
    By the finiteness assumption, the union 
    \[
       \mathcal{S} := {\cup}\{\Gamma \in \xEG: \mathcal{H}^2(\Gamma) = 2M_{0}\} 
    \]
    is a closed set. 
    Fix a $\mZ_2$-invariant open ball $B \subset\subset S^3 \setminus B_{r}(\mathcal{S})$. 
    We have the following claim: 
    \begin{claim}\label{Claim: sweepout with small area}
        For any $\delta > 0$, there is $ \Phi_{\delta} \in \mathscr{P}_1$ so that $\sup_{x \in \dmn(\Phi_{\delta})}\mathcal{H}^{2}(\Phi_{\delta}(x) \cap B) < \delta$.     
    \end{claim}
    \begin{proof}[Proof of Claim \ref{Claim: sweepout with small area}]
        Suppose by contradiction that for some $ \delta_{0} > 0$, every $\Phi \in \mathscr{P}_{1}$ has $x \in \dmn(\Phi)$ with $\mathcal{H}^{2}(\Phi(x) \cap B) \geq \delta_{0}$. Let $\{\Phi_{i}\}$ be a minimizing sequence in $\mathscr{P}_{1}$, i.e.  
        \[
        \lim_{i \rightarrow \infty} \sup_{x \in \dmn(\Phi_i)} \mathcal{H}^{2}(\Phi_{i}(x)) =\bL(\mathscr{P}_1) = 2M_{0}. 
        \]
        For each $i$, we may pick some $x_i \in \dmn(\Phi_i)$ so that $\mathcal{H}^{2}(\Phi_{i}(x_i) \cap B) \geq \delta_{0}$ by assumptions. Based on the fact that
        \[
        2M_0 \leq \mathcal{H}^{2}(\Phi_{i}(x_{i})) \leq \sup_{x \in \dmn(\Phi_i)} \mathcal{H}^{2}(\Phi_{i}(x)) \rightarrow 2M_0 \text{ as } i \rightarrow \infty, 
        \]
        we have $\Phi_{i}(x_i)/\mZ_2\subset\RP^3$ converges (up to a subsequence) as varifolds to an area minimizing projective plane (cf. \cite[Proposition 2.3]{bray2010area}), and thus the varifolds limit $V_{\infty}$ of $\Phi_{i}(x_i)$ satisfies $\spt||V_{\infty}|| \subset \mathcal{S}$. This forces $\lim_{i \rightarrow \infty} \mathcal{H}^2(\Phi_i(x_i) \cap B) = 0$, which contradicts the choice of $x_i$.  
    \end{proof}

    Let $B_r(q)\subset B$ be a small ball with $0<r<r_0$ and $\delta\in (0,\alpha_0 r^2)$, where $\alpha_0,r_0>0$ are given by \cite[Proposition 8.2]{marques2017existence} with respect to $\Z_2(S^3;\mZ_2)$. Consider $\Phi_{\delta}$ in Claim \ref{Claim: sweepout with small area} with respect to this given $\delta$. 
    Then by definitions, $\iota\circ\Phi_\delta : \dmn(\Phi_\delta)\to\Z_2(S^3;\mZ_2)$ is a $1$-sweepout in the sense of Almgren-Pitts. 
    However, since $\M(\iota\circ\Phi_\delta(x)\llcorner B_r(q) )\leq \M(\iota\circ\Phi_\delta(x)\llcorner B )<\delta<\alpha_0r^2$ for all $x\in\dmn(\Phi_\delta)$, we obtain a contradiction to \cite[Proposition 8.2]{marques2017existence}. 
\end{proof}

At the end of this section, we prove Theorem \ref{Thm: minimal RP2 (main)}.

\begin{proof}[Proof of Theorem \ref{Thm: minimal RP2 (main)}]
    Firstly, we have a minimizing $\mathbb{RP}^2$ embedded in $(\mathbb{RP}^3,g_{_{\RP^3}})$ denoted as $\Sigma_0$ (cf. \cite[Proposition 2.3]{bray2010area}).

    {\bf Case I: $g_{_{\RP^3}}$ is a metric with positive Ricci curvature.}

    In this case, $g_{_M}$ is a $\mathbb Z_2$-invariant metric on $M=S^3$ with positive Ricci curvature. 
    Hence,  
    \begin{align}\label{Eq: no sphere in positive Ricci}
        \mbox{every $\mathbb Z_2$-connected minimal $2$-sphere in $(M,g_{_M})$ is $(\mathbb Z_2)_{\pm}$-separating,}
    \end{align}
    i.e. there is no minimal $2$-sphere in $(\RP^3,g_{_{\RP^3}})$. 
    Otherwise, there will be a pair of minimal $S^2\subset M$, denoted as $\Gamma=\{\Gamma_{+},\Gamma_{-}\}$, such that $[1]\cdot \Gamma_{\pm}=\Gamma_{\mp}$ and $\Gamma_+\cap\Gamma_-=\emptyset $, where $[1]\in\mZ_2$ acting on $M=S^3$ by the antipodal map. 
    This violates the embedded Frankel's property \cite{frankel1966fundamental}. 
    
    Next, without loss of generality, we assume 
    \begin{align}\label{Eq: finite RP2 assumption}
        \mbox{$(M=S^3,g_{_M})$ contains finitely many minimal $2$-spheres in $\mathscr{X}_{\mZ_{2,\pm}}$,} 
    \end{align}
    whose quotients in $\mathbb{RP}^3$ are minimal projective planes. 
    By Theorem \ref{Thm:L-S} and \eqref{Eq: finite RP2 assumption}, we have
    \begin{align}\label{Eq: strictly increasing width}
        0<2M_0<\mathbf{L}(\mathscr{P}_1)<\mathbf{L}(\mathscr{P}_2)<\mathbf{L}(\mathscr{P}_3).
    \end{align}
    Additionally, although each $\mathscr{P}_k$, $k\in\{1,2,3\}$, may contain many different homotopy classes, it follows from the finiteness assumptions \eqref{Eq: finite RP2 assumption} 
    and Corollary \ref{Cor: exist min-max RP2} that the min-max values of the homotopy classes in $\mathscr{P}_k$ must be stabilized to $\bL(\mathscr{P}_k)$. 
    Moreover, for each $k\in\{1,2,3\}$, $\mathbf{L}(\mathscr{P}_k)$ is realized by the area of an embedded connected minimal $2$-sphere $\Gamma_k\in\xEG$ with integer multiplicity $m_k\in\mZ_+$, i.e. $\bL(\mathscr{P}_k)=m_{k}\mcH^2(\Gamma_{k})$. 
    
    Since manifolds with positive Ricci contain no $2$-sided stable minimal hypersurface, we see $\Gamma_i$ must be unstable, and thus $m_i=1$ by Theorem \ref{Thm: multiplicity one classical}. 
    We can then conclude that $\Sigma_0,\pi(\Gamma_1),\pi(\Gamma_2),\pi(\Gamma_3)$ are distinct 
    embedded minimal real projective planes in $(\mathbb{RP}^3,g_{_{\RP^3}})$.

    {\bf Case II: $g_{_{\RP^3}}$ is a bumpy metric.}

    In this case, every embedded $\mZ_2$-invariant minimal surface in $(M=S^3,g_{_M})$ is non-degenerate. 
    Without loss of generality, we also assume that 
    \begin{align}\label{Eq: finite RP2+S2 assumption}
        \mbox{$(M=S^3,g_{_M})$ contains finitely many $\mZ_2$-invariant minimal $2$-spheres,} 
    \end{align}
    i.e. $(\mathbb{RP}^3,g_{_{\RP^3}})$ has finitely many embedded minimal $\mathbb{RP}^2$ and finitely many embedded minimal $S^2$. 
    In particular, \eqref{Eq: strictly increasing width} is still valid by Theorem \ref{Thm:L-S}. 

    For each $k\in\{1,2,3\}$, it follows from \eqref{Eq: finite RP2+S2 assumption} and Theorem \ref{Thm: multiplicity one classical} that the min-max values of homotopy classes in $\mathscr{P}_k$ must be stabilized to $\bL(\mathscr{P}_k)$, and there are disjoint embedded $\mZ_2$-invariant $\mZ_2$-connected minimal $2$-spheres $\{\Gamma_{k,j}\}_{j=1}^{J_k}$ and integer multiplicities $\{m_{k,j}\}_{j=1}^{J_k}$ so  that 
    \begin{align}\label{Eq: width realized by RP2}
        \bL(\mathscr{P}_k)=\sum_{j=1}^{J_k}m_{k,j}\mcH^2(\Gamma_{k,j}). 
    \end{align}
    Also, by Corollary \ref{Cor: exist min-max RP2}, we can assume $\Gamma_{k,1}\in\xEG$, and $\Gamma_{k,j}$ with $j\geq 2$ is a $\mZ_2$-invariant disjoint union of two minimal $2$-spheres.  
	By Theorem \ref{Thm: multiplicity one}, $\Gamma_{k,1}$ is either unstable with multiplicity one or stable. 

    {\bf Sub-case II-a: $\Gamma_{k,1}$ is unstable with multiplicity one for all $k\in\{1,2,3\}$.}
    
    Suppose Theorem \ref{Thm: minimal RP2 (main)}(i) fails, i.e. the number of distinct minimal embedded real projective planes in $(\RP^3,g_{_{\RP^3}})$ is strictly less than $4$. 
    Clearly, this assumption and \eqref{Eq: strictly increasing width}\eqref{Eq: width realized by RP2} indicate there must be some $k_0\in \{1,2,3\}$ with $J_{k_0}\geq 2$. 
    Thus, $\Gamma_{k_0,2}$ is a $\mZ_2$-invariant union of two minimal spheres $\Gamma_\pm$ with $[1]\cdot\Gamma_\pm=\Gamma_\mp$ and $\Gamma_\pm$ lies in the different components of $M\setminus\Gamma_{k_0,1}$. 
    
    If $\Gamma_+$ is stable, and thus strictly stable due to the bumpiness of the metric. 
    Then by \cite[Proposition 8.8]{wangzc2023existenceFour} (using Song's strategy \cite{song2018existence}), we can find another two minimal spheres in $(M,g_{_M})$ lying in a fundamental domain of $\mathbb{RP}^3$ (i.e. in a component of $M\setminus\Gamma_{k_0,1}$ containing $\Gamma_+$). 
    Together, we obtain three minimal $2$-spheres and one area minimizing real projective plane $\Sigma_0$ in $(\mathbb{RP}^3,g_{_{\RP^3}})$. 
    
    If $\Gamma_+$ is unstable. 
    Then, since $\Gamma_{k_0,1}$ is also assumed to be unstable, there exists an (isotopic area minimizing) stable minimal sphere $\widetilde{\Gamma}_+$ lying between $\Gamma_+$ and $\Gamma_{k_0,1}$. 
    Using \cite[Proposition 8.8]{wangzc2023existenceFour}, we obtain at least one more minimal sphere in a fundamental domain of $\mathbb{RP}^3$ (i.e. in the component of $M\setminus\widetilde{\Gamma}_+$ containing $\Gamma_+$) which is different from $\Gamma_+,\widetilde{\Gamma}_+$. 
    Therefore, we still have one area minimizing minimal real projective plane $\Sigma_0$ and three minimal spheres in $(\mathbb{RP}^3,g_{_{\RP^3}})$. 

    {\bf Sub-case II-b: $\Gamma_{k,1}$ is stable for some $k\in\{1,2,3\}$.}
    
    Since $\Gamma_{k,1}$ is stable (and thus strictly stable due to the bumpiness) for some $k\in\{1,2,3\}$, we can apply \cite[Proposition 8.8]{wangzc2023existenceFour} again to find another two distinct minimal spheres in the fundamental domain of $(\mathbb{RP}^3,g_{_{\RP^3}})$ (i.e. in a component of $M\setminus\Gamma_{k,1}$). 
    In conclusion, we can find one area minimizing real projective plane $\Sigma_0$ and two distinct minimal spheres in $(\mathbb{RP}^3,g_{_{\RP^3}})$.
\end{proof}

\section{Minimal Klein bottles in $L(4m, 2m\pm 1)$}\label{Sec: Minimal K2}

Consider the unit $3$-sphere $\mS^3=\{(z,w)\in\mathbb{C}^2:|z|+|w|=1\}$. 
For any two coprime integers $p\geq 1$ and $q\in [1,p)$, we have the cyclic $\mZ_p$-action on $\mS^3$ generated by $\xi_{p,q}$:
\begin{align}\label{Eq: lens space action}
    [1]\cdot (z,w) =\xi_{p,q}(z,w) := \left(e^{2\pi i/p}\cdot z, e^{2\pi q i/p}\cdot w \right).
\end{align}
Then the lens space $L(p,q)$ is defined to be $\mS^3/\mZ_p$. 

By \cite[Corollary 6.4]{bredon1969nonorientable}, the Klein bottle embeds into $L(4m,2m\pm 1)$ only. 
Hence, in this section, we always denote by
\begin{align}\label{Eq: pull-back lens space}
    M:=S^3,\quad G:=\mZ_p=\langle \xi_{p,q}\rangle, \quad G_+=\mZ_{p/2}:= \langle \xi_{p,q}^2\rangle,\quad G_-:=G\setminus G_+,
\end{align}
where $p=4m$, $q=2m\pm 1$, and $m\geq 1$. 
In addition, consider the space
\begin{align}
    \mathscr{X}_{G_{\pm}}=\left\{ \phi(T^2) \left|\phi: T^2\to S^3 \text{ a $G$-equivariant $G_{\pm}$-separating smooth embedding} \right.\right\},
\end{align}
endowed with the smooth un-oriented topology. 
One can verify that for any $G$-surface $\Sigma\subset M$, $\Sigma/G$ is a Klein bottle in $L(4m,2m\pm 1)$ if and only if $\Sigma\in \xEG$.

\subsection{Sweepouts formed by Klein bottles}\label{Sec: sweepouts by K^2}

We now describe the sweepouts that detect the non-trivial cohomology classes of $H^k(\xEG;\mZ_2)$, where $k=3$ for $m=1$, and $k=1$ for $m>1$. 
Firstly, we mention the following results given by Ketover \cite[Proposition 4.2]{ketover2022flipping}.

\begin{proposition}[\cite{ketover2022flipping}]
    Let $L(p,q)$ denote the lens space endowed with the round metric, $p\geq 2$. Then $L(p,q)$ admits an embedded Klein bottle if and only if $p=4m$ and $q=2m\pm 1$ for $m\geq 1$. If $m>1$ then $L(4m,2m\pm1)$ 
    admits an $S^1$-family of minimal Klein bottles. If $m=1$, then $L(4,1)$ admits an $S^1\times \mathbb{RP}^2$-family of minimal Klein bottles.
\end{proposition}

For the sake of completeness, we provide a relatively detailed explanation following the constructions in \cite[Section 4]{ketover2022flipping} (see also \cite{tucker2013geodesic}). 
Firstly, one can identify $\mS^3$ and $\mS^2$ with the group of unit quaternions and pure unit quaternions (without real part) respectively, i.e. 
\begin{align*}
    \mS^3 :=\{a+bi+cj+dk:|a|^2+|b|^2+|c|^2+|d|^2=1\}, \quad
    \mS^2 :=\{bi+cj+dk\in \mS^3\}.
\end{align*}
Note any orientation preserving isometry $f\in {\rm Isom}_+(\mS^3)=SO(4)$ can be represented by $f(q)=q_1qq_2^{-1}$ for some $q_1,q_2\in \mS^3$. 
Hence, an oriented $2$-plane in $\R^4$ spanned by two orthonormal vectors $u,v\in \mS^3$, can be written as $\langle u,v\rangle = (q_1,q_2)\cdot \langle 1, i\rangle:= \langle q_1q_2^{-1}, q_1iq_2^{-1}\rangle$ for some $q_1,q_2\in \mS^3$. 
Denote by $\tilde{G}_2(\mathbb R^4)$ the Grassmannian manifold consisting of all oriented $2$-dimensional subspaces of $\R^4$. 
Then the space of oriented geodesics in round $\mS^3$ is homeomorphic to $\tilde{G}_2(\mathbb R^4)\cong \mS^2\times \mS^2$, where the homeomorphism is specified by the map 
\[P:\tilde{G}_2(\mathbb R^4)\to \mS^2\times \mS^2, \qquad P((q_1,q_2)\cdot \langle 1, i\rangle) = (q_1iq_1^{-1}, q_2iq_2^{-1}).\] 
Note $(a,b),(-a,-b)\in\mS^2\times \mS^2$ correspond to the same geodesic with opposite orientations.

Next, for $a\in \mS^2,B\subset \mS^2$, denote 
\begin{align}
    \tau (a,B):=\{x\in \mS^3:x\in P^{-1}(a,b)\cap \mS^3, b\in B \}.
\end{align}
It's known that $\tau(a, B)$ is a Clifford torus if $a \in \mS^2$ and $B$ is a great circle of $\mS^2$ (cf. \cite[(4.16)]{ketover2022flipping}). 
Given $b\in \mS^2$, denote by 
\begin{align}
    E(b):=\bd \left(\left\{x\in \mS^2: b_1x_1+b_2x_2+b_3x_3<0\right\}\right)\subset \mS^2
\end{align}
an oriented equator. 
Since $\tau(a,E(b))=\tau(-a,E(-b))$ is $\tau(a,E(-b))=\tau(-a,E(b))$ with the opposite orientation, we see the space of unoriented Clifford tori is homeomorphic to $\mathbb{RP}^2 \times \mathbb{RP}^2$. 
In addition, by \cite[Section 5.1]{tucker2013geodesic}, we know $\xi_{p,q}$ can be represented by $\xi_{p,q}(x)=e^{i\pi (q+1)/p}\cdot x\cdot e^{-i\pi(q-1)/p}$ using quaternions. 
Moreover, given $a\in\mS^2$ and a great circle $E\subset\mS^2$, 
\begin{align}\label{Eq: lens action on torus}
    \xi_{p,q}(\tau(a,E)) = \tau(\hat{\eta}^1_{p,q}(a), \hat{\eta}^2_{p,q}(E) ),
\end{align}
where $\hat{\eta}^1_{p,q},\hat{\eta}^2_{p,q}$ is the rotation (in the $\{j,k\}$-plane) on $\mS^2$ by the angle of $2\pi(q+1)/p$ and $2\pi(q-1)/p$ respectively with fixed points $\{\pm i\}$. 

{\bf Case A: $p=4$ and $q=1$ (or $q=3$).}

In this case, one verifies that $\xi_{4,1}$ is orientation preserving, and $\xi_{4,1}(\tau \left(\cos(\theta)j+\sin(\theta)k, E(b) \right) )= \tau \left(-\cos(\theta)j - \sin(\theta)k, E(b) \right) $ is $\tau \left(\cos(\theta)j+\sin(\theta)k, E(b) \right)$ with the opposite orientation for any $\theta\in [0,2\pi], b\in\mS^2$, which implies the support of $ \tau \left(\cos(\theta)j+\sin(\theta)k, E(b) \right)$ is an element in $\xEG$. 
Hence, we have a family of Klein bottles in $L(4,1)\cong L(4,3)$ parameterized by $\mathcal{G}: \RP^1\times\RP^2 \to \xEG$, 
\begin{align}
    \mathcal{G}([a_1,a_2], [b_1,b_2,b_3]) := \spt\left(\tau\left(a_1j+a_2k, E(b_1i+b_2j+b_3k) \right) \right).
\end{align}
We can now define three maps into $\xEG$:
\begin{align*}
    \Phi_1:\mathbb{RP}^1\to \mathscr{X}_{G_{\pm}},&\quad [b_1,b_2]\mapsto \mathcal{G}([1,0],[b_1,b_2,0]); \\ 
    \Phi_2:\mathbb {RP}^2\to \mathscr{X}_{G_{\pm}}, &\quad  [b_1,b_2,b_3]\mapsto \mathcal{G}([1,0],[b_1,b_2,b_3]); \\
    \Phi_3:\mathbb{RP}^1\times \mathbb{RP}^2\to \mathscr{X}_{G_{\pm}},&\quad ([a_1,a_2],[b_1,b_2,b_3])\mapsto \mathcal{G}([a_1,a_2],[b_1,b_2,b_3]).
\end{align*}
One notices that $\Phi_k$, $k\in\{1,2,3\}$, is a $k$-sweepout in the sense of Almgren-Pitts (cf. \cite[Definition 4.1]{marques2017existence}). 
Indeed, denote by $\alpha_l$ the generator of $H^1(\RP^l;\mZ_2)$, $l\in\{1,2\}$, by $\bar{\lambda}$ the generator of $H^1(\Z_2(M;\mZ_2);\mZ_2)$, and by $\iota: \xEG\to\Z_2(M;\mZ_2)$ the natural inclusion. 
Then the closed curves $\gamma_1(e^{i\theta})=([\cos(\theta/2),\sin(\theta/2)],[1,0,0])$ and $\gamma_2(e^{i\theta})=([1,0],[\cos(\theta/2),\sin(\theta/2),0])$ in $\RP^1\times \RP^2$ satisfy that 
\begin{itemize}
    \item $\iota\circ\mathcal{G}\circ\gamma_1:S^1\to\Z_2(M;\mZ_2)$ is a sweepout in the sense of Almgren-Pitts, since $\tau(-1j,E(1j))$ is $\tau(1j,E(1j))$ with the opposite orientation;
    \item $\iota\circ\mathcal{G}\circ\gamma_2:S^1\to\Z_2(M;\mZ_2)$ is a sweepout in the sense of Almgren-Pitts, since $\tau(1j,E(-1j))$ is $\tau(1j,E(1j))$ with the opposite orientation;
    \item $\alpha_1\oplus \alpha_2(\gamma_1)=\alpha_1\oplus \alpha_2(\gamma_2)=1,$ $\alpha_1\oplus \alpha_2(\gamma_1+\gamma_2)=0$,
\end{itemize}
where $0\neq \alpha_1\oplus \alpha_2\in H^1(\RP^1\times \RP^2;\mZ_2)$. 
Combining the second bullet with the fact that $\alpha_2^2\neq 0 \in H^2(\RP^2;\mZ_2)$, we conclude $\Phi_1$ and $\Phi_2$ are $1$-sweepout and $2$-sweepout in the sense of Almgren-Pitts respectively. 
Moreover, we also have $\mathcal{G}^*\iota^*(\bar{\lambda}) = \alpha_1\oplus \alpha_2\in H^1(\RP^1\times \RP^2;\mZ_2)$ by the above three bullets. 
Together with $(\alpha_1\oplus \alpha_2)^3\neq 0 \in H^3(\RP^1\times \RP^2;\mZ_2)$, we conclude that $\Phi_3$ is a $3$-sweepout in the sense of Almgren-Pitts. 

Furthermore, since $0\neq (\alpha_1\oplus \alpha_2)^3 = \mathcal{G}^*\iota^*(\bar{\lambda}^3)\in H^3(\RP^1\times \RP^2;\mZ_2)$, we know  
\[\alpha:=\iota^*(\bar{\lambda})\in H^1(\xEG;\mZ_2)\]
satisfies $\alpha^k\neq 0\in H^k(\xEG;\mZ_2)$ for every $k\in \{1,2,3\}$. 
Thus, we can define $\mathscr{P}_k$, $k\in \{1,2,3\}$, to be the collection of continuous maps $\Phi$ from any cubical complex $X$ to $\xEG$ that detects $\alpha^k\in H^k(\xEG;\mZ_2)$, i.e. 
\begin{align}\label{Eq: sweepouts of K2 in L41}
    \mathscr{P}_k:=\left\{\Phi:X\to \xEG \left\vert \Phi^*(\alpha^k)\neq 0 \in H^k(X;\mZ_2)\right.\right\}.
\end{align}
Clearly, the above $\Phi_k\in\mathscr{P}_k$ for all $k=1,2,3$. 

{\bf Case B: $p=4m$ and $q=2m\pm 1$ with $m\geq 2$.}

In this case, one verifies that $\xi_{p,q}(\tau(\cos(\theta)j+\sin(\theta)k, E(1j)))$ is $\tau(\cos(\theta)j+\sin(\theta)k, E(1j))$ with the opposite orientation for any $\theta\in [0,2\pi]$.  
Hence, we have a family of Klein bottles in $L(4m,2m\pm1)$ ($m\geq 2$) parameterized by $\mathcal{G}:\RP^1\to \xEG$, 
\begin{align}
    \mathcal{G}([a_1,a_2]):=\spt\left(\tau\left(a_1j+a_2k, E(1j)\right)\right).
\end{align}
We can define the map 
\begin{align*}
    \Phi_1=\mathcal{G}: \mathbb{RP}^1\to \mathscr{X}_{G_{\pm}},
\end{align*}
which is a sweepout in the sense of Almgren-Pitts since $\tau(-1j, E(1j))$ is $\tau(1j, E(1j))$ with the opposite orientation. 
Hence, $\mathcal{G}^*\iota^*(\bar{\lambda})\neq 0 \in H^1(\RP^1;\mZ_2)$, 
\[\alpha:=\iota^*(\bar{\lambda})\neq 0\in H^1(\xEG;\mZ_2),\]
and we can similarly define 
\begin{align}\label{Eq: sweepouts of K2 in L4m}
    \mathscr{P}_1:=\left\{\Phi:X\to \xEG \left\vert \Phi^*(\alpha)\neq 0 \in H^1(X;\mZ_2) \right.\right\}
\end{align}
with $\Phi_1\in\mathscr P_1$.

\subsection{Proof of Theorem \ref{Thm: minimal Klein bottles (main)}} 

Using Theorem \ref{Thm: multiplicity one classical} and Proposition \ref{Prop: exist 1-sided component}, we have the following corollary, which gives the existence of min-max and minimizing minimal Klein bottles. 

\begin{corollary}\label{Cor: exist min-max and minimizing K2}
    Suppose $\mathscr{P}_k$ is given by \eqref{Eq: sweepouts of K2 in L41} or \eqref{Eq: sweepouts of K2 in L4m}, and $\Pi\subset \mathscr{P}_k$ is a homotopy class of sweepouts in $\xEG$. 
    Then $\bL(\Pi)>0$, and the $G$-connected min-max minimal $G$-surfaces $\{\Gamma_j\}_{j=1}^J$ with integer multiplicities $\{m_j\}_{j=1}^J$ given by Theorem \ref{Thm: multiplicity one classical} satisfy that
    \begin{itemize}
        \item[(i)] there is exactly one $G$-component, say $\Gamma_1$, that is a $G_\pm$-separating minimal torus with multiplicity one, i.e. $\Gamma_1\in\xEG$ and $m_1=1$; 
        \item[(ii)] any other $G$-component $\Gamma_j$, $j\geq 2$, is a $G$-invariant disjoint union of $\#G=4m$ numbers of minimal $2$-spheres in $M$. 
    \end{itemize}
    
    Moreover, support $\{\Sigma_k\}_{k\in\N}\subset\xEG$ is an area minimizing sequence in $\xEG$ satisfying 
    \[\lim_{k\to\infty}\mcH^2(\Sigma_k)=\inf\{\mcH^2(\Gamma):\Gamma\in\xEG\},\] then up to a subsequence, $\Sigma_k$ converges as varifolds to an area minimizing minimal torus $\Sigma_0\in\xEG$ with multiplicity one. 
\end{corollary}
\begin{proof}
    Firstly, $\bL(\Pi)>0$ since $\Pi\subset\mathscr{P}_k$ are $k$-sweepouts in the sense of Almgren-Pitts. 
    By Theorem \ref{Thm: multiplicity one classical}, $\{\Gamma_j\}_{j=1}^J$ satisfies the genus bound \eqref{eq:genus bound0}, which implies each $\Gamma_j$ is either 
    \begin{itemize}
        \item[(a)] a $G_\pm$-separating minimal torus in $\xEG$ with multiplicity one, i.e. $\pi(\Gamma_j)$ is a minimal Klein bottle in $L(4m,2m\pm 1)$; or
        \item[(b)] a connected $G$-invariant minimal torus not in $\xEG$ with multiplicity one, i.e. $\pi(\Gamma_j)$ is a minimal torus in $L(4m,2m\pm 1)$; or
        \item[(c)] a $G$-invariant disjoint union of minimal spheres, i.e. $\pi(\Gamma_j)$ is a minimal sphere in $L(4m,2m\pm 1)$. (Note that there is no embedded $\RP^2$ in $L(4m, 2m\pm 1)$, see Geiges-Thies \cite{geiges2022klein}.)
    \end{itemize} 
    Combining Proposition \ref{Prop: exist 1-sided component} with the fact that every two elements in $\xEG$ have non-empty intersections, we know there is exactly one element of $\{\Gamma_j\}_{j=1}^J$ that is in $\xEG$, which gives (i). 
    Then by \eqref{eq:genus bound0}, every other $G$-component $\Gamma_j$, $j\geq 2$, is in case (c). 
    Note that there is no embedded $\RP^2$ in $L(4m,2m\pm 1)$, and the only non-trivial finite effective free action on $S^2$ is $\mZ_2$ with the quotient homeomorphic to $\RP^2$. 
    Hence, $\Gamma_j$, $j\geq 2$, has $\#G=4m$ components with $G$ permuting them. 

    Suppose $\{\Sigma_k\}_{k\in\N}\subset\xEG$ is an area minimizing sequence in $\xEG$. 
    It then follows from the $G$-invariance of $\Sigma_k$ and \cite[Theorem 1]{meeks1982embeddedMS} that $\Sigma_k$ converges (up to a subsequence) as varifolds to a disjoint union of embedded $G$-connected minimal $G$-surfaces $\{\Gamma_j\}_{j=1}^J$ with integer multiplicities $\{m_j\}_{j=1}^J$ so that the genus bound \eqref{eq:genus bound0} is also valid. 
    Next, combining the proof of Proposition \ref{Prop: exist 1-sided component} and the above arguments, we know there is exactly one $G$-component, say $\Gamma_1$, satisfying $\Gamma_1\in \xEG$ and $m_j=1$. 
    Finally, we notice $J=1$ since $\inf\{\mcH^2(\Gamma):\Gamma\in\xEG\}\leq\mcH^2(\Gamma_1)\leq \lim_{k\to\infty}\mcH^2(\Sigma_k)$. 
\end{proof}

Next, we can use the Lusternik-Schnirelmann theory to show the following result. 

\begin{theorem}\label{Thm: L-S in Lens space}
    Given any Riemannian metric $g_{_L}$ on $L(4m,2m\pm 1)$ with $m\geq 1$, let $M,G,G_\pm$ be given by \eqref{Eq: pull-back lens space}, and $g_{_M}$ be the $G$-invariant pull-back metric on $M$. 
    Suppose $\xEG$ contains only finitely many elements that are minimal in $(M,g_{_M})$. 
    Then 
    \[0<4M_0<\bL(\mathscr{P}_1)<\bL(\mathscr{P}_2)<\bL(\mathscr{P}_3) ~\mbox{if $m=1$},\quad {\rm and}\quad 0<4m M_0<\bL(\mathscr{P}_1)~\mbox{if $m\geq 2$},\]
    where $4mM_0=\inf\{\mcH^2(\Gamma):\Gamma\in\xEG\}$, and $\mathscr{P}_k$ is defined in \eqref{Eq: sweepouts of K2 in L41}\eqref{Eq: sweepouts of K2 in L4m}. 
\end{theorem}
\begin{proof}
    Using Corollary \ref{Cor: exist min-max and minimizing K2} in place of Corollary \ref{Cor: exist min-max RP2} and \cite[Proposition 2.3]{bray2010area}, the proof in Theorem \ref{Thm:L-S} can be taken almost verbatim to show the desired results. 
\end{proof}

At the end of this subsection, we prove Theorem \ref{Thm: minimal Klein bottles (main)}.

\begin{proof}[Proof of Theorem \ref{Thm: minimal Klein bottles (main)}] 
    Firstly, given any Riemannian metric $g_{_{L}}$ on $L(4m,2m\pm 1)$, we have an embedded area minimizing Klein bottle denoted as $\Sigma_0$ by the second part of Corollary \ref{Cor: exist min-max and minimizing K2}. 
    Let $M=S^3$, $G=\mZ_{4m}$, and $G_\pm\subset G$ be defined as in \eqref{Eq: pull-back lens space}. 
    Endow $M$ with the pull-back metric $g_{_M}$ so that $(M,g_{_M})$ is locally isometric to $(L(4m,2m\pm 1),g_{_L})$. 

    {\bf Case I: $g_{_L}$ is a metric with positive Ricci curvature.}

    In this case, $g_{_M}$ is a $G$-invariant Riemannian metric on $M=S^3$ with positive Ricci curvature. 
    Hence, similar to \eqref{Eq: no sphere in positive Ricci}, the embedded Frankel's property indicates there is no embedded $G$-invariant union of minimal $2$-spheres in $(M,g_{_M})$, 
    i.e. there is no minimal $2$-sphere in $(L(4m,2m\pm 1),g_{_L})$. 
    
    Without loss of generality, we assume that $(L(4m,2m\pm 1),g_{_L})$ contains only finitely many distinct minimal embedded Klein bottles, i.e. 
    \begin{align}\label{Eq: finite K2 assumption}
        \mbox{$(M,g_{_M})$ contains finitely many minimal tori in $\xEG$.}
    \end{align}

    {\bf Sub-case I.A: $m=1$.} 
    Combining \eqref{Eq: finite K2 assumption} with Theorem \ref{Thm: L-S in Lens space}, $\{\mathscr{P}_k\}_{k=1}^3$ in \eqref{Eq: sweepouts of K2 in L41} satisfies
    \begin{align}\label{Eq: increasing width in Lens}
        0<4\mcH^2(\Sigma_0)<\mathbf{L}(\mathscr{P}_1)<\mathbf{L}(\mathscr{P}_2)<\mathbf{L}(\mathscr{P}_3).
    \end{align}
    For each $k\in\{1,2,3\}$, although $\mathscr{P}_k$ may contain many different homotopy classes, we see from \eqref{Eq: finite K2 assumption} 
    and Corollary \ref{Cor: exist min-max and minimizing K2}(i)(ii) that the min-max values of the homotopy classes in $\mathscr{P}_k$ must be stabilized to $\bL(\mathscr{P}_k)$, and moreover, $\bL(\mathscr{P}_k)$ is realized by the area of an embedded $G$-connected minimal torus $\Gamma_k\in\xEG$ with multiplicity one, i.e. $\bL(\mathscr{P}_k)=\mcH^2(\Gamma_k)$. 

    Therefore, we have found $4$ distinct embedded minimal Klein bottles in $(L(4,1),g_{_L})$.

    {\bf Sub-case I.B: $m\geq 2$.}
    The proof in {\bf Sub-case I.A} would carry over for $\Sigma_0$ and $\mathscr{P}_1$ in \eqref{Eq: sweepouts of K2 in L4m} to find $2$ distinct embedded minimal Klein bottles in $(L(4m,2m\pm 1), g_{_L})$.

    {\bf Case II: $g_{_L}$ is a bumpy metric.}

    In this case, every embedded $G$-invariant minimal surface in $(M,g_{_M})$ is non-degenerate. 
    Without loss of generality, we also assume that $(M,g_{_M})$ contains finitely many $G$-invariant minimal tori in $\xEG$ and finitely many $G$-invariant union of minimal $2$-spheres.

    {\bf Sub-case II.A: $m=1$.} 
    Combining the constructions of $\{\mathscr{P}_k\}_{k=1}^3$ in \eqref{Eq: sweepouts of K2 in L41} and the above finiteness assumptions with Theorem \ref{Thm: L-S in Lens space}, we know \eqref{Eq: increasing width in Lens} is still valid here. 
    In addition, for each $k\in\{1,2,3\}$, using the arguments in {\bf Sub-case I.A} and Corollary \ref{Cor: exist min-max and minimizing K2}(i)(ii), there are disjoint embedded $G$-connected minimal $G$-surfaces $\{\Gamma_{k,j}\}_{j=1}^{J_k}$ and integer multiplicities $\{m_{k,j}\}_{j=1}^{J_k}$ so that 
    \begin{align}\label{Eq: min-max surface in lens}
        \bL(\mathscr{P}_k)=\sum_{j=1}^{J_k}m_{k,j}\mcH^2(\Gamma_{k,j}),  
    \end{align} 
    where 
    $\Gamma_{k,1}\in\xEG$ has multiplicity $m_{k,1}=1$, and $\Gamma_{k,j}$ with $j\in\{2,\dots,J_k\}$ is a $G$-invariant disjoint union of $4m$ numbers of minimal $2$-spheres. 

    Suppose Theorem \ref{Thm: minimal Klein bottles (main)}(i) fails, i.e. the number of distinct embedded minimal Klein bottles in $(L(4,1),g_{_L})$ is strictly less than $4$. 
    Clearly, this assumption and \eqref{Eq: increasing width in Lens} indicate the existence of $k_0\in\{1,2,3\}$ with $J_{k_0}\geq 2$. 
    Hence, we have a minimal Klein bottle $K=\pi(\Gamma_{k_0,1})$ and a minimal $2$-sphere $S=\pi(\Gamma_{k_0,2})$ in $(L(4,1),g_{_L})$ with $K\cap S=\emptyset$. 
    
    By cutting $L(4,1)$ along $K$, we obtain a compact manifold $\widetilde{L}$ with a connected boundary $\bd\widetilde{L}$ diffeomorphic to the oriented double cover $\widetilde{K}$ of $K$. 
    Then, by isotopic area minimizing $S$ within the region in $\widetilde{L}$ enclosed between $\bd\widetilde{L}=\widetilde{K}$ and $S$, we have a stable (and thus strictly stable) minimal $2$-sphere $S_0\subset\widetilde{L}\setminus\bd\widetilde{L}$ so that $\widetilde{L}\setminus S_0$ has a component $N$ with $N\cap\bd\widetilde{L}=\emptyset$.  
    After applying \cite[Proposition 8.8]{wangzc2023existenceFour} (using Song's strategy \cite{song2018existence}) to $N$, we can find either
    \begin{itemize}
        \item at least one more minimal sphere in $N$ that is different from $S$, if $S\neq S_0$; or
        \item at least two more minimal spheres in $N$, if $S=S_0$.
    \end{itemize}
    Thus, we obtain three minimal $2$-spheres and an area minimizing Klein bottle in $(L(4,1),g_{_L})$. 
    
    {\bf Sub-case II.B: $m\geq 2$.}
    The proof in {\bf Sub-case II.A} would carry over for $\Sigma_0$ and $\mathscr{P}_1$ in \eqref{Eq: sweepouts of K2 in L4m} to find either two distinct embedded minimal Klein bottles or one area minimizing Klein bottle together with three embedded minimal spheres in $(L(4m,2m\pm 1), g_{_L})$. 
\end{proof}

\section{Minimal torus in lens spaces}\label{Sec: minimal T2}

Note Theorem \ref{Thm: minimal tori (main)}(ii) has been proven by Ketover in \cite[Theorem 4.8]{ketover2022flipping}. Hence, in this section, we only consider minimal tori in $L(2,1)=\RP^3$ and $L(p,q)$ with $q\notin \{1,p-1\}$. 
In particular, we always denote by 
\begin{align}\label{Eq: notaions for G-invariant T2}
    M:=S^3,\quad G:=\mZ_p=\langle \xi_{p,q}\rangle, \quad G_+=\mZ_{p/2}:= \langle \xi_{p,q}^2\rangle,\quad G_-:=G\setminus G_+,
\end{align}
where $\xi_{p,q}$ is given by \eqref{Eq: lens space action}, and either $(p,q)=(2,1)$ or $q\notin\{1,p-1\}$. 
In addition, we consider the following spaces 
\begin{itemize}
    \item $\msX :=\left\{ \phi(T^2) \left|\phi: T^2\to M/G \text{ a separating embedding} \right.\right\}$,
    \item $\msY :=\left\{ \pi(S^1) \left| S^1 \text{ is a $G$-invariant equator (great circle) in $S^3$} \right. \right\}$,
    \item $\bmsX := \msX\cup\msY$,
\end{itemize}
which are endowed with the smooth un-oriented topology. 
Note that $\phi(T^2)$ is separating if $(M/G)\setminus\phi(T^2)$ has two connected components $U_1,U_2$. 
Moreover, denote by $\iota:\bmsX\to\Z_2(M/G;\mZ_2)$ the natural inclusion, and by $\bar{\lambda}$ the generator of $H^1(\Z_2(M/G;\mZ_2);\mZ_2)$.

\subsection{Sweepouts formed by tori}\label{Sec: sweepouts by T^2}

To begin with, we have a family of $T^2$ in $\msX$ by Ketover's construction \cite[Proposition 4.2]{ketover2022flipping}. 

\begin{proposition}[Ketover \cite{ketover2022flipping}]\label{prop: parametrization of T^2}
    Let $L(p,q)$ denote the lens space endowed with the round metric, $p\geq 2$. Then we have
    \begin{enumerate}
        \item $\mathbb{RP}^3$ admits a family of Clifford torus parameterized by $\mathbb{RP}^2\times \mathbb{RP}^2$; 
        \item $L(p,q)$ with $q\notin \{1,p-1\}$ admits exactly one Clifford torus.
    \end{enumerate}
\end{proposition}

For the sake of completeness, we adopt the notations as in Section \ref{Sec: Minimal K2} and explain the constructions of sweepouts formed by tori in $\RP^3$ or $L(p,q)$ with $q\notin\{1,p-1\}$.

Firstly, for any $a,b\in\mS^2$, recall that $E(b)=\bd \{x\in\mS^2: x\cdot b <0\}$ is an oriented equator and $\tau(a,E(b))$ is an oriented Clifford torus in $\mS^3$. 
Then, for any $t\in [-1,1]$, we define 
\begin{align}\label{Eq: parallel circle}
    E_t(b):= \bd \left\{x\in\mS^2: x\cdot b := x_1b_1+x_2b_2+x_3b_3 < t \right\}
\end{align}
to be an oriented circle parallel to $E(b)$ with $E_1(b)=\{b\}$ and $E_{-1}(b)=\{-b\}$. 
It's known that $\tau(a,E_t(b))$ is an oriented flat torus in $\mS^3$ parallel to the Clifford torus $\tau(a,E_0(b))$ (cf. \cite[Page 22]{ketover2022flipping}), which degenerates to an oriented great circle in $\mS^3$ at $t=\pm1$.

{\bf Case A: $(p,q)=(2,1)$ and $L(p,q)=\RP^3$.} 

Using the notations in Section \ref{Sec: Minimal K2} and \eqref{Eq: lens action on torus}, we see $\xi_{2,1}$ is orientation preserving and the oriented (possibly degenerated) torus $\xi_{2,1}\left(\tau\left( a,E_t(b) \right)\right) $ in $\mS^3$ is the same as $ \tau\left( a,E_t(b) \right) $ for all $a,b\in\mS^2, t\in[-1,1]$, which implies the quotient of $\spt(\tau\left( a,E_t(b) \right))$ is an element in $\bmsX$. 
Therefore, we have a continuous family of oriented tori in $\RP^3$ parameterized by $\widetilde{\mathcal{G}}:[-1,1]\times \mS^2\times\mS^2 \to \bmsX$ (with oriented smooth topology) 
\begin{align}
    \widetilde{\mathcal{G}}(t,a,b)=  (\tau\left(a,E_t(b) \right) )/G, \qquad \forall (t,a,b)\in [-1,1]\times\mS^2\times\mS^2.    
\end{align}
Note that the oriented torus $\widetilde{\mathcal{G}}(t,a,b)=\widetilde{\mathcal{G}}(t,-a,-b)$ is the oriented torus $\widetilde{\mathcal{G}}(-t,-a,b)=\widetilde{\mathcal{G}}(-t,a,-b)$ with the opposite orientation for $t\in (-1,1)$. 
Hence, $\widetilde{\mathcal{G}}$ induces a continuous map into $(\bmsX,\msY)$ with un-oriented smooth topology:
\begin{align}\label{Eq: parameter T2 in RP3}
    \mathcal{G}: (\bmcX,\mcY) \to (\bmsX,\msY), \qquad \mathcal{G}([t,a,b])= \spt\left( \widetilde{\mathcal{G}}(t,a,b)  \right),
\end{align}
where $[t,a,b]\in \bmcX$, 
\begin{align}\label{Eq: parameter space of T2 in RP3}
    \bmcX:= \left([-1,1]\times \mS^2\times\mS^2\right) / \mZ_2\oplus\mZ_2, \qquad \mcY:= \bd \bmcX, \qquad \mcX:=\bmcX\setminus \mcY,
\end{align}
and $\mZ_2\oplus\mZ_2$ acts on $[-1,1]\times \mS^2\times\mS^2$ by 
$(0,0) = id$, 
\begin{align}\label{Eq: action on parameter space}
    (1,0)\cdot(t,a,b)=(-t,-a,b), ~~~(0,1)\cdot (t,a,b)=(-t,a,-b), ~~~(1,1)\cdot(t,a,b)=(t,-a,-b).
\end{align}
In addition, one notices that the parameter space $\bmcX$ retracts onto
\begin{align}
    \mcX_4:= \{ [0,a,b]\in \mcX: a,b\in\mS^2 \} \cong \RP^2\times \RP^2 
\end{align}
the parameter space of Clifford tori. 
Then, by Poincar\'{e}–Lefschetz Duality, 
\begin{align}\label{Eq: poincare dual of parameter space}
    H^1(\bmcX,\mcY;\mZ_2)\cong H_4(\bmcX;\mZ_2) \cong H_4(\mcX_4;\mZ_2) \cong \mZ_2
\end{align}
has a generator $\lambda\in H^1(\bmcX,\mcY;\mZ_2)$. 
Moreover, for any fixed $a,b\in\mS^2$, the curve $\gamma(t):=[t,a,b]\in\bmcX$ satisfies that $\iota\circ \mathcal{G}\circ\gamma : ([-1,1],\{\pm 1\})\to  (\Z_2(M;\mZ_2), \{0\})$ is a sweepout in the sense of Almgren-Pitts, which implies $\mathcal{G}^*\iota^*(\bar{\lambda}) = \lambda\in H^1(\bmcX,\mcY;\mZ_2).$

\begin{lemma}\label{Lem: lambda4 not 0}
    $\lambda^4=\mathcal{G}^*\iota^*(\bar{\lambda}^4) \neq 0 \in H^4(\bmcX,\mcY;\mZ_2)$ and $\lambda^5=0$. 
\end{lemma}

The proof of Lemma \ref{Lem: lambda4 not 0} is left in Appendix \ref{Appendix: lambda4}.  
Now, we know $\iota\circ\mathcal{G}:(\bmcX,\mcY) \to (\Z_2(M/G;\mZ_2), \{0\})$ is a $4$-sweepout in the sense of Almgren-Pitts as $\lambda^4= (\iota\circ\mathcal{G})^*(\bar{\lambda}^4)\neq 0$. 
In particular, 
\[\alpha = \iota^*(\bar{\lambda})\in H^1(\bmsX,\msY;\mZ_2)\]
satisfies $\alpha^4\neq 0 \in H^4(\bmsX,\msY;\mZ_2)$. 
Hence, we can define the collection of continuous maps
\begin{align}\label{Eq: sweepouts of T2 in RP3}
    \mathscr{P}_k:=\left\{\Phi: (X,Z) \to (\bmsX ,\msY) \left\vert \Phi^*(\alpha^k)\neq 0 \in H^k(X,Z;\mZ_2) \right.\right\},
\end{align}
where $k\in\{1,2,3,4\}$ and $Z\subset X$ are any two cubical complexes.

{\bf Case B: $q\notin\{1,p-1\}$.} 

In this case, we fix $a_0=b_0=(1,0,0)\in\mS^3$. 
Similarly, by \eqref{Eq: lens action on torus}, the oriented (possibly degenerated) torus $\xi_{p,q}\left(\tau\left( a_0,E_t(b_0) \right)\right) $ in $\mS^3$ is the same as $ \tau\left( a_0,E_t(b_0) \right) $ for all $t\in[-1,1]$, which implies the quotient of $\spt(\tau\left( a_0,E_t(b_0) \right))$ is an element in $\bmsX$. 
Therefore, we have a continuous family of un-oriented tori in $\RP^3$ parameterized by 
\begin{align}\label{Eq: parameter T2 in Lpq}
    \mathcal{G}:([-1,1],\{\pm 1\})\to (\bmsX,\msY),\qquad\mathcal{G} (t) :=  \spt\left(\tau\left(a_0,E_t(b_0) \right) /G\right), ~\forall t\in [-1,1].    
\end{align}
Note $\iota\circ \mathcal{G}$ is a sweepout in the sense of Almgren-Pitts. 
Hence, 
\[\alpha:=\iota^*(\bar{\lambda})\neq 0 \in H^1(\bmsX,\msY;\mZ_2),\]
and we can define the collection of continuous maps 
\begin{align}\label{Eq: sweepouts of T2 in Lpq}
    \mathscr{P}_1:=\left\{\Phi: (X,Z) \to (\bmsX,\msY) \left\vert \Phi^*(\alpha)\neq 0 \in H^1(X,Z;\mZ_2) \right.\right\},
\end{align}
where $Z\subset X$ are any two cubical complexes.

\subsection{Proof of Theorem \ref{Thm: minimal tori (main)}}

Since $T^2$ is orientable, we can apply \cite[Theorem B]{wangzc2023existenceFour} to obtain the following corollary. 

\begin{corollary}\label{Cor: exist min-max T2}
    Let $g_{_{L}}$ be a Riemannian metric on $L(p,q)$ with positive Ricci curvature, where $(p,q)=(2,1)$ or $q\notin\{1,p-1\}$. 
    Suppose $\mathscr{P}_k$ is given by \eqref{Eq: sweepouts of T2 in RP3} or \eqref{Eq: sweepouts of T2 in Lpq}, and $\Pi\subset \mathscr{P}_k$ is a homotopy class of sweepouts in $\bmsX$. 
    Then,  
    \[ \bL(\Pi) = \mcH^2(\Gamma)>0\]
    for a smooth connected embedded minimal torus $\Gamma$ in $(L(p,q),g_{_{L}})$. 
\end{corollary}
\begin{proof}
    Firstly, we know $\bL(\Pi)>0$ since $\Pi$ is a family of sweepouts in the sense of Almgren-Pitts. 
    Then, combining \cite[Theorem B]{wangzc2023existenceFour} with the embedded Frankel's property, there exists a connected smooth embedded minimal surfaces $\Gamma$ in $(L(p,q),g_{_L})$ so that $\bL(\Pi)=m\mcH^2(\Gamma)$ for some integer $m\in\mZ_+$, and (by the genus bound) $\Gamma$ is either 
    \begin{itemize}
        \item[(a)] a minimal torus with multiplicity one, or
        \item[(b)] a minimal $2$-sphere, or 
        \item[(c)] a minimal real projective plane with a stable connected double cover (only if $L(p,q)\cong\RP^3$),
        \item[(d)] a minimal Klein bottle with a stable connected double cover (only if $(p,q)=(4m,2m\pm 1)$).
    \end{itemize}
    Note the pull-back metric $g_{_M}=\pi^*g_{_L}$ on $M=S^3$ also have positive Ricci curvature. 
    Hence, (c)(d) can not happen. 
    Additionally, if $\Gamma$ is in case (b), then $\pi^{-1}(\Gamma)$ is a $\#G=p$ numbers union of minimal spheres in $S^3$, since $2$-spheres only have trivial finite covers. 
    This contradicts the embedded Frankel's property in $(M,g_{_M})$. 
    Therefore, $\Gamma$ has to be in case (a). 
\end{proof}

Recall that $\bL(\mathscr{P}_k):=\inf_{\Phi\in\mathscr{P}_k}\sup_{x\in\dmn(\Phi)}\mcH^2(\Phi(x))$. 

\begin{theorem}\label{Thm: L-S in Lens space for T2}
    Let $g_{_L}$ be a Riemannian metric on $L(p,q)$ with positive Ricci curvature, where $(p,q)=(2,1)$ or $q\notin\{1,p-1\}$. 
    Suppose $\msX$ contains only finitely many elements that are minimal in $(L(p,q),g_{_L})$. 
    Then 
    \[0<\bL(\mathscr{P}_1)<\bL(\mathscr{P}_2)<\bL(\mathscr{P}_3)<\bL(\mathscr{P}_4) \]
    when $(p,q)=(2,1)$, and $0<\bL(\mathscr{P}_1)$ when $q\notin\{1,p-1\}$. 
\end{theorem}

\begin{proof}
    This result follows from the Lusternik-Schnirelmann theory. One can easily use Corollary \ref{Cor: exist min-max T2} and \cite[Lemma 5.3]{haslhofer2019minimalS2} to extend the proof of \cite[Theorem 6.1]{marques2017existence} and \cite[Theorem 5.2]{haslhofer2019minimalS2}. 
\end{proof}

We mention that for a general Riemannian metric $g_{_L}$ on $L(p,q)$, the above result is still valid if $(L(p,q),g_{_L})$ contains only finitely many minimal $2$-spheres, minimal real projective planes, minimal Klein bottles, and minimal tori. 

Finally, we prove Theorem \ref{Thm: minimal tori (main)}.

\begin{proof}[Proof of Theorem \ref{Thm: minimal tori (main)}] 
By combining Corollary \ref{Cor: exist min-max T2} and Theorem \ref{Thm: L-S in Lens space for T2}, we find four embedded minimal tori in $L(2, 1) = \mathbb{RP}^3$ and one embedded minimal torus in $L(p, q)$ with $q \neq \{1, p - 1\}$. It remains to show the second part of (i), namely to find one more embedded minimal tori in $\mathbb{RP}^3$ under the additional assumption of bumpiness.    

To begin with, we recall the setup of White's Manifold Structure Theorem \cite{white1991space}. Let $M, N$ be compact smooth Riemannian manifolds such that $\dim N < \dim M$. Given $\Gamma$ an open subset of $C^q$ Riemannian metrics on $M$, consider the space $\mathcal{M}$ of ordered pairs $(\gamma, [N])$, where $\gamma \in \Gamma$ and $[N]$ denotes the diffeomorphism classes of minimal embeddings of a Riemannian manifold $N$ into $M$ with respect to $\gamma$. It was shown that $\mathcal{M}$ is a separable $C^2$ Banach manifold and the projection map 
\begin{align*}
    \Pi: \mathcal{M} \rightarrow \Gamma, \qquad (\gamma, [N]) \mapsto \gamma, 
\end{align*}
is smooth Fredhom with Fredhom index $0$. 

Now, let $N$ be an unoriented torus, $M = \mathbb{RP}^3$, and $\Gamma$ be the set of all $C^4$ Riemannian metrics on $\mathbb{RP}^3$ with positive Ricci curvature. Since $\Gamma$ is connected \cite{hamilton82} and $\Pi: \mathcal{M} \rightarrow \Gamma$ is proper \cite{schoen85}, we can calculate the mapping degree $d$ of $\Pi$ using Theorem 5.1 of \cite{white1991space}. By Proposition \ref{prop: parametrization of T^2} (1), the set of all Clifford torus in $\mathbb{RP}^3$ is diffeomorphic to $\mathbb{RP}^2 \times \mathbb{RP}^2$. Denote by $\gamma_{0}$ the round metric on $\mathbb{RP}^3$. As a direct corollary of Lawson's Conjecture \cite{brendle13lawson, brendle13survey}, every embedded minimal torus in $(\mathbb{RP}^3, \gamma_0)$ is congruent to the Clifford torus. Hence $\Sigma = \Pi^{-1}(\gamma_0)$ is diffeomorphic to $\mathbb{RP}^2 \times \mathbb{RP}^2$ with $\chi(\Sigma) = 1$. Together with the fact that every Clifford tori in $(\mathbb{RP}^3, \gamma_0)$ has nullity $4$ and Morse index $1$, we see
\[
d = (-1)^{\Index(\Sigma)}\chi(\Sigma) = -1.   
\]
For generic metrics $\gamma$ of positive Ricci curvature on $\mathbb{RP}^3$, we further have 
\[
\sum_{S \in \Pi^{-1}(\gamma)} (-1)^{\Index(S)} = d = -1,  
\]
which implies that the number of elements in $\Pi^{-1}(\gamma)$ is odd. Given that we already find four embedded minimal tori in $\mathbb{RP}^3$, there must be at least a fifth one. This completes the proof.  
\end{proof}

\appendix

\section{Proof of Lemma \ref{Lem: lambda4 not 0}}\label{Appendix: lambda4}

\begin{proof}[Proof of Lemma \ref{Lem: lambda4 not 0}]
    Using the notations in Section \ref{Sec: sweepouts by T^2} {\bf Case A}, for any $k$-dimensional closed submanifold $N$ in $\bmcX$, we denote by $[N]\in H_k(\bmcX;\mZ_2)$ the images of the fundamental classes of $N$ under the inclusion into $\bmcX$, and denote by $[N]^*\in H^{5-k}(\bmcX,\mcY;\mZ_2)$ its Poincar\'{e}–Lefschetz dual. 
    It then follows from the Poincar\'{e}–Lefschetz duality \eqref{Eq: poincare dual of parameter space} that the generator $\lambda\in H^1(\bmcX,\mcY;\mZ_2)$ coincides with $[\mcX_4]^*$, where $\mcX_4 :=\{[0,a,b]\in \bmcX\}\cong \RP^2\times\RP^2$ is the parameter space of Clifford tori. 
    We now use the transversal intersection product to visualize the cup product $\lambda^k$, $k=2,3,4,5$. 

    Firstly, after regarding $a=(a_1,a_2,a_3),b=(b_1,b_2,b_3)\in\mS^2$ as vectors in $\R^3$, we can define 
    \begin{align*}
        \mcX_4' & := \{ [a\cdot b, a, b]\in\bmcX : a,b\in\mS^2 \}, \\
        \mcX_3 & :=\mcX_4\cap\mcX_4' = \{[0,a,b]\in \bmcX: a,b\in\mS^2 {\rm ~with~}a\cdot b=0 \},
    \end{align*}
    where $a\cdot b = \sum_{i=1}^3a_ib_i$ is the inner product in $\R^3$. 
    Since $\{(a\cdot b, a,b):a,b\in\mS^2\}$ is a $4$-submanifold of $[-1,1]\times\mS^2\times\mS^2$ invariant under the $(\mZ_2\oplus\mZ_2)$-action in \eqref{Eq: action on parameter space}, we know $\mcX_4'$ is a closed $4$-submanifold of $\bmcX$ and is homotopic to $\mcX_4$ through $(s, [t,a,b])\mapsto [st,a,b] $, $s\in [0,1]$. 
    In addition, for any $[0,a,b]\in \mcX_3 =\mcX_4\cap\mcX_4'$, we have $a\perp b$, and thus there exists a curve $\sigma\subset \mS^2$ with $\sigma(0)=a$ and $\sigma'(0)=b$. 
    Hence, $\tilde{\sigma}(s)=[\sigma(s)\cdot b, \sigma(s), b]$ is a curve in $\mcX_4'$ so that $\tilde{\sigma}'(0)$ has a component $(\sigma \cdot b)'(0)\cdot \frac{\bd}{\bd t}=\frac{\bd}{\bd t}$, which implies $\mcX_4'$ is transversal to $\mcX_4$ along their intersection $\mcX_3$. 
    Therefore, by \cite[3.6]{goresky1981whitneyChain} (see also \cite[VI, Theorem 11.9]{bredon2013topology}), we have $\lambda^2 = [\mcX_4]^*\smile[\mcX_4']^* = [\mcX_4\cap \mcX_4']^* =[\mcX_3]^*\in H^2(\bmcX,\mcY;\mZ_2)$. 

    Next, denote by $e_1:=(1,0,0)$, $e_2:=(0,1,0)$, and $e_3:=(0,0,1)$. 
    For any $a,b\in\mS^2$ with $a\cdot b = 0$, we can take the cross product vector $v=(v_1,v_2,v_3):=a\times b$ in $\R^3$, i.e. 
    \[v_1(a,b):=a_1b_3-a_3b_2, \quad v_2(a,b):=a_3b_1-a_1b_3, \quad v_3(a,b)=a_1b_2-a_2b_1.\]
    Then $\{a,b,v\}$ forms an orthonormal basis of $\R^3$. 
    Define then 
    \begin{align*}
        \mcX_3'&:=\{[v_1(a,b), a, b]\in\bmcX: a,b\in \mS^2 {\rm ~with~} a\cdot b =0\},\\
        \mcX_2&:= \mcX_4\cap \mcX_3'=\{[0, a, b]\in\bmcX: a,b\in \mS^2 {\rm ~with~} a\cdot b =0, v_1(a,b)=0\}.
    \end{align*}
    Since $\{(v_1(a,b),a,b):a,b\in\mS^2,a\cdot b=0\}$ is a closed $3$-submanifold of $[-1,1]\times\mS^2\times\mS^2$ invariant under the $(\mZ_2\oplus\mZ_2)$-action in \eqref{Eq: action on parameter space}, it's clear that $\mcX_3'$ is a closed $3$-submanifold of $\bmcX$ and is homotopic to $\mcX_3$. 
    Additionally, for any $a,b\in\mS^2$ with $a\cdot b =0$ and $ v_1(a,b)=0$, we can rotate $a,b$, and $v=a\times b$ together around the $(v\times e_1)$-axis to obtain curves $a(\theta),b(\theta),v(\theta)=a(\theta)\times b(\theta) \subset \mS^2$ so that $a(0)=a,b(0)=b,v(0)=v$ and $v_1(\theta)=v_1(a(\theta),b(\theta))=\sin(\theta)$. 
    Hence, $\tilde{\sigma}(\theta)=[v_1(\theta), a(\theta), b(\theta)]$ is a curve in $\mcX_3'$ so that $\tilde{\sigma}'(0)$ has a component $v_1'(0)\cdot \frac{\bd}{\bd t}=\frac{\bd}{\bd t}$, which implies $\mcX_3'$ is transversal to $\mcX_4$ along their intersection $\mcX_2$. 
    Therefore, we have $\lambda^3=[\mcX_4]^*\smile[\mcX_3']^* = [\mcX_4\cap \mcX_3']^* =[\mcX_2]^*\in H^3(\bmcX,\mcY;\mZ_2)$. 

    Moreover, define 
    \begin{align*}
        \mcX_2'&:=\{[v_2(a,b), a, b]\in\bmcX: a,b\in \mS^2 {\rm ~with~} a\cdot b =0, v_1(a,b)=0 \},\\
        \mcX_1&:= \mcX_4\cap \mcX_2'=\{[0, a, b]\in\bmcX: a,b\in \mS^2 {\rm ~with~} a\cdot b =0, v_1(a,b)=v_2(a,b)=0\}. 
    \end{align*}
    Similarly, one notices that $\mcX_2'$ is a closed $2$-submanifold of $\bmcX$ and is homotopic to $\mcX_2$. 
    Given any $a,b\in\mS^2$ with $a\cdot b=0$ and $v_1(a,b)=v_2(a,b)=0$, one can use the rotation around the $e_1$-axis ($\pm e_1=v\times e_2$) to show that $\mcX_2'$ is transversal to $\mcX_4$ along their intersection $\mcX_1$.  
    Hence, $\lambda^4 = [\mcX_4]^*\smile[\mcX_2']^* = [\mcX_4\cap \mcX_2']^*=[\mcX_1]^*\in H^4(\bmcX,\mcY;\mZ_2)$. 

    Finally, we notice that $\mcX_1$ is a non-trivial closed curve in $\mcX_4\subset \mcX$ given by 
    \[\theta\mapsto [0, (\cos(\theta/2),\sin(\theta/2),0), (-\sin(\theta/2),\cos(\theta/2),0)], \qquad \forall \theta\in [0,2\pi].\]
    Hence, $[\mcX_1]\neq 0 \in H_1(\mcX_4;\mZ_2)\cong H_1(\bmcX;\mZ_2)$, and thus $\lambda^4=[\mcX_1]^*\neq 0$ in $H^4(\bmcX,\mcY;\mZ_2)$. 
    Moreover, one verifies that $\mcX_1':=\{[a_1  b_1, a,b]: [0,a,b]\in\mcX_1\}$ is homotopic to $\mcX_1$ and transversal to $\mcX_4$, and $\mcX_1'\cap\mcX_4$ has two points, which implies $\lambda^5=[\mcX_1'\cap\mcX_4]^*=0\in H^5(\bmcX,\mcY;\mZ_2)$. 
\end{proof}

\bibliographystyle{abbrv}

\end{document}